\documentclass[reqno, english]{amsart}
\RequirePackage{mathrsfs} \let\mathcal\mathscr
\usepackage[hyperindex,breaklinks,colorlinks=true,backref=page,linktoc=page]{hyperref}
\usepackage{a4wide}
\usepackage{enumerate}
\usepackage{color}
\usepackage{amssymb,amsmath,amsfonts,amsthm}

\DeclareMathOperator{\Gal}{Gal}
\DeclareMathOperator{\spl}{spl}

\newtheorem{theorem}{Theorem}
\newtheorem{lemma}[theorem]{Lemma}
\newtheorem{proposition}[theorem]{Proposition}
\newtheorem{corollary}[theorem]{Corollary}
\newtheorem*{theoremA1}{Theorem A.1}
\theoremstyle{definition}
\newtheorem{definition}[theorem]{Definition}
\newtheorem{remark}[theorem]{Remark}
\newtheorem{example}[theorem]{Example}

\numberwithin{theorem}{section}
\numberwithin{equation}{section}
\numberwithin{table}{section}

\newcommand\EE{\mathbb{E}}
\newcommand\FF{\mathbb{F}}
\newcommand\ZZ{\mathbb{Z}}
\newcommand\NN{\mathbb{N}}
\newcommand\QQ{\mathbb{Q}}
\newcommand\RR{\mathbb{R}}
\newcommand\CC{\mathbb{C}}
\newcommand{\OO}{\mathcal{O}}
\newcommand{\vt}{{\mathbf{t}}}
\newcommand{\vx}{\mathbf{x}}
\newcommand{\vy}{\mathbf{y}}

\newcommand{\vc}{\mathbf{c}}
\newcommand{\ideals}{\mathcal{I}_K}
\newcommand{\idealsL}{\mathcal{I}_L}
\newcommand{\id}[1]{\mathfrak{#1}}
\newcommand{\ppp}{\id{p}}
\newcommand{\qqq}{\id{q}}
\newcommand{\PPP}{\id{P}}
\newcommand{\aaa}{\id{a}}
\newcommand{\bbb}{\id{b}}
\newcommand{\ccc}{\id{c}}
\newcommand{\ddd}{\id{d}}

\newcommand{\w}{\id{w}}
\newcommand{\norm}{\mathfrak{N}}
\newcommand{\where}{\ :\ }
\newcommand{\vecnorm}[1]{\lVert #1 \rVert}

\DeclareMathOperator{\vol}{vol}

\newcommand{\abs}[1]{\left|#1\right|}
\newcommand{\vb}{\mathbf{b}}
\newcommand{\va}{\mathbf{a}}

\newcommand{\vK}{\mathbf{K}}
\newcommand{\vC}{\mathbf{C}}
\newcommand{\mmm}{\mathfrak{m}}

\newcommand\one{\mathbf{1}}
\newcommand\cP{\mathcal{P}}

\newcommand{\cA}{\mathcal{A}}
\newcommand{\cD}{\mathcal{D}}
\newcommand{\cX}{\mathcal{X}}
\newcommand{\vw}{\boldsymbol{\omega}}
\newcommand{\vh}{\mathbf{h}}
\newcommand{\cF}{\mathcal{F}}
\newcommand{\cC}{\mathcal{C}}
\newcommand{\fS}{\mathfrak{S}}

\DeclareMathOperator\lcm{lcm}
\DeclareMathOperator\HRH{HRH}
\DeclareMathOperator\ident{id}
\DeclareMathOperator\Li{Li}
\DeclareMathOperator\Disc{Disc}
\DeclareMathOperator\Frob{Frob}
\DeclareMathOperator\poly{poly}

\DeclareMathOperator\Lip{Lip}
\DeclareMathOperator\Sn{S}

\newcommand\ab{{\operatorname{ab}}}
\DeclareMathOperator\Cinf{C^\infty}
\newcommand{\RmZ}{{\RR/\ZZ}}
\DeclareMathOperator\mods{mod*}

\DeclareMathOperator\PGL{PGL}
\newcommand\Ia{{\operatorname{Ia}}}
\newcommand\Ib{{\operatorname{Ib}}}
\newcommand\II{{\operatorname{II}}}

\definecolor{dblackcolor}{rgb}{0.0,0.0,0.0}
\definecolor{dbluecolor}{rgb}{0.01,0.02,0.7}
\definecolor{dgreencolor}{rgb}{0.2,0.4,0.0}
\definecolor{dgraycolor}{rgb}{0.30,0.3,0.30}

\makeatletter
\@namedef{subjclassname@2020}{\textup{2020} Mathematics Subject Classification}
\makeatother

\begin{document}

\title[
Linear constellations in primes with arithmetic restrictions
]
{
Linear constellations in primes with arithmetic restrictions
}

\author{Christopher Frei}
\address{
Technische Universit\"at Graz\\
Institut f\"ur Analysis und Zahlentheorie\\
Kopernikusgasse 24/II\\
A-8010 Graz\\
Austria
}
\email{frei@math.tugraz.at}

\author{Magdal\'ena Tinkov\'a\\\ \\ With an appendix by\\ Christopher Frei, Joachim K\"onig and Magdal\'ena Tinkov\'a}
\address{Faculty of Information Technology, Czech Technical University in Prague, Th\'akurova 9, 160 00 Praha 6, Czech Republic}
\email{tinkova.magdalena@gmail.com}

\address{Department of Mathematics Education, Korea National University of Education, Cheongju 28173, South Korea}
\email{jkoenig@knue.ac.kr}

\begin{abstract}
  We prove analogues of the theorem of Green and Tao on linear
  constellations in primes, in which the primes under consideration
  are restricted by certain arithmetic conditions. Our first main
  result is conditional upon Hooley's Riemann hypothesis and imposes
  the extra condition that the primes have prescribed primitive
  roots. Our second main result is unconditional and imposes the extra
  condition that the primes have prescribed Artin symbols in given
  Galois number fields. In the appendix we present an application of
  the second result in inverse Galois theory.
\end{abstract}


\maketitle

\setcounter{tocdepth}{1}
\tableofcontents

\markright{\sc{Linear constellations in primes with arithmetic restrictions}}
\markleft{\sc{Christopher Frei and Magdal\'ena Tinkov\'a}}

\section{Introduction}

In this paper, we are interested in linear constellations in primes that satisfy certain arithmetic restrictions.

\subsection{Primes with prescribed primitive roots}\label{subsec:preli_artin_primes}
For our first main result, we require the primes to have prescribed primitive roots. For example, are there three-term arithmetic progressions in the primes whose common difference is one less than a prime, and such that all involved primes have $5$ as a primitive root? In other words, we are looking for integers $n_1,n_2$ such that all of the affine-linear forms
\begin{equation}\label{eq:comp_2_example}
  n_1,n_2,n_1+n_2-1,n_1+2n_2-2
\end{equation}
simultaneously take values in the set of primes with primitive root $5$. This is a constellation of complexity $2$ in the sense of Green and Tao. Their celebrated result \cite{MR2680398}, together with \cite{MR2877066,MR2877065} and the work \cite{MR2950773} of Green, Tao and Ziegler, gives asymptotics for all constellations of finite complexity in unrestricted primes, implying in particular the existence of infinitely many solutions $(n_1,n_2)$ to our question if one ignores the primitive root condition.

However, if $5$ is a primitive root for $p>2$, then it is in particular a quadratic non-residue modulo $p$, so by quadratic reciprocity $p\equiv 2,3\bmod 5$. As the simultaneous congruences
\begin{equation*}
  n_1,n_2,n_1+n_2-1,n_1+2n_2-2 \equiv 2,3 \bmod 5
\end{equation*}
have no solution, we conclude that there are no $(n_1,n_2)\in\ZZ$ for which all of \eqref{eq:comp_2_example} are primes
with primitive root $5$. If $5$ is replaced by, say, $7$, then there are no congruence obstructions and indeed a quick search yields the solution $(n_1,n_2)=(41,67)$. But are there infinitely many solutions?

Currently, the existence of infinitely many primes with primitive root $a$ is not known unconditionally for any value of $a$. \emph{Artin's conjecture} predicts an asymptotic formula, which Hooley \cite{MR207630} has famously proved assuming the following version of the Riemann hypothesis.

\begin{definition}[Hooley's Riemann hypothesis]
  For $a\in\ZZ\smallsetminus\{0\}$, we let $\HRH(a)$ denote the following
  proposition: for all squarefree $k\in\NN$, the Dedekind zeta function of the
  number field $\QQ(\mu_k,\sqrt[k]{a})$ satisfies the Riemann
  hypothesis.
\end{definition}

Here, $\mu_k$ denotes the group of $k$-th roots of unity in $\CC$.
For background and history surrounding Artin's conjecture and Hooley's result,
we recommend Moree's survey \cite{MR3011564}.
Our first main result establishes, conditionally on $\HRH(7)$, the existence of infinitely many solutions $(n_1,n_2)$ as above. In fact, we allow arbitrary primitive roots, arbitrary finite-complexity systems of affine-linear forms, and obtain an asymptotic formula analogous to the theorem of Green and Tao for unrestricted primes.

In order to set up and motivate the precise statement of our result,
we need to introduce some notation. Let $a\in\ZZ$ such that $a\neq -1$ and $a$ is
not a perfect square, let $q\in\NN$ and $b\in\ZZ$. Conditionally upon
$\HRH(a)$ and building upon Hooley's work, Lenstra \cite{Lenstra1977}
has shown an asymptotic formula for the number of primes up to $N$ and
congruent to $b\bmod q$, for which $a$ is a primitive root:
\begin{equation}\label{eq:artin_prime_asymptotic}
   \#\left\{p\leq N\where p\equiv b\bmod q,\ \FF_p^\times=\langle a\rangle\right\} =
   \delta(a,b,q)\frac{N}{\log N}+o_{a,q}\left(\frac{N}{\log N}\right).
\end{equation}
The leading constant $\delta(a,b,q)$ in this formula is defined as follows. For 
$q\in\NN$ and squarefree $k\in\NN$, we consider the number fields
\begin{align}
  F(q,k,a) &:= \QQ(\mu_q,\mu_k,\sqrt[k]{a})\quad\text{ and }\quad G(k,a):= F(1,k,a) = \QQ(\mu_k,\sqrt[k]{a}).\label{eq:def_Gka}
\end{align}
Fix a generator $\zeta_q:=e^{2\pi i/q}$ of $\mu_q$. For any $b\in \ZZ$ with $\gcd(b,q)=1$, we let $\sigma_b\in\Gal(\QQ(\mu_q)/\QQ)$ denote the automorphism with $\sigma_b(\zeta_q)=\zeta_q^b$. With
\begin{equation}\label{eq:def_eta}
\eta(a,b,k,q) :=
           \begin{cases}
             \frac{1}{[F(q,k,a):\QQ]},&
             \begin{aligned}
               &\text{ if $\gcd(b,q)=1$ and $\sigma_b$ fixes $\QQ(\mu_q)\cap G(k,a)$},
           \end{aligned}
\\
             0,&\text{ otherwise,}
           \end{cases}
         \end{equation}
which is by Chebotarev's density theorem equal to the density of primes congruent to $b$ modulo
$q$ and splitting completely in $G(k,a)$,
the leading constant in \eqref{eq:artin_prime_asymptotic} is defined by the absolutely
convergent series
\begin{equation*}
  \delta(a,b,q):=\sum_{k=1}^\infty\mu(k)\eta(a,b,k,q),
\end{equation*}
were $\mu(k)$ is the M\"obius function.
In this paper, we will always work with a more explicit formula for $\delta(a,b,q)$ due to Moree \cite{MR2490093}, which will be reviewed in \S\ref{sec:delta}. Introducing a corresponding modified von Mangoldt function
\begin{equation}\label{eq:def_artin_von_mangoldt}
  \Lambda_a(n):=
  \begin{cases}
    \log p, &\text{ if }n=p^e\text{ with $e\in\NN$,\   $p$ prime and }\FF_p^\times=\langle a\rangle,\\
    0, &\text{ otherwise},
  \end{cases}
\end{equation}
one can phrase \eqref{eq:artin_prime_asymptotic} equivalently for $b\in\{0,\ldots,q-1\}$ as
\begin{equation}\label{eq:artin_mangoldt_asymptotic}
  \sum_{\substack{n\in [N]}}\Lambda_a(b+nq) = q\delta(a,b,q)N+o_{a,q}(N),
\end{equation}
where we have used the notation $[N]=\{1,\ldots,N\}$. Our main result extends this from $\psi(n)=b+qn$ to arbitrary finite complexity\footnote{Recall the characterisation of \emph{finite complexity} from \cite[Lemma 1.6]{MR2680398}: no two of the forms $\psi_i$ have linearly dependent linear parts.
} systems of affine-linear forms
$\Psi=(\psi_1,\ldots,\psi_t):\ZZ^s\to\ZZ^t$ as in \cite[Definition 1.1 and Definition 1.5]{MR2680398}. For $N\in\NN$, the \emph{size of $\Psi$ relative to $N$} is defined as
\begin{equation*}
  \vecnorm{\Psi}_N:=\sum_{i=1}^t\sum_{j=1}^s|\dot \psi_i(e_j)| + \sum_{i=1}^t\left|\frac{\psi_i(0)}{N}\right|,
\end{equation*}
where $\dot\psi_i=\psi_i-\psi_i(0)$ is the linear part of $\psi_i$ and $e_1,\ldots,e_s$ is the standard basis of $\ZZ^s$. Let
$\va=(a_1,\ldots,a_t)\in\ZZ^t$ such that no $a_i$ is equal to $-1$ or
a perfect square. For any $q\in\NN$, we define the density
\begin{equation}\label{eq:def_sigma}
  \sigma_{\va,\Psi}(q) := \EE_{n\in(\ZZ/q\ZZ)^s}\prod_{i\in[t]}\frac{q\delta(a_i,\psi_i(n),q)}{\delta(a_i,0,1)}=q^{t-s}\sum_{n\in(\ZZ/q\ZZ)^s}\prod_{i\in[t]}\frac{\delta(a_i,\psi_i(n),q)}{\delta(a_i,0,1)}.
\end{equation}
Moreover, we let $\Delta_{a_i}$ denote the discriminant of the quadratic field $\QQ(\sqrt{a_i})$ and define
\begin{equation*}
  \cD_\va := \lcm(|\Delta_{a_i}|\where 1\leq i\leq t).
\end{equation*}
With this setup in place, we can state our first main result, a version of \cite[Main Theorem]{MR2680398} with prescribed primitive roots, conditional upon the minimal version of GRH that is currently needed to even establish the existence of infinitely many primes with these prescribed primitive roots.

\begin{theorem}\label{thm:main_artin}
  Let $s,t,L,N\in\NN$. Let $\va=(a_1,\ldots,a_t)\in\ZZ^t$ such that no $a_i$ is equal to $-1$ or a perfect square, and assume that $\HRH(a_i)$ holds for all $1\leq i\leq t$. Let $\Psi=(\psi_1,\ldots,\psi_t):\ZZ^s\to\ZZ^t$ be a system of affine-linear forms of finite complexity with size $\vecnorm{\Psi}_N\leq L$. Let $X\subseteq [-N,N]^s$ be a convex set. Then
  \begin{equation*}
    \sum_{n\in X\cap\ZZ^s}\prod_{i\in[t]}\Lambda_{a_i}(\psi_i(n)) = \vol(X\cap\Psi^{-1}(\RR_+^t))\fS(\va,\Psi)+o_{s,t,\va,L}(N^s),
  \end{equation*}
  where
  \begin{equation*}        \fS(\va,\Psi):=\Big(\prod_{i\in[t]}\delta(a_i,0,1)\Big)\sigma_{\va,\Psi}(\cD_{\va})\prod_{p\nmid\cD_{\va}}\sigma_{\va,\Psi}(p).
  \end{equation*}
Moreover, for each prime $p\nmid\cD_\va$, we have $\sigma_{\va,\Psi}(p)=1+O_{s,t,L,\va}(p^{-2})$, and thus the infinite product $\prod_{p\nmid\cD_\va}\sigma_{\va,\Psi}(p)$ converges absolutely. 
\end{theorem}

Note that, assuming $\HRH(a_i)$, the quotient
$\delta(a_i,\psi_i(n),q)/\delta(a_i,0,1)$ appearing as the summand in
\eqref{eq:def_sigma} can can be interpreted as the conditional
probability that a random prime with primitive root $a_i$ is congruent
to $\psi_i(n)$ modulo $q$.
Our densities
$\sigma_{\va,\Psi}(q)$ are thus in perfect analogy with the densities
\begin{equation}\label{eq:def_beta_p}
\beta_{\Psi, q}=\beta_q:=\EE_{n\in(\ZZ/q\ZZ)^s}\prod_{i\in[t]}\frac{q\one_{\gcd(\psi_i(n),q)=1}}{\phi(q)}
=
q^{t-s}\sum_{n\in(\ZZ/p\ZZ)^s}\prod_{i\in[t]}\frac{\one_{\gcd(\psi_i(n),q)=1}}{\phi(q)}
\end{equation}
for unrestricted primes in \cite[(1.6)]{MR2680398}, where $\phi(q)$ denotes Euler's totient and for any proposition $P$ we write $\one_{P}=1$ if $P$ is true and $\one_P=0$ otherwise. Here,
$\one_{\gcd(\psi_i(n),q)=1}/\phi(q)$ can be seen as the probability
that a random prime is congruent to $\psi_i(n)\bmod q$ by the prime number theorem in arithmetic progressions.
Contrary to \cite{MR2680398}, our factor $\fS(\Psi,\va)$ does not factor
completely as an Euler product, hence the need to introduce the
quantity $\cD_\va$. Moreover, our formula features an additional
global factor $\prod_{i=1}^t\delta(a_i,0,1)$, where of course
$\delta(a_i,0,1)$ should be seen as the probability that a random
prime has primitive root $a_i$. Both of these phenomena were already
observed in a special case in \cite{MR4201547}, see Example
\ref{ex:vinogradov} below.

\begin{remark}\label{rmk:artin}
  In typical applications, the convex set $X$ will satisfy
  $\vol(X\cap\Psi^{-1}(\RR_+^t))\asymp N^s$. It may happen that
  $\fS(\va,\Psi)=0$ and thus the main term in Theorem
  \ref{thm:main_artin} vanishes. In this case, there may be no,
  finitely many, or infinitely many $n\in\ZZ^s$ for which each
  $\psi_i(n)$ is prime with primitive root $a_i$.

  If $\sigma_{\va,\Psi}(q)=0$ for some
  $q\in \{\cD_\va\}\cup\{p\ :\ p\nmid \cD_\va,\ p\ll_{s,t,L,\va}1\}$, then for
  each $c\in(\ZZ/q\ZZ)^s$ there is some $i\in [t]$ with
  $\delta(a_i,\psi_i(c),q)=0$. From Moree's explicit description of
  $\delta(a,b,q)$, see \S\ref{sec:delta}, it is straightforward to see that
  then every prime congruent to $\psi_i(c)\bmod q$ with primitive root $a_i$
  must lie in the set $\{2, \gcd(\psi_i(c),q)\}$. Taking all possible
  candidates for all $c$, we end up with a finite collection of cases, in each
  of which one of the forms $\psi_i(n)$ has to take a fixed value. In each
  case, one can use linear algebra to either show that no solutions exist, or
  express the system in $s-1$ variables. However, this new system might have
  infinite complexity, so Theorem \ref{thm:main_artin} might not apply any
  longer. See the end of \S\ref{sec:prog_example} for an example.
\end{remark}

\subsection{Primes with prescribed Artin symbols}
Our second main result is unconditional. Here, instead of prescribed primitive roots, we require the primes to have prescribed Artin symbols $[K/\QQ,p]$ for some finite Galois extensions $K/\QQ$. Recall that this Artin symbol is defined for primes $p$ unramified in $K$ and yields a conjugacy class in the Galois group $\Gal(K/\QQ)$, the class of Frobenius elements of prime ideals $\ppp$ of the ring of integers of $K$ lying above $p$. For any such conjugacy class $C$, we define the \emph{Chebotarev-von Mangoldt function}
\begin{equation}\label{def:lambda_K_C}
  \Lambda_{K,C}(n) :=
  \begin{cases}
    \log p,&\text{ if }n=p^e\text{ with $e\in\NN$, $p$ prime, unramified in $K$ with } [K/\QQ,p]=C,\\
    0,&\text{ otherwise.}
  \end{cases}
\end{equation}
For $q\in\NN$, $b\in\ZZ$ with $\gcd(b,q)=1$, recall that $\sigma_b\in\Gal(\QQ(\mu_q)/\QQ)$ is the automorphism with $\zeta_q\mapsto\zeta_q^b$. For any conjugacy class $C\subseteq\Gal(K/\QQ)$, let
\begin{equation}\label{eq:def_eta_K_C}
  \eta_{K,C}(b,q) :=
  \begin{cases}
    \frac{|C|}{[K(\mu_q):\QQ]},
    &\begin{aligned}
&\text{ if $\gcd(b,q)=1$ and $\sigma_b$ restricts to the same element}\\
&\text{ of $\Gal(K\cap \QQ(\mu_q)/\QQ)$ as the elements of $C$},
\end{aligned}
\\
    0, &\text{ otherwise.}
  \end{cases}
\end{equation}
This generalises the quantities defined in \eqref{eq:def_eta}, as clearly
\begin{equation}\label{eq:eta_special_case}
\eta(a,b,k,q)=\eta_{G(k,a),\{\ident\}}(b,q).
\end{equation}
Chebotarev's density theorem applied to $K(\mu_q)$ shows that
  \begin{equation}\label{eq:cheb}
  \sum_{\substack{n\in[N]}}\Lambda_{K,C}(b+nq) = q\eta_{K,C}(b,q)N+o_{K,q}(N).
\end{equation}
As before, we extend
\eqref{eq:cheb} from $\psi(n)=b+qn$ to arbitrary finite complexity
systems of affine-linear forms
$\Psi=(\psi_1,\ldots,\psi_t):\ZZ^s\to\ZZ^t$. To this end, let
$\vK=(K_1,\ldots,K_t)$ be number fields normal over $\QQ$ and $\vC=(C_1,\ldots,C_t)$,
where $C_i$ is a conjugacy class in $\Gal(K_i/\QQ)$.
For any $q\in\NN$, we define the density
\begin{equation}\label{eq:def_tau}
  \tau_{\vK,\vC,\Psi}(q) := \EE_{n\in(\ZZ/q\ZZ)^s}\prod_{i\in[t]}\frac{q\eta_{K_i,C_i}(\psi_i(n),q)}{\eta_{K_i,C_i}(0,1)}=q^{t-s}\sum_{n\in(\ZZ/q\ZZ)^s}\prod_{i\in[t]}\frac{\eta_{K_i,C_i}(\psi_i(n),q)}{\eta_{K_i,C_i}(0,1)}.
\end{equation}
Furthermore, let
\begin{equation}\label{eq:def_DK}
\cD_{\vK}:=\lcm(\Phi_{K_i^\ab}\where 1\leq i\leq t),
\end{equation}
where $\Phi_{K_i^\ab}\in\NN$ is
the finite part of the conductor\footnote{In concrete terms, $\Phi_{K_i^\ab}$ is the minimal $f\in\NN$ such that $K_i^\ab\subseteq \QQ(\mu_f)$. Such $f$ exist by the Kronecker-Weber theorem.} of the maximal abelian subextension $K_i^\ab/\QQ$ of $K_i/\QQ$.
With this setup in place, our second main result is the following version of the theorem of Green and Tao.

\begin{theorem}\label{thm:main_cheb}
  Let $s,t,L,N\in\NN$. Let $\vK=(K_1,\ldots,K_t)$ and
  $\vC=(C_1,\ldots,C_t)$, where $K_i/\QQ$ is a Galois number field
  and $C_i\subseteq \Gal(K_i/\QQ)$ a conjugacy class for all
  $i\in[t]$. Let $\Psi=(\psi_1,\ldots,\psi_t)$ be a system of
  affine-linear forms of finite complexity with size
  $\vecnorm{\Psi}_N\leq L$. Let $X\subseteq[-N,N]^s$ be a convex
  set. Then
  \begin{equation*}
    \sum_{n\in X\cap\ZZ^s}\prod_{i\in[t]}\Lambda_{K_i,C_i}(\psi_i(n)) = \vol(X\cap\Psi^{-1}(\RR_+^t))\fS(\vK,\vC,\Psi) + o_{\vK,\vC,s,t,L}(N^s),
  \end{equation*}
  where
  \begin{equation*}
    \fS(\vK,\vC,\Psi):=\Big(\prod_{i\in[t]}\eta_{K_i,C_i}(0,1)\Big)\tau_{\vK,\vC,\Psi}(\cD_{\vK})\prod_{p\nmid\cD_{\vK}}\tau_{\vK,\vC,\Psi}(p).
  \end{equation*}
Here, the infinite product $\prod_{p\nmid\cD_\vK}\tau_{\vK,\vC,\Psi}(p)$ over all primes not dividing $\cD_\vK$ converges absolutely. 
\end{theorem}

The factor $\fS(\vK,\vC,\Psi)$ in the  main term has an analogous interpretation as in Theorem
\ref{thm:main_artin}.  In particular, the quotient
$\eta_{K_i,C_i}(\psi_i(n),q)/\eta_{K_i,C_i}(0,1)$ in
\eqref{eq:def_tau} can be seen as the conditional probability that a
random prime $p$ with Artin symbol $[K_i/\QQ,p]=C_i$ is congruent to
$\psi_i(n)$ modulo $q$.
In contrast to Theorem \ref{thm:main_artin}, these conditional probabilities simplify considerably:
we will see in Lemma \ref{lem:euler_factor_bound_cheb} that $\tau_{\vK,\vC,\Psi}(\cD_\vK)=\beta_{\Psi,p}$ whenever $p\nmid\cD_K$, with $\beta_{\Psi,p}$ the same density as for unrestricted primes in \cite{MR2680398}, defined in \eqref{eq:def_beta_p}. Hence, we can also write 
\begin{equation}\label{eq:cheb_leading_term}
\fS(\vK,\vC,\Psi) = \EE_{n\in(\ZZ/\cD_\vK\ZZ)^s}\Big(\prod_{i\in[t]}\cD_{\vK}\eta_{K_i,C_i}(\psi_i(n),\cD_\vK)\Big)\prod_{p\nmid\cD_{\vK}}\beta_{\Psi,p},
\end{equation}
which is very close to the formula for unrestricted primes in \cite[Main Theorem]{MR2680398}.

\begin{remark}
  Analogous statements as in Remark \ref{rmk:artin} apply. One easily sees directly from the definition that if $\eta_{K,C}(b,q)=0$, then there is at most one possible prime $p\equiv b\bmod q$ with $[K/\QQ,p]=C$, namely $\gcd(b,q)$.

  It is interesting to note that whether or not $\fS(\vK,\vC,\Psi)$ is positive does not depend on the full data of $\vK$ and $\vC$. From \eqref{eq:cheb_leading_term}, one sees that the only information on $\vK,\vC$ relevant for this question are the restrictions of the conjugacy classes $C_i$ to the maximal abelian subextensions $K_i^\ab$ of $K_i$.
\end{remark}

The theorem of Green and Tao \cite{MR2680398} has been applied in combination
with fibration or specialisation techniques in more algabraically-flavored
contexts, where it is used to show that certain linear forms, e.g. the linear
factors of a well-chosen discriminant polynomial, take prime values simultaneously
infinitely often (see e.g. \cite{MR3292295,MR4094712,MR4177544,MR4736376}). Theorem
\ref{thm:main_cheb} adds the capability of prescribing splitting conditions at
these prime values, which we hope will be useful. To facilitate applications,
we state the following implication of Theorem \ref{thm:main_cheb}.

\begin{corollary}\label{cor:existence_cheb}
  Let $s,t,K_i,C_i$ be as in Theorem \ref{thm:main_cheb},
  $\Psi=(\psi_1,\ldots,\psi_t)$ a system of affine-linear forms and
  $X\subseteq \RR^s$ an open convex cone, such that
  \begin{enumerate}
  \item no two of the forms $\psi_i,\ldots,\psi_t$ have linearly dependent linear parts,
  \item the product $\prod_{i\in[t]}\psi_i$ has no fixed prime divisor, i.e. for
    every prime $p$ there is $n\in\ZZ^s$ such that none of
    $\psi_1(n),\ldots,\psi_t(n)$ are divisible by $p$,
  \item there exists $x\in K$ with $\dot\psi_1(x),\ldots,\dot\psi_k(x)>0$, where $\dot\psi_i=\psi_i-\psi_i(0)$ is the linear part of the form $\psi_i$,
  \item there are $q\in\NN$ with $\cD_{\vK}\mid q$ and $n\in\ZZ^s$, such that,
    for all $1\leq i\leq t$, we have $\gcd(q,\psi_i(n))=1$ and the automorphism
    $\sigma_{\psi_i(n)}:\zeta_q\mapsto\zeta_q^{\psi_i(n)}$ of $\QQ(\mu_q)$ restricts to the
    same element of $\Gal(K_i^\ab/\QQ)$ as the elements of $C$.
  \end{enumerate}
  Then there are infinitely many $n\in\ZZ^s\cap X$ for which the values $\psi_{i}(n)$ are distinct primes with $[K_i/\QQ,\psi_i(n)]=C_i$ for all $i$.
\end{corollary}
Let us emphasize again that condition \emph{(4)} does not depend on
the full fields $K_i$ and conjugacy classes $C_i$, but only on the
restrictions of $C_i$ to the maximal abelian subextensions of
$K_i$. Hence, \emph{(4)} is a collection of congruence conditions on
$n$ modulo $\cD_\vK$. The inclusion of $q$ is only there for
convenience in applications; if \emph{(4)} holds for any $q$ with
$\cD_\vK\mid q$, then also for $q=\cD_\vK$.

\subsection{Application in inverse Galois theory}\label{sec:galois}
A special case of Corollary \ref{cor:existence_cheb} was asked for explicitly
by Kim and K\"onig in the context of inverse Galois problems with local
conditions. In \cite[Remark 6]{MR4177544}, they ask whether, for given $n\geq 2$, an analog of the theorem of Green and Tao can prove the
existence of infinitely many pairs $(s,t)\in\ZZ^2$, such that
\begin{align*}
    &|t|& &\text{ is a prime congruent to } 1\bmod n,\\
    &|s|& &\text{ is a prime congruent to } 1\bmod n-1,\\
    &|(n-1)^{n-1}t-n^ns|& &\text{ is a prime splitting completely in } F_{n-1},
\end{align*}
where $F_k$ is the splitting field of the polynomial\footnote{The problem is
  formulated for the reciprocal polynomial of $f_k$, which clearly has
  the same splitting field.}
\begin{equation*}
  f_k := x^{k-1}+2x^{k-2}+3x^{k-3}+\cdots+(k-1)x+k
\end{equation*}
over $\QQ$. In Appendix \ref{sec:appendix}, which is joint work with
Joachim K\"onig, we apply Corollary \ref{cor:existence_cheb} to answer
this question positively for all $n\geq 2$. Using this to replace an
application of the original theorem of Green and Tao in the proof of
\cite[Theorem 5.5]{MR4177544}, we prove the following result on
locally cyclic $\Sn_n$-extensions with prescribed Artin symbol at
finitely many primes. This generalisation of the special case $n=5$
treated in \cite[Theorem 5.5]{MR4177544} was suggested by Kim and
K\"onig as the motivation for their question. Recall that a Galois
number field $K/\QQ$ is \emph{locally cyclic} if all decomposition
groups are cyclic.

\begin{theoremA1}
  Let $n\geq 2$ be an integer. Let $M_1$ be a finite set of primes and $M_2$ a finite set of
  sufficiently large primes in terms of $n$. For each $p\in M_2$, let
  $C_p$ be a conjugacy class of $\Sn_n$. Then there are infinitely many
  linearly disjoint Galois-extensions $K/\QQ$, such that:
  \begin{enumerate}
  \item $\Gal(K/\QQ)\simeq \Sn_n$,
  \item $K/\QQ$ is locally cyclic,
  \item All $p\in M_1\cup M_2$ are unramified in $K$,
  \item For each $p\in M_2$, we have $[K/\QQ,p]=C_p$.
  \end{enumerate}
\end{theoremA1}

This result is new even without the conditions \emph{(3)} and \emph{(4)}. It answers positively a question of Bubboloni and Sonn \cite{MR3556829}, who asked whether locally cyclic Galois extensions of $\QQ$ with Galois group $\Sn_n$ exist for all $n$. Let us note that \cite[Theorem 5.5]{MR4177544}
also covers the case of $\PGL_2(7)$. A generalisation to $\PGL_2(p)$
for arbitrary primes $p$ follows from the main result of
\cite{MR4736376}.

\subsection{Previous results and examples}

Three significant special cases of our results are already present in
the literature and of relevance to our work. Two are related to
Vinogradov's three primes theorem, which concerns a classical
constellation of complexity $1$, and one to primes in arithmetic
progressions.

\begin{example}\label{ex:kane}
  Let $t\geq 3$, let $\vK,\vC$ as in Theorem \ref{thm:main_cheb}, and let $c_1,\ldots,c_t\in\ZZ$ with $\gcd(c_1,\ldots,c_t)=1$, let $N,M\in\NN$. In \cite{MR3060874}, Kane obtains an asymptotic formula for
  \begin{equation*}
    \sum_{\substack{x_i\in [N]\\ c_1 x_1+\cdots+c_t x_t = M}} \prod_{i\in[t]} \Lambda_{K_i,C_i}(x_i).
  \end{equation*}
  Actually, Kane's formula is stated in terms of primes $p_i$ with $[K_i/\QQ,p_i]=C$ weighted
  by $\log(p_i)$ instead of integers $n_i$ weighted by
  $\Lambda_{K_i,C_i}(n_i)$, but the contribution of prime powers
  making up the difference between these two settings is clearly
  negligible. Considering $c_1,\ldots,c_t$ as fixed, we may assume
  that $M\ll N$, as otherwise the sum is zero.  As explained 
  in \cite[Theorem 1.8]{MR2680398} and its proof, one may choose a
  basis for the affine sublattice of $\ZZ^t$ given by
  $c_1x_1+\cdots+c_tx_t=M$, and thus describe it as the image of an
  affine-linear system $\Psi:\ZZ^{t-1}\to\ZZ^t$ of size
  $\vecnorm{\Psi}\ll 1$ and complexity $\leq 1$.

  We describe Kane's formula in \S\ref{subsec:comparison_Kane} and
  show that it agrees with the corresponding special case of Theorem
  \ref{thm:main_cheb}. In order to obtain his result, Kane generalises
  the proof of Vinogradov's three primes theorem presented in \cite[\S
  19]{MR2061214}, based on classical analytic number theory
  surrounding the Hardy-Littlewood circle method. Some of Kane's
  techniques play a role in our work. Most importantly, our
  Proposition \ref{prop:norms_vs_equidist_nilsequences} can be seen as
  a higher-order Fourier analysis version of \cite[Lemma 21]{MR3060874},
  Kane's exponential sum estimate for polynomial phases evaluated at
  ideal norms.
\end{example}

We should mention that Kane's result features a quantitative error
term saving an arbitrary power of $\log N$, whereas our error terms
are purely qualitative.  There has been great recent
  progress in obtaining quantitative versions of the methods of Green,
  Tao and Ziegler.  Tao and Ter\"av\"ainen \cite{taotera}, based on
  Manners' quantitative inverse theorem for Gowers norms
  \cite{manners}, obtained a version of the theorem of Green and Tao
  with quantitative error terms saving a small power of $\log\log
  N$. A very recent breakthrough in the area is the quasipolynomial
  inverse theorem of Leng, Sah and Sawhney \cite{LSS}. One might hope
  that applications of this work could lead to quantitative error
  terms and also dispose of the need for pseudorandom majorants in our
  work. In the context of Theorem \ref{thm:main_artin}, this does not
  seem very promising with what we have at the moment, as our bounds for correlations with
  nilsequences are extremely inexplicit, see Theorem
  \ref{thm:W_tricked_artin_mangoldt_with_nilsequences_nonconst_avg}. Although
  we are not following this direction in the present work, we would be
  very interested in seeing explicit bounds in the context of Theorem
  \ref{thm:main_cheb}.

  There are some other recent analytic results concerning patterns in
  primes with restricted Artin symbols, for example generalisations
  \cite{MR3344173,MR3667488} of Maynard's \cite{MR3272929} celebrated
  bounded gaps result, and results on primes in Beatty sequences
  \cite{MR4130085}. Similarly, there are some recent results and
  conjectures on gaps between primes with prescribed primitive roots \cite{MR3272281,MR3518382,MR4517480}.

\begin{example} \label{ex:vinogradov}
  Let $t=3$ and $\va$ as in Theorem \ref{thm:main_artin}. In \cite{MR4201547}, Frei, Koymans and Sofos prove, conditionally upon $\HRH(a_i)$ for $1\leq i\leq 3$, an asymptotic formula for
\begin{equation*} 
\sum_{\substack{x_i\in\NN\\x_1+x_2+x_3=N}}\prod_{i=1}^{3}\Lambda_a(x_i),
\end{equation*}
i.e. for Vinogradov's three primes theorem with prescribed primitive roots. Similarly as in Example \ref{ex:kane}, this is a special case of our Theorem \ref{thm:main_artin} for the system $\Psi(n_1,n_2)=(n_1,n_2,N-n_1-n_2)$ of complexity $1$. We present the formula of \cite{MR4201547} in \S\ref{subsec:comparison_FKS} and show that it agrees with Theorem \ref{thm:main_artin}.

To prove their result, the authors of \cite{MR4201547} use the
Hardy-Littlewood circle method to show a version of the three primes
theorem with additional splitting conditions, which they then insert
into the technique of Hooley's conditional proof of Artin's conjecture
\cite{MR207630}. This global application of Hooley's techniques leads
to a highly complex expression for the leading term as an infinite
sums over singular series twisted by M\"obius functions, requiring a long and intricate analysis in order to
render this leading term in the form of an (incomplete) Euler product
as in Theorem \ref{thm:main_artin}. In our proof of Theorem
\ref{thm:main_artin}, we apply Hooley's techniques more locally,
namely only in order to show non-correlation of (a suitably normalised
and $W$-tricked version of) our Artin-von Mangoldt functions with
nilsequences. This makes the shape of our main term in Theorem
\ref{thm:main_artin} arise naturally from the $W$-trick
of Green and Tao and factorisation properties of the densities
$\delta(a_i,b,q)$, and might therefore serve as a more satisfying explanation
for the observations made in \cite{MR4201547}.
\end{example}

\begin{example} \label{ex:2}
  By the Green-Tao theorem \cite{MR2415379}, any positive-density subset $\cP$ of the primes contains arbitrarily long arithmetic progressions, i.e., for any $t\in\NN$ the system
  \begin{equation*}
    \Psi(n_1,n_2)=(n_1,n_1+n_2,\ldots,n_1+(t-1)n_2)
  \end{equation*}
  of complexity $t-2$ takes infinitely many values in $\cP^t$. Taking $\cP$ to
  be the set of primes with $[K/\QQ,p]=C$ for some fixed $K,C$, which has
  positive density in the primes by \eqref{eq:cheb} with $q=1$, one obtains a
  weakened special case of Theorem \ref{thm:main_cheb}. Taking $\cP$ to be the
  set of primes with primitive root $a$, one obtains a weakened special case of
  Theorem \ref{thm:main_artin}. In both cases, the precise asymptotic formulae
  in our results are new. In \S\ref{sec:prog_example}, we work out for
  illustration the leading constant $\fS(\va,\Psi)$ in Theorem
  \ref{thm:main_artin} in case $t=3$ for some values of $a$. For example, if
  $a=2$, we get
\[
\fS((2,2,2),\Psi)=2\cA_2^3\prod_{p\neq 2}\left(1-\frac{p^4-p^3-3p^2-2p-1}{(p^2-p-1)^3}\right),
\]
where $\cA_{2}$ is the constant arising from the naive heuristic for the density of primes with primitive root $2$, defined at the start of \S\ref{sec:delta}.

The proof of \cite{MR2415379}, establishing a transference principle to
transfer Szemeredi's theorem \cite{MR369312} from the integers to the primes,
would work, more generally, for \emph{homogeneous, translation-invariant}
systems $\Psi$, i.e. systems $\Psi:\ZZ^s\to\ZZ^t$ of \emph{linear} forms whose
image contains the point $(1,\ldots,1)\in\ZZ^t$. Hence, also for these systems,
existence-versions of our main theorems were known already.
\end{example}

For arbitrary finite-complexity systems, a general transference principle was recently established by Bienvenu, Shao and Ter\"av\"ainen \cite[Theorem 3.2]{MR4564765}. The authors apply their transference principle to show lower bounds for linear constellations in, e.g., \emph{Chen-primes}. Their result could also be applied to ($W$-tricked versions of) our functions $\Lambda_a$ and $\Lambda_{K,C}$ and would lead to asymptotic lower bounds of the right order of magnitude in Theorem \ref{thm:main_artin} and Theorem \ref{thm:main_cheb}. However, the verification of the hypotheses of \cite[Theorem 3.2]{MR4564765} in our situation requires essentially the same work that takes up the bulk of our paper: the investigation of correlations of $\Lambda_a$ and $\Lambda_{K,C}$ with nilsequences.

\begin{example} \label{ex:13}
  Let us verify the claim made at the start of the introduction that, conditionally on $\HRH(7)$ our Theorem \ref{thm:main_artin} shows the existence of infinitely many pairs $(n_1,n_2)\in\ZZ^2$ for which all of  \eqref{eq:comp_2_example} are primes with primitive root $7$. To this end, we investigate factor $\fS(\va,\Psi)$ in the main term in Theorem \ref{thm:main_artin} for the system
  \begin{equation*}
  \Psi(n_1,n_2)=(n_1,n_2,n_1+n_2-1,n_1+2n_2-2)
\end{equation*}
of complexity $2$. We do so for various choices of $a_1=a_2=a_3=a_4=a$ in \S\ref{sec:example_intro}. For $a=2,3,5,6$ we see that $\sigma_{\va,\Psi}(|\Delta_a|)=0$, whereas for $a=7,10,11,13$ we get $\fS(\va,\Psi)>0$.
For $a=7$, we compute 
\begin{equation*}
\fS((7,7,7,7),\Psi)=\frac{914838624}{353220125}\cA_{7}^4\prod_{p\neq 2,3,7}\left(1-\frac{p^6-11p^4-4p^3+p^2+4p+1}{(p^2-p-1)^4}\right).
\end{equation*}
where $\cA_{7}$ is the constant arising from the naive heuristic for the density of primes with primitive root $7$, defined at the start of \S\ref{sec:delta}.
\end{example}

\subsection*{Acknowledgements} We thank Pierre-Yves Bienvenu for helpful discussions. C.F. was first exposed to the machinery invented by Green, Tao and Ziegler in a still ongoing collaboration with Lilian Matthiesen. In particular, he thanks Matthiesen for introducing him to her work on polynomial subsequences of equidistributed nilsequences. We thank Lilian Matthiesen and an anonymous referee for pointing out relevant references and bibliographic inaccuracies in an earlier version.

Most of this work was completed while M.T. was a visiting scholar at TU Graz in 2022--2024. The authors thank TU Graz for its hospitality. C.F. was supported by EPSRC grants EP/T01170X/1 and EP/T01170X/2. M.T. was supported by Czech Science Foundation (GA\v{C}R) grant 22-11563O.

\section{Notation, outline and proof strategy}
\subsection{Some notation}
We use asymptotic notation similarly as in \cite{MR2680398}. In $\ll$-, $\asymp$-
and $O$-notation, we indicate any dependences of the implicit constant by
subscripts. For example, $A=O_a(B)$ means that $|A|\leq C(a)B$, where $C(a)$ is
a positive constant depending only on $a$. In $o$-notation, it is understood
that $N$ is the parameter going to infinity. Hence, $A=o_a(B)$ means that for every
$\epsilon>0$ we have  $|A|< \epsilon B$ if only $N\geq N_0(a,\epsilon)$, where
$N_0(a,\epsilon)$ is a positive integer depending only on $a$ and $\epsilon$.  
For a finite set $S$ and a function $f:S\to\CC$, we use the notation
\begin{equation*}
\EE_{x\in S}f(s):=\frac{1}{|S|}\sum_{x\in S}f(x).
\end{equation*}
As already stated, for $N\in\NN$ we let $[N]:=\{1,\ldots,N\}$. We
think of number fields as subfields of $\CC$. We denote the
discriminant of a number field $K$ by $\Delta_K$, the monoid of
non-zero ideals of its ring of integers $\OO_K$ by $\ideals$, the
maximal abelian subextionsion of $K$ over $\QQ$ by $K^\ab$, and the
finite part of the conductor of this abelian subfield by
$\Phi_{K^\ab}$. We denote the absolute norm of $\aaa\in\ideals$ by
$\norm\aaa$.

\subsection{Local densities}
In Section \ref{sec:local_densities}, we study the local factors
$\tau_{\vK,\vC,\Psi}(q)$ and $\sigma_{\va,\Psi}(q)$ appearing in
Theorem \ref{thm:main_cheb} and Theorem \ref{thm:main_artin},
respectively. In particular, we will prove that the Euler products
appearing in these results are indeed absolutely convergent.

\subsection{Polynomial subsequences of equidistributed nilsequences}
We prove our main results using the machinery developed by Green, Tao and
Ziegler, in particular the inverse theorem for Gowers norms
\cite{MR2950773}. This makes it necessary to deal with polynomial sequences in
nilmanifolds. Green and Tao have established the quantitative theory of polynomial nilsequences in \cite{MR2877065}. We use the definitions and terminology of this work freely and will refer to it frequently. In particular, we use Mal'cev bases as defined in \cite[Definition 2.1]{MR2877065} and the metrics induced by these Mal'cev bases as defined in \cite[Definition 2.2]{MR2877065}. Lipschitz constants are understood with respect to this metric.

In Section \ref{sec:equidistributed_nilsequences}, we compile some
facts regarding polynomial subsequences of equidistributed
nilsequences, which are probably well known to experts.  Polynomial
subsequences of nilsequences were first investigated by Matthiesen in
\cite{MR2949879,MR3742557}, and the main results in this section are
technical variations of hers. We have included full proofs, as
we were not able to find the exact required statements in the
literature.

\subsection{W-trick}\label{sec:W_trick}
A central aspect in the strategy of Green, Tao and Ziegler is the $W$-trick, which already appeared  in \cite{MR2180408,MR2415379}. It requires us to
consider an increasing function $w(N)$ that grows to infinity with $N$ and
satisfies $1\leq w(N)\leq \frac{1}{3}\log\log N$, a constant $\cD\in\NN$, and
the function
\begin{equation}\label{eq:cheb_def_W}
  W=W(N):=\cD\prod_{\substack{p<w(N)\\p\nmid \cD}}p \ll_\cD (\log N)^{2/3}.
\end{equation}
Throughout the paper, the letters $w$, $\cD$ and $W$ will always refer to the
objects chosen here.  For the proof of Theorem \ref{thm:main_cheb}, we can
make essentially the same choice $w(N)=\log\log\log N$ as in \cite{MR2680398},
except that we need $\cD=\cD_\vK$ instead of $\cD=1$. For the proof of Theorem
\ref{thm:main_artin}, however, the function $w(\cdot)$ needs to grow extremely
slowly with $N$, with the growth rate depending on implicit constants that
appear in particular in
\S\ref{sec:W_tricked_lambda_a_with_nilsequences}. Ultimately the choice of
$w(\cdot)$ will depend only on the data of $s,t,\va,L$ given in the statement
of Theorem \ref{thm:main_artin} and the constant $\epsilon>0$ from the
$o$-notation.

In addition to the von Mangoldt functions already defined, it is convenient to work with the
following versions, which are non-zero only on primes:
\begin{align*}
  \Lambda'(n) &:=
  \begin{cases}
    \log p,&\text{ if }n=p \text{ prime},\\
    0,&\text{ otherwise,}
  \end{cases}\\
  \Lambda_{K,C}'(n) &:=
  \begin{cases}
    \log p,&\text{ if }n=p\text{ prime with $p$ unramified in $K$ and $[K/\QQ,p]=C$},\\
    0,&\text{ otherwise,}
  \end{cases}\\
  \Lambda_{a}'(n) &:=
  \begin{cases}
    \log p,&\text{ if }n=p \text{ prime with }\FF_p^\times=\langle a\rangle,\\
    0,&\text{ otherwise.}
  \end{cases}
\end{align*}
For $b\in \{0,\ldots,W-1\}$, we consider the following $W$-tricked versions of our
various von Mangoldt functions:
\begin{align*}
  \Lambda_{b,W}(n) &:= \frac{\phi(W)}{W}\Lambda(b+nW), &  \Lambda_{b,W}'(n) &:= \frac{\phi(W)}{W}\Lambda'(b+nW), \\
  \Lambda_{K,C,b,W}(n) &:= \frac{\phi(W)}{W}\Lambda_{K,C}(b+nW),&  \Lambda_{K,C,b,W}'(n) &:= \frac{\phi(W)}{W}\Lambda_{K,C}'(b+nW), \\
  \Lambda_{a,b,W}(n) &:= \frac{\phi(W)}{W}\Lambda_a(b+nW),&  \Lambda_{a,b,W}'(n) &:= \frac{\phi(W)}{W}\Lambda_{a}'(b+nW).
\end{align*}
Now \eqref{eq:cheb} suggests that  
\begin{equation*}
  \EE_{n\in[N]}\Lambda_{K,C,b,W}(n) = \phi(W)\eta_{K,C}(b,W)+o_{K,\cD}(1),
\end{equation*}
and \eqref{eq:artin_mangoldt_asymptotic} suggests that, assuming $\HRH(a)$,
\begin{equation}\label{eq:artin_W_asymp}
  \EE_{n\in[N]}\Lambda_{a,b,W}(n) = \phi(W)\delta(a,b,W)+o_{a,\cD}(1).
\end{equation}
As $W$ grows with $N$, one would actually need uniform versions of
\eqref{eq:cheb} and \eqref{eq:artin_mangoldt_asymptotic} to prove the
statements above. For \eqref{eq:cheb}, we will show such a uniform
version in Proposition \ref{prop:better_than_chebotarev}. For
\eqref{eq:artin_mangoldt_asymptotic}, a uniform version was proved in
\cite{MR4033249}, conditionally upon a stronger version of
$\HRH(a)$. It seems difficult to do so assuming just $\HRH(a)$. We
will neither prove nor use \eqref{eq:artin_W_asymp} in this paper. A
slightly weaker version, where the error term also depends on
$w(\cdot)$, follows from a special case of Theorem
\ref{thm:W_tricked_artin_mangoldt_with_nilsequences_nonconst_avg}
below.  As higher prime powers are easily seen to be irrelevant in
these avarages, analogous asymptotics hold for $\Lambda_{K,C,b,W}'$
and $\Lambda_{a,b,W}'$.

We will convince ourselves (Lemma \ref{lem:pseudorandom_majorant}) that the
pseudorandom majorant constructed by Green and Tao in \cite[Proposition
6.4]{MR2680398} for $\Lambda_{b,W}$ is also sufficient for our
$\Lambda_{K,C,b,W}$ and $\Lambda_{a,b,W}$. Standard arguments, including the
generalised von Neumann Theorem \cite[Proposition 7.1]{MR2680398} (see Lemma
\ref{lem:neumann}), then reduce Theorem \ref{thm:main_cheb} to the following
Gowers norm estimate. For the definition of $\vecnorm{\cdot}_{U^{u+1}[N]}$ we
refer to \cite[(B.13)]{MR2680398}. 

\begin{proposition}\label{prop:gowers_estimate_cheb}
  Let $K$ be a Galois number field with $\Phi_{K^\ab}\mid W$, let $C\subseteq\Gal(K/\QQ)$ be a conjugacy class, let $u,N\in\NN$ and $b\in\{0,\ldots,W-1\}$. Then  
  \begin{equation*}
    \vecnorm{\Lambda_{K,c,b,W}'(\cdot)-\phi(W)\eta_{K,C}(b,W)}_{U^{u+1}[N]}= o_{K,u,\cD,w(\cdot)}(1).
  \end{equation*}
\end{proposition}

Similarly, we will deduce Theorem \ref{thm:main_artin} from the following conditional result. 

\begin{proposition}\label{prop:gowers_estimate_artin}
  Let $0<\delta\leq 1$, $u,N\in\NN$, $b\in\{0,\ldots,W-1\}$ and
  $a\in\ZZ\smallsetminus\{-1\}$ not a perfect square, and assume $\HRH(a)$. If
  $|\Delta_a|$ divides $\cD$ and the function $w(\cdot)$ grows sufficiently slowly in terms of
  $\delta,u,a$, then
  \begin{equation*}
    \vecnorm{\Lambda_{a,b,W}'(\cdot)-\phi(W)\delta(a,b,W)}_{U^{u+1}[N]}\leq \delta, 
  \end{equation*}
  if only $N$ is sufficiently large in terms of $\delta,u,a,\cD,w(\cdot)$.
\end{proposition}

The analog of Proposition \ref{prop:gowers_estimate_cheb} and Proposition \ref{prop:gowers_estimate_artin} for the classical von Mangoldt function is the estimate
\begin{equation}\label{eq:gowers_estimate_classical}
  \vecnorm{\Lambda_{b,W}'(\cdot)-1_{\gcd(b,W)=1}}_{U^{u+1}[N]}= o_{u,\cD,w(\cdot)}(1),
\end{equation}
which was proved in \cite[Theorem 7.2]{MR2680398} with the choice of $\cD=1$ and $w(N)=\log\log\log N$ but holds also in general (Lemma \ref{lem:gowers_estimate_classical}).

\subsection{Von Mangoldt model and orthogonality with nilsequences}

Green and Tao use the inverse theorem for Gowers norms \cite{MR2950773} to deduce
\eqref{eq:gowers_estimate_classical} from an orthogonality result of
$\Lambda_{b,W}-1$ with nilsequences (\cite[Proposition
10.2]{MR2680398}). The proof of this result uses
the identity
\begin{equation}\label{eq:Lambda_mu_identity}
  \Lambda(n)=-\sum_{d\mid n}\mu(d)\log d.
\end{equation}
Starting from this identity, Green and Tao reduce \cite[Proposition
10.2]{MR2680398} to a sieve estimate \cite[(12.5)]{MR2680398} and a
strong orthogonality result for the M\"obius function $\mu$ with
nisequences, which they prove in \cite[Theorem 1.1]{MR2877066}. Unfortunately, an
analog of \eqref{eq:Lambda_mu_identity} is not available for our
functions $\Lambda_{K,C}$ and $\Lambda_a$.

Therefore, we proceed in a different way in our proofs of Proposition
\ref{prop:gowers_estimate_cheb} and Proposition
\ref{prop:gowers_estimate_artin}: in order to show orthogonality with
nilsequences, we apply directly the strategy used for the M\"obius
function in \cite{MR2877066}. This makes crucial use of the fact that
$\mu(n)$ is equidistributed in arithmetic progressions, which is
clearly not entirely true for our function
$\Lambda_{K,C,b,W}(n)-\phi(W)\eta_{K,C}(b,W)$ due to progressions
such as $n\equiv 0\bmod b$. To deal with this issue, we replace the
constant average $\phi(W)\eta_{K,C}(b,W)$ by a suitably scaled version
of the $W$-tricked classical von Mangoldt function $\Lambda_{b,W}(n)$
and show that this provides a good model for $\Lambda_{K,C,b,W}(n)$ in all
arithmetic progressions.  With
this model, we show the following orthogonality result, which is
analogous to \cite[Theorem
1.1]{MR2877066}.

\begin{theorem}\label{thm:W_tricked_chebotarev_mangoldt_with_nilsequences_nonconst_avg}
  Let $m, d\geq 1$, $Q\geq 2$.  Let $G/\Gamma$ be a nilmanifold of
  dimension $m$, let $G_\bullet$ be a filtration of $G$ of degree $d$,
  and $g\in\poly(\ZZ,G_\bullet)$ a polynomial sequence. Suppose
  $G/\Gamma$ has a $Q$-rational Mal'cev basis $\cX$ adapted to
  $G_\bullet$, defining a metric $d_\cX$ on $G/\Gamma$.

  Then, for any Lipschitz function $F:G/\Gamma\to [-1,1]$, $A>0$, $N\geq 2$,
  $b\in \{0,\ldots,W-1\}$, Galois number field $K$ with $\Phi_{K^\ab}\mid W$ and conjugacy
  class $C\subseteq\Gal(K/\QQ)$, we
  have
  \begin{align*}
    \EE_{n\leq N}\left(\Lambda_{K,C,b,W}(n)-\phi(W)\eta_{K,C}(b,W) \Lambda_{b,W}(n)\right)&F(g(n)\Gamma)\\ &\ll_{K,m,d,A,\cD} Q^{O_{m,d,A}(1)}(1+\vecnorm{F}_{\Lip})(\log N)^{-A}.
  \end{align*}
\end{theorem}

To prove Theorem
\ref{thm:W_tricked_chebotarev_mangoldt_with_nilsequences_nonconst_avg}, we
follow the approach of \cite{MR2877066} and reduce to the special case of
equidistributed nilsequences and test functions of zero average (Proposition
\ref{prop:equidist_nilsequences}). Using some class field theory, we further reduce
this result to the following theorem concerning the von Mangoldt function for
ideals $\aaa\in\ideals$,
\begin{equation}\label{eq:def_ideal_von_mangoldt}
  \Lambda_K(\aaa) =
    \begin{cases}
      \log \norm\ppp &\text{ if }\aaa=\ppp^e \text{ with $\ppp$ prime ideal and
        $e\in\NN$},\\
      0 &\text{ otherwise.}
    \end{cases}.
  \end{equation}
  A \emph{Dirichlet character} of $K$ is a finite-order Hecke character.

\begin{theorem} \label{thm:ideal_von_mangoldt_equidist_nilsequences} For all integers $m\geq 0$ and
  $d,D\geq 1$, there is a constant $c(m,d,D)>0$, such that the following
  holds.

  Let $N\in\NN$ be sufficiently large depending
  on $m,d,D$. Let $\delta\in(0,1)$ and $Q\geq 2$. Let $G/\Gamma$ be an
  $m$-dimensional nilmanifold with a filtration $G_{\bullet}$ of degree $d$ and
  a $Q$-rational Mal'cev basis $\cX$ adapted to $G_{\bullet}$. Let
  $g\in\poly(\ZZ,G_\bullet)$ and suppose that $(g(n)\Gamma)_{n\in[N]}$ is
  totally $\delta$-equidistributed. Let $K$ be a number field of degree $D$
  and $\xi$ a Dirichlet character of $K$ of modulus $\mmm$. Then, for any Lipschitz function
  $F:G/\Gamma\rightarrow\RR$ with $\int_{G/\Gamma}F=0$ and for any arithmetic
  progression $P\subseteq [N]$ of size at least $N/Q$, we have the bound
  \begin{equation}\label{eq:prop2.1analog}
  |\EE_{\norm \aaa\leq N}\Lambda_K(\aaa)\xi(\aaa)\one_P(\norm \aaa)F(g(\norm \aaa)\Gamma)|\ll_{m,d,K,\mmm}
  \delta^{c(m,d,D)}Q\vecnorm{F}_{\textnormal{Lip}}(\log N)^2. 
  \end{equation}
\end{theorem}

We prove Theorem \ref{thm:ideal_von_mangoldt_equidist_nilsequences} in
\S\ref{sec:ideal_von_mangoldt} and deduce Theorem
\ref{thm:W_tricked_chebotarev_mangoldt_with_nilsequences_nonconst_avg}
from it in \S\ref{sec:von_mangoldt_proxy}. A generalisation of the results of Green and Tao to prime \emph{elements} in number fields was obtained by Kai \cite{Kai} and recently applied to Hilbert's tenth problem by Koymans and Pagano \cite{KP}. Our Theorem \ref{thm:ideal_von_mangoldt_equidist_nilsequences} concerns norms of \emph{ideals} and is independent of Kai's work.

Let us just mention in
passing that it would also be possible to deal with $\Lambda_{K,C}$ in
terms of the von Mangoldt function $\Lambda_L$ for ideals in certain
subfields $L$ of $K$ (see Lemma \ref{lem:passing_to_ideals}). For
these ideal von Mangoldt functions, an analogue of
\eqref{eq:Lambda_mu_identity} exists and our techniques from
\S\ref{sec:ideal_von_mangoldt} would be capable of also proving an
analogue of \cite[Theorem 1.1]{MR2877066} for the M\"obius function
$\mu_L$ of ideals of $L$. However, it seems quite difficult to prove the required
analogue of the sieve estimate \cite[(12.5)]{MR2680398} for
$\Lambda_L$, as that could require power-saving estimates for linear
correlations of the ideal counting function
$a_L(n):=|\{\aaa\text{ ideal of L}\ :\ \norm\aaa=n\}|$. Such
correlations could be studied using techniques of Browning
and Matthiesen \cite{MR3742196}, but to achieve power-saving seems unlikely. Hence, we have decided to
proceed in a different way.

\subsection{Hooley's method}

Recall the definition of the number fields $G(k,a)$ in
\eqref{eq:def_Gka}. By index calculus, a prime $p\nmid a$ has
primitive root $a$ if and only if it does not split completely in any
of the fields $G(k,a)$ with $k$ prime. In \cite{MR207630}, Hooley deduces the case
$q=1$ of \eqref{eq:artin_prime_asymptotic} by applying
inclusion-exclusion to this observation and then truncating via $\HRH(a)$. In order to study
correlations of our $W$-tricked Artin-von Mangoldt functions
$\Lambda_{a,b,W}(n)$ with nilsequences, we modify Hooley's method to
deduce the following result, which also incorporates the observation
from \cite{MR4201547} that the truncation parameter can be chosen to
grow to infinity with $N$ arbitrarily slowly and still yield a
saving. For $k\in\NN$, we denote by $p_+(k)$ the largest prime divisor
of $k$ if $k>1$ and $p_+(1):=1$.

\begin{theorem}\label{thm:hooley}
  Let $a\in\ZZ$ be not equal to $-1$ or a perfect square, and assume that $\HRH(a)$ holds true.
  Let $F:\NN\to\CC$ be a function with $|F(n)|\leq C$ for all $n\in\NN$. Let
  $b\in\{0,\ldots,W-1\}$ and assume that $N\in\NN$ is sufficiently large in terms of $a$. Then, 
  \begin{align*}
    \EE_{n\leq N}\Lambda_{a,b,W}(n)F(n) &=
    \sum_{\substack{k\in\NN\\p_+(k)\leq w(N)}}\mu(k)\EE_{n\leq
    N}\Lambda_{G(k,a),\{\ident\},b,W}(n)F(n) +
    O_{a,C,\cD}\left(\frac{1}{w(N)}\right).
  \end{align*}
\end{theorem}

We will prove Theorem \ref{thm:hooley} in \S\ref{sec:hooley}. We will use in in \S\ref{sec:W_tricked_lambda_a_with_nilsequences} to deduce from Theorem \ref{thm:W_tricked_chebotarev_mangoldt_with_nilsequences_nonconst_avg} the following conditional version for $\Lambda_{a,b,W}$.

\begin{theorem}\label{thm:W_tricked_artin_mangoldt_with_nilsequences_nonconst_avg}
  Let $m, d, M\geq 1$, $Q\geq 2$.  Let $G/\Gamma$ be a nilmanifold of
  dimension $m$, let $G_\bullet$ be a filtration of $G$ of degree $d$,
  and $g\in\poly(\ZZ,G_\bullet)$ a polynomial sequence. Suppose
  $G/\Gamma$ has a $Q$-rational Mal'cev basis $\cX$ adapted to
  $G_\bullet$, defining a metric $d_\cX$ on $G/\Gamma$.

  Let $a\in\ZZ$ be not equal to $-1$ or a perfect square and assume
  $\HRH(a)$. Suppose that $w(N)$ grows to infinity with $N$
  sufficiently slowly in terms of $a, m, d, Q, M$, and that
  $|\Delta_a|$ divides $\cD$.
  
  Then, for any Lipschitz function $F:G/\Gamma\to [-1,1]$ with
  $\vecnorm{F}_{\Lip}\leq M$,
  $N\geq 2$ and
  $b\in\{0,\ldots,W-1\}$, we have
  \begin{equation*}
    \left|\EE_{n\leq
        N}\left(\Lambda_{a,b,W}(n)-\phi(W)\delta(a,b,W)\Lambda_{b,W}(n)\right)F(g(n)\Gamma)\right|\ll_{a,\cD,w(\cdot)}\frac{1}{w(N)}.
  \end{equation*}
\end{theorem}

\subsection{Completion of proofs}

We deduce Proposition \ref{prop:gowers_estimate_artin} and Theorem \ref{thm:main_artin} from Theorem \ref{thm:W_tricked_artin_mangoldt_with_nilsequences_nonconst_avg} in \S\ref{sec:proof_main_artin}, following the method developed by Green and Tao in \cite{MR2680398}. In particular, we use a transferred version due to Dodos and Kanellopoulos \cite{MR4436255} of the inverse theorem for Gowers norms of Green, Tao and Ziegler \cite{MR2950773}. We have to be particularly careful here with regard to the dependence of error terms on the function $w(\cdot)$, as the choice of this function itself has to depend on the constant $\epsilon$ hidden in the $o$-notation in Theorem \ref{thm:main_artin}. For this reason, we have decided to write out the full proof carefully, even though it is essentially the same as in \cite{MR2680398}.

The deductions of Proposition \ref{prop:gowers_estimate_cheb} and Theorem
\ref{thm:main_cheb} from Theorem
\ref{thm:W_tricked_chebotarev_mangoldt_with_nilsequences_nonconst_avg} follow
the same steps, but is simpler, as there we can just take the same
$w(N)=\log\log\log N$ as in \cite{MR2680398}. We only give a short summary in
\S\ref{sec:proof_main_cheb}, where we also prove Corollary \ref{cor:existence_cheb}.

\subsection{Examples}
The examples mentioned in the introduction are explained in \S\ref{sec:examples}. 

\section{Local densities}\label{sec:local_densities}

\subsection{Elementary facts about maximal abelian subextensions}

Here we provide some simple facts concerning maximal abelian subextensions $K^\ab/\QQ$ of Galois number fields $K/\QQ$ that we will require later on. We are not aiming for the highest generality in these statements.
\begin{lemma}\label{lem:max_ab_subext_compositum}
  Let $K/\QQ$ be a Galois number field and $q\in\NN$ with $K^\ab\subseteq \QQ(\mu_q)$. Then
  \begin{equation*}
   K(\mu_q)^\ab=\QQ(\mu_q).
\end{equation*}
\end{lemma}

\begin{proof}
  As $K\cap\QQ(\mu_q)=K^\ab$, the Galois group of $K(\mu_q)/\QQ$ is isomorphic to the fiber product
  \begin{equation*}
    G := \{(\sigma,\tau)\in \Gal(K/\QQ)\times\Gal(\QQ(\mu_q)/\QQ)\ :\ \sigma|_{K^\ab}=\tau|_{K^\ab}\}
  \end{equation*}
  via $\sigma\mapsto (\sigma|_K,\sigma|_{\QQ(\mu_q)})$. Hence, we need to show that the commutator subgroup $G'$ of $G$ is equal to the kernel $H$ of the projection $G\to \Gal(\QQ(\mu_q)/\QQ)$.

  Clearly, as $\Gal(\QQ(\mu_q)/K)$ is abelian, every commutator in $G$ satisfies
  \begin{equation*}
    [(\sigma_1,\tau_1),(\sigma_2,\tau_2)]=([\sigma_1,\sigma_2],[\tau_1,\tau_2])=([\sigma_1,\sigma_2],\ident)\in H,
  \end{equation*}
  and thus $G'\subseteq H$.

  On the other hand, every $(\sigma,\ident)\in H$ satisfies $\sigma|_{K^\ab}=\ident$ and thus $\sigma\in\Gal(K/\QQ)'$. Hence, we may write $\sigma=[\sigma_1,\rho_1][\sigma_2,\rho_2]\cdots[\sigma_k,\rho_k]$ with $\sigma_i,\rho_i\in \Gal(K/\QQ)$ and $k\in\NN$. As the projection $G\to\Gal(K/\QQ)$ is surjective, there are $\tau_i,\lambda_i\in\Gal(\QQ(\mu_q)/\QQ)$ with $(\sigma_i,\tau_i)\in G$ and $(\rho_i,\lambda_i)\in G$ for all $i$. Then, again as $\Gal(\QQ(\mu_q)/\QQ)$ is abelian,
  \begin{equation*}
    [(\sigma_1,\tau_1),(\rho_1,\lambda_1)]\cdots[(\sigma_k,\tau_k),(\rho_k,\lambda_k)]=([\sigma_1,\rho_1]\cdots[\sigma_k,\rho_k],[\tau_1,\lambda_1]\cdots[\tau_k,\lambda_k])=(\sigma,\ident),
  \end{equation*}
  showing that $(\sigma,\ident)\in G'$ and hence $H\subseteq G'$.  
\end{proof}

\begin{lemma}\label{lem:max_abelian_Gka}
  Let $a\in\QQ^\times$, let $k\in\NN$, and write
  $K=\QQ(\mu_k,\sqrt[k]{a})$. Then the maximal abelian subextension of $K/\QQ$ is
  \begin{equation*}
    K^{\ab}=
    \begin{cases}
      \QQ(\mu_k),&\text{ if }2\nmid k,\\
      \QQ(\mu_k,\sqrt{a}),&\text{ if }2\mid k.
    \end{cases}
  \end{equation*}
\end{lemma}

\begin{proof}
  In each of the two cases, let $F$ be the field claimed to coincide with $K^{\ab}$
  in the lemma. Then in both cases $F/\QQ$ is abelian and $F\subseteq K$, so
  $F\subseteq K^\ab$.

  We have $\QQ(\mu_k)\subseteq K^\ab\subseteq K$ and the extension
  $K/\QQ(\mu_k)$ is cyclic of degree dividing $k$. Hence, for some $l\mid k$ the extension
  $K^\ab/\QQ(\mu_k)$ is the unique degree-$l$-subextension of $K/\QQ(\mu_k)$,
  and $K^\ab=\QQ(\mu_k,\beta)$ with $\beta^l\in \QQ(\mu_k)$.

 Every automorphism
  $\sigma\in\Gal(K^\ab/\QQ)$ must satisfy $\sigma(\zeta_k)=\zeta_k^i$ for some
  $i\in\left(\ZZ/k\ZZ\right)^\times$ and $\sigma(\beta)=\zeta_k^{jk/l}\beta$
  for some $j\in\ZZ/l\ZZ$. As $|\Gal(K^\ab/\QQ)|=\phi(k)l$, each of these
  specifications is indeed realised by exactly one
  $\sigma_{i,j}\in\Gal(K^\ab/\QQ)$. Now, for any $i\in(\ZZ/k\ZZ)^\times$,
  \begin{align*}
    \sigma_{i,0}\sigma_{1,1}(\beta)&=\sigma_{i,0}(\zeta_k^{k/l}\beta)=\zeta_k^{ik/l}\beta,\\
    \sigma_{1,1}\sigma_{i,0}(\beta)&=\sigma_{1,1}(\beta)=\zeta_k^{k/l}\beta.
  \end{align*}
  Hence, these two elements of $\Gal(K^\ab/\QQ)$ do not commute unless $i\equiv
  1\bmod l$. As $\Gal(K^\ab/\QQ)$ is abelian, this shows that $l\in\{1,2\}$,
  and hence $K^\ab=F$.
\end{proof}

Recall that by $\Phi_{K^\ab}$ we denote the finite part of the conductor of the maximal abelian
subfield $K^\ab\subseteq K$, i.e. the smallest $f\in\NN$ such that
$K^\ab\subseteq \QQ(\zeta_f)$.
The above lemma shows in particular that  $\Phi_{G(k,a)^\ab}|\lcm(\Delta_{a},k)$ for the fields $G(k,a)$ defined in \eqref{eq:def_Gka}.

\subsection{Chebotarev classes}\label{sec:local_densities_cheb}

Recall that for $\vK=(K_1,\ldots,K_t)$, we defined
$\cD_\vK$ in \eqref{eq:def_DK}. Furthermore, we note that for every Galois number field $K$, conjugacy class $C\subseteq\Gal(K/\QQ)$, $q\in\NN$ and $b\in\ZZ$, it is obvious from \eqref{eq:cheb} that $\eta_{K,C}(b,q)\leq 1/\phi(q)$, and thus
\begin{equation}\label{eq:phi_eta_bound}
  \phi(q)\eta_{K,C}(b,q)\leq 1.
\end{equation}

\begin{lemma}[Almost-multiplicativity]\label{lem:almost_multiplicativity_cheb}
Let $q_1,q_2\in\NN$ with $\gcd(q_1,q_2)=1$ and $b\in\ZZ$. 
\begin{enumerate}
\item Let $K/\QQ$ be a finite Galois extension and
  $C\subseteq\Gal(K/\QQ)$ a conjugacy class.  Suppose that
  $\Phi_{K^{ab}}\mid q_1$. Then
    \begin{equation*}
      \eta_{K,C}(0,1)\eta_{K,C}(b,q_1q_2)=\eta_{K,C}(b,q_1)\eta_{K,C}(b,q_2).
    \end{equation*}
\item Let $\vK,\vC,\Psi$ be as in Theorem \ref{thm:main_cheb}. Suppose that $\cD_{\vK}\mid q_1$. Then
    \begin{equation*}
      \tau_{\vK,\vC,\Psi}(q_1q_2)=\tau_{\vK,\vC,\Psi}(q_1)\tau_{\vK,\vC,\Psi}(q_2).
    \end{equation*}
\end{enumerate}
\end{lemma}

\begin{proof}
  Let us start with \emph{(1)}. If $\gcd(b,q_1q_2)\neq 1$ then both sides of the equation are zero. Hence, we assume now that $\gcd(b,q_1q_2)=1$. Clearly, $\eta_{K,C}(0,1)=|C|/[K:\QQ]$. Considering ramified primes in $K^\ab$ and $\QQ(\mu_{q_2})$, we see that $K\cap\QQ(\mu_{q_2})=\QQ$, and thus $\eta_{K,C}(b,q_2)=|C|/[K(\mu_{q_2}):\QQ]$. Hence, the desired equality reduces to
  \begin{equation}\label{eq:almost_mult_cheb_1}
    \frac{\eta_{K,C}(b,q_1q_2)}{[K:\QQ]}=\frac{\eta_{K,C}(b,q_1)}{[K(\mu_{q_2}):\QQ]}.
  \end{equation}
  As $\Phi_{K^\ab}\mid q_1$, we get $K\cap\QQ(\mu_{q_1})=K^\ab=K\cap\QQ(\mu_{q_1q_2})$. In particular, the restrictions of the automorphisms $\zeta_{q_1}\mapsto\zeta_{q_1}^b\in\Gal(\QQ(\mu_{q_1})/\QQ)$ and $\zeta_{q_1q_2}\mapsto\zeta_{q_1q_2}^b\in\Gal(\QQ(\mu_{q_1q_2})/\QQ)$ to this field coincide, thus showing that
  \begin{equation} \label{eq:almost_mult_cheb_2}
  \eta_{K,C}(b,q_1q_2)[K(\mu_{q_1q_2}):\QQ] = \eta_{K,C}(b,q_1)[K(\mu_{q_1}):\QQ].   
\end{equation}
With Lemma \ref{lem:max_ab_subext_compositum} we get $K(\mu_{q_1})\cap \QQ(\mu_{q_2}) = K(\mu_{q_1})^\ab\cap\QQ(\mu_{q_2})= \QQ(\mu_{q_1})\cap\QQ(\mu_{q_2})=\QQ$, and thus
\begin{equation*}
 [K:\QQ][K(\mu_{q_1q_2}):\QQ]=[K:\QQ][K(\mu_{q_1}):\QQ][\QQ(\mu_{q_2}):\QQ]=[K(\mu_{q_1}):\QQ][K(\mu_{q_2}):\QQ].
\end{equation*}
Now \eqref{eq:almost_mult_cheb_1} follows from \eqref{eq:almost_mult_cheb_2} upon dividing both sides by this quantity.

If $\cD_\vK\mid q_1$, then we can apply \emph{(1)} for each $\eta_{K_i,C_i}(\cdot,\cdot)$ to see that
\begin{align*}
\tau_{\vK,\vC,\Psi}(q_1q_2)&=q_1^{t-s}q_2^{t-s}\sum_{n\in(\ZZ/q_1q_2\ZZ)}\prod_{i\in[t]}\frac{\eta_{K_i,C_i}(\psi_i(n),q_1)}{\eta_{K_i,C_i}(0,1)}\frac{\eta_{K_i,C_i}(\psi_i(n),q_2)}{\eta_{K_i,C_i}(0,1)}.
\end{align*}
Together with the Chinese remainder theorem, this shows \emph{(2)}.
\end{proof}

Recall the definition \eqref{eq:def_beta_p} of the local densities $\beta_{\Psi,p}$ from \cite{MR2680398}.

\begin{lemma}\label{lem:euler_factor_bound_cheb}
  Let $\vK,\vC,\Psi$ be as in Theorem \ref{thm:main_cheb}. Then, for all primes $p\nmid\cD_{\vK}$, we have
  \begin{equation*}
    \tau_{\vK,\vC,\Psi}(p)=\beta_{\Psi,p}.
  \end{equation*}
  In particular, for all primes $p$, we have
  \begin{equation*}
    \tau_{\vK,\vC,\Psi}(p)=1+O_{s,t,L,\vK}(1/p^2)
  \end{equation*}
  and the infinite product in Theorem \ref{thm:main_cheb} converges absolutely.
\end{lemma}

\begin{proof}
If $p\nmid \cD_{\vK}$, then  
$K_i\cap \QQ(\mu_p)=\QQ$ for all $1\leq i\leq t$. If $\gcd(b,q)=1$, we get 
\[
\frac{p\eta_{K_i,C_i}(b,p)}{\eta_{K_i,C_i}(0,1)}=p\frac{|C_i|}{[K_i(\mu_p):\QQ]}\frac{[K_i:\QQ]}{|C_i|}=\frac{p}{p-1}.
\]
If $\gcd(b,q)\neq 1$, then $\eta_{K_i,C_i}(b,q)=0$. Hence, for all $n\in(\ZZ/p\ZZ)^s$ we have
\begin{equation*}
  \frac{p\eta_{K_i,C_i}(\psi_i(n),p)}{\eta_{K_i,C_i}(0,1)}=\one_{\gcd(\psi(n),p)=1}\frac{p}{p-1},
\end{equation*}
and thus $\tau_{\vK,\vC,\Psi}(p)=\beta_{\Psi,p}$. The remaining assertions now follow from \cite[Lemma 1.3]{MR2680398}.
\end{proof}

\begin{lemma}\label{lem:equidistribution_cheb}
  Let $K/\QQ$ be a finite Galois extension and $C\subseteq\Gal(K/\QQ)$ a conjugacy class. 
  Let $Q,q\in\NN$ and $b\in\ZZ$, and assume that $\Phi_{K^\ab}\mid Q$. Then
  \begin{equation*}
    \eta_{K,C}(b,Qq) = \one_{\gcd(b,q)=1}\frac{\phi(Q)}{\phi(Qq)}\eta_{K,C}(b,Q).
  \end{equation*}
\end{lemma}

\begin{proof}
  We can assume that $\gcd(b,q)= 1$, as otherwise both sides of the
  equation are zero. From $\Phi_{K^\ab}\mid Q$ we get
  $K\cap\QQ(\mu_Q)=K^\ab=K\cap\QQ(\mu_{Qq})$. This shows that
  \begin{equation*}
  [K(\mu_Q):\QQ(\mu_Q)]=[K:K^\ab]=[K(\mu_{Qq}):\QQ(\mu_{Qq})].
\end{equation*}
  Moreover, as in the proof of Lemma \ref{lem:almost_multiplicativity_cheb}, it implies that
  \begin{equation*}
    \eta_{K,C}(b,Qq)=\eta_{K,C}(b,Q)\frac{[K(\mu_Q):\QQ]}{[K(\mu_{Qq}):\QQ]} =
    \eta_{K,C}(b,Q)\frac{\phi(Q)}{\phi(Qq)}.
  \end{equation*}
\end{proof}

\subsection{Artin primes}\label{sec:delta} 
Here we prove analogues of Lemma \ref{lem:almost_multiplicativity_cheb} and Lemma \ref{lem:euler_factor_bound_cheb} for the densities $\delta(a,b,q)$ and $\sigma_{\va,\Psi}(q)$ appearing in Theorem \ref{thm:main_artin}. We will deduce these analogues from closed formulae for $\delta(a,b,q)$, which we recall now. Write
\begin{align*}
h_a&:=\max\big\{m \in \NN:a \text{ is an } m\text{th power}\big\},\\
\cA_a&:=
\prod_{p\mid h_a} \left(1-\frac{1}{p-1}\right)
       \prod_{p\nmid h_a} \left(1-\frac{1}{p(p-1)}\right).
\end{align*}
The quantity $\cA_a$ results from a naive heuristic for the expected density of primes with prescribed primitive root $a$. It is not always equal to $\delta(a,0,1)$ due to dependencies between splitting conditions. More precisely, for a positive integer $m$, we let  
\begin{equation*}
f_a^\ddagger(m):=
\prod_{\substack{p\mid m,p \mid h_a}}\frac{1}{p-2}
\prod_{\substack{p\mid m,p\nmid h_a}}\frac{1}{p^2-p-1}.
\end{equation*}
Then Hooley \cite{MR207630} has shown that
\begin{equation}\label{eq:hooley_density}
\delta(a,0,1)=
\cA_a
\cdot
\big(
1+
\mu(2|\Delta_a|)
f_a^\ddagger(|\Delta_a|)
\big)
.\end{equation}
For the general formula due to  Moree \cite{MR2490093}, we require some more notation. Let
\begin{align}\label{eq:beta_q_def}
\beta_a(q)
&:=
\begin{cases}  
(-1)^{\frac{\frac{\Delta_a}{\gcd(q,\Delta_a)}-1}{2}}\gcd(q,\Delta_a),
&\mbox{if } \frac{\Delta_a}{\gcd(q,\Delta_a)} \text{ is odd, }\\ 
1 & \mbox{otherwise, }
\end{cases}\\
  \nonumber
  f_a^\dagger(q)&:=
\prod_{\substack{p\mid h_a,p \mid q }} \left(1-\frac{1}{p-1}\right)^{-1}
\prod_{\substack{p\nmid h_a,p \mid q }} \left(1-\frac{1}{p(p-1)}\right)^{-1},
\end{align}
and 
\begin{equation} \label{eq:cA_a(b,q)}
\cA_a(b,q):=
\cA_a \cdot
\begin{cases}  
\frac{f_a^\dagger(q)}{\phi(q)} \prod_{p|b-1, p| q}
\Big(1-\frac{1}{p}\Big),&\mbox{if } \gcd(b-1,q,h_a)=\gcd(b,q)=1,\\  
0,&\mbox{otherwise. } \end{cases}
\end{equation}
Then
 
\begin{equation}\label{eq:delta_closed_formula}
\delta(a,b,q) 
=
\cA_a(b,q)
\Bigg(
1+\mu\left(\frac{2|\Delta_a|}{\gcd(q,\Delta_a)}\right)\left(\frac{\beta_a(q)}{b}\right)f_a^\ddagger\left(\frac{|\Delta_a|}{\gcd(q,\Delta_a)}\right) 
\Bigg)
.
\end{equation}
One can deduce from these formulas that
\begin{equation}\label{eq:delta_phi_bound}
  \delta(a,b,q)\leq \frac{1}{\phi(q)},
\end{equation}
but of course (under $\HRH(a)$) this is also immediate from \eqref{eq:artin_prime_asymptotic} and the prime
number theorem in arithmetic progressions. Now we can prove our analogue of Lemma \ref{lem:almost_multiplicativity_cheb}.

\begin{lemma}[Almost-multiplicativity]\label{lem:almost_multiplicativity}
  Let $q_1,q_2\in\NN$ with $\gcd(q_1,q_2)=1$ and $b\in\ZZ$. 
  \begin{enumerate}
  \item Let $a\in\ZZ$ be not equal to $-1$ or a perfect square, and suppose that $|\Delta_a|\mid q_1$. Then
    \begin{equation*}
      \delta(a,0,1)\delta(a,b,q_1q_2)=\delta(a,b,q_1)\delta(a,b,q_2).
    \end{equation*}
  \item Let $\va,\Psi$ be as in Theorem \ref{thm:main_artin} and suppose that $\cD_{\va}\mid q_1$. Then
  \begin{equation*}
      \sigma_{\va,\Psi}(q_1q_2)=\sigma_{\va,\Psi}(q_1)\sigma_{\va,\Psi}(q_2).
    \end{equation*}
  \end{enumerate}
\end{lemma}

\begin{proof}
Let us start with \emph{(1)}. In \eqref{eq:delta_closed_formula}, 
we see from our hypotheses on $q_1,q_2$ that
\begin{align*}
\mu\left(\frac{2|\Delta_a|}{\gcd(q_1,\Delta_a)}\right)&=\mu\left(\frac{2|\Delta_a|}{\gcd(q_1q_2,\Delta_a)}\right)=\mu(2)=-1,\\
\mu\left(\frac{2|\Delta_a|}{\gcd(1,\Delta_a)}\right)&=\mu\left(\frac{2|\Delta_a|}{\gcd(q_2,\Delta_a)}\right)=\mu(2|\Delta_a|).
\end{align*}
Moreover, we have $\beta_a(q_1)=\beta_a(q_1q_2)=\Delta_a$ and $\beta_a(1)=\beta_a(q_2)=1$, as well as
\begin{align*}
f_a^\ddagger\left(\frac{|\Delta_a|}{\gcd(q_1,\Delta_a)}\right)&=f_a^\ddagger\left(\frac{|\Delta_a|}{\gcd(q_1q_2,\Delta_a)}\right)=f_a^\ddagger(1),\\
f_a^\ddagger\left(\frac{|\Delta_a|}{\gcd(1,\Delta_a)}\right)&=f_a^\ddagger\left(\frac{|\Delta_a|}{\gcd(q_2,\Delta_a)}\right)=f_a^\ddagger(|\Delta_a|).  
\end{align*}
Therefore, it suffices to study the values of $\cA_a(b,q)$ defined in (\ref{eq:cA_a(b,q)}). We can immediately conclude that $\cA_a(0,1)=\cA_a$. Now clearly $\gcd(b-1,q_1q_2,h_a)\neq 1$ if and only if $\gcd(b-1,q_1,h_a)\neq 1$ or $\gcd(b-1,q_2,h_a)\neq 1$, and similarly for $\gcd(b,q_1q_2)$. Therefore, $\cA_a(b,q_1q_2)=0$ if and only if at least one of $\cA_a(b,q_1)$ and $\cA_a(b,q_2)$ is zero, in which case both sides of our equality are zero. We are left with the case where $\gcd(b-1,q_1q_2,h_a)=\gcd(b,q_1q_2)=1$, and then we have
\begin{equation*}
  \cA_a(0,1)\cA_a(b,q_1q_2)=\cA_a^2\frac{f_a^\dagger(q_1q_2)}{\phi(q_1q_2)}\prod_{p\mid b-1, p\mid q_1q_2}\left(1-\frac{1}{p}\right) = \cA_a(b,q_1)\cA_a(b,q_2).
\end{equation*}

Part \emph{(2)} follows from \emph{(1)} and the Chinese remainder theorem, as in the proof of Lemma \ref{lem:almost_multiplicativity_cheb}.
\end{proof}

Recall that $\vecnorm{\Psi}_N\leq L$ in Theorem \ref{thm:main_artin}, so in particular the linear coefficients of all $\psi_i$ are bounded by $L$. The following lemma is the analog of Lemma \ref{lem:euler_factor_bound_cheb} for the densities $\sigma_{\va,\Psi}(p)$, though it yields less precise information.

\begin{lemma}\label{lem:euler_factor_bounds}
  Let $\va,\Psi$ be as in Theorem \ref{thm:main_artin}. 
  For all primes $p$, 
  we have
  \begin{equation*}
    \sigma_{\va,\Psi}(p)=1+O_{s,t,L,\va}(1/p^2).
  \end{equation*}
\end{lemma}

\begin{proof}
  Even though the execution is slightly more complicated, the basic idea of this proof is the same as in \cite[Lemma 1.3]{MR2680398}.
  We may assume $p$ to be sufficiently large so that
  $p\nmid h_{a_1}\cdots h_{a_t}\cD_\va$
  and such that no form is constant and no two of the linear parts of the forms $\psi_i$ are linearly dependent over $\FF_p$.
  Recall from \eqref{eq:def_sigma} that
\[
\sigma_{\va,\Psi}(p)=\EE_{n\in(\ZZ/p\ZZ)^s}\prod_{i\in[t]}\frac{p\delta(a_i,\psi_i(n),p)}{\delta(a_i,0,1)}.
\]
One checks from  \eqref{eq:delta_closed_formula} and $p\nmid\Delta_{a_i}$ that
\[
\frac{p\delta(a_i,\psi_i(n),p)}{\delta(a_i,0,1)}=\frac{p\cA_{a_i}(\psi_i(n),p)}{\cA_{a_i}}.
\] 
Let us now discuss the value of this expression. As $p\nmid h_{a_i}$, we obtain the following three cases:
\[
\frac{\cA_{a_i}(\psi_i(n),p)}{\cA_{a_i}}=\left\{\begin{array}{ll}
0, & \text{ if }\psi_i(n)\equiv 0\bmod p,\\
\frac{1}{p-1}\left(1-\frac{1}{p(p-1)}\right)^{-1}, & \text{ if } \psi_i(n)\not\equiv 0,1\bmod p,\\
\frac{1}{p-1}\left(1-\frac{1}{p}\right)\left(1-\frac{1}{p(p-1)}\right)^{-1}, &
                                                                               \text{ if } \psi_i(n)\equiv 1\bmod p.
                                                \end{array}\right.
\] 
Therefore, we get
\begin{equation}\label{eq:sumk}
\sigma_{\va,\Psi}(p)=p^{-s}\sum_{l=0}^t\left(\frac{p}{p-1}\right)^t\left(1-\frac{1}{p(p-1)}\right)^{-t}\left(1-\frac{1}{p}\right)^l|S_l|,
\end{equation}
where 
\[
S_l:=\{n\in(\ZZ/p\ZZ)^s\where 
 \psi_i(n)\not\equiv 0\bmod p\ \text{ for all  $i$ and } \ \psi_i(n)\equiv 1\bmod p \text{ for exactly }l\text{ indices } i  \}.  
\]
Let us now discuss each summand of the above sum. We will start with $l=0$. Let
\[
A_i:=\{n\in(\ZZ/p\ZZ)^s\where \psi_i(n)\equiv 0,1\bmod p\}.
\]
As no two of the linear parts of the forms $\psi_i$ are linearly dependent over $\FF_p$, we get
\begin{equation}\label{eq:sigma_euler_sets}
  |A_i|=2p^{s-1}\quad\text{ and }\quad|A_i\cap A_j|=4p^{s-2} \text{ for all $i\neq j$}.
\end{equation}
Using the Bonferroni inequalities (i.e. truncated inclusion-exclusion), we obtain
\[
|S_0|\geq |(\ZZ/p\ZZ)^s|-\sum_{i=1}^t|A_i|=p^s-2tp^{s-1}
\] 
and
\[
|S_0|\leq |(\ZZ/p\ZZ)^s|-\sum_{i=1}^t|A_i|+\sum_{1\leq i<j\leq t}|A_i\cap A_j|=p^s-2tp^{s-1}+4\binom{t}{2}p^{s-2},
\]
which gives $|S_0|=p^s-2tp^{s-1}+O_t(p^{s-2})$. With
\[
\left(\frac{p}{p-1}\right)^t\left(1-\frac{1}{p(p-1)}\right)^{-t}=\left(1+\frac{1}{p}+O(p^{-2})\right)^t\left(1+O(p^{-2})\right)^{-t}=1+\frac{t}{p}+O_t(p^{-2}),
\]
we can therefore compute the contribution from $l=0$ to $\sigma_{\va,\Psi}(p)$ in \eqref{eq:sumk} as
\begin{align*}
                                                                                  &p^{-s}\left(1+\frac{t}{p}+O_t(p^{-2})\right)(p^s-2tp^{s-1}+O_t(p^{s-2}))=1-\frac{t}{p}+O_t(p^{-2}).
\end{align*}

We proceed with the contribution of $l=1$. For that, we have
\[
S_1=\bigcup_{i=1}^t\tilde{A}_i\setminus\bigcup_{1\leq i<j\leq t}(A_i\cap A_j),
\]
where
\[
\tilde{A}_i:=\{n\in(\ZZ/p\ZZ)^s\where\psi_i(n)\equiv 1\bmod p)\}
\]
satisfies $|\tilde{A}_i|=p^{s-1}$ and $|\tilde{A_i}\cap\tilde{A_j}|=p^{s-2}$ for all $i\neq j$. From \eqref{eq:sigma_euler_sets}, we know that $|A_i\cap A_j|=O(p^{s-2})$, and therefore inclusion-exclusion yields
\[
|S_1|=\sum_{i=1}^t\tilde{A}_i+O_t(p^{s-2})=tp^{s-1}+O_t(p^{s-2}).
\]
Similarly as before, we see for all $l\in [t]$ that
\begin{equation}\label{eq:sigma_euler_factors}
\left(\frac{p}{p-1}\right)^t\left(1-\frac{1}{p(p-1)}\right)^{-t}\left(1-\frac{1}{p}\right)^l=\left(1+O_t(p^{-1})\right)\left(1+O_l(p^{-1})\right)=1+O_t(p^{-1}).
\end{equation}
Therefore, the contribution of $l=1$ to $\sigma_{\va,\Psi}(p)$ in \eqref{eq:sumk} is
\begin{align*}
  &p^{-s}\left(1+O_t(p^{-1})\right)(tp^{s-1}+O_t(p^{s-2}))=\frac{t}{p}+O_t(p^{-2}).
\end{align*}

Thus, the total contribution of $l=0,1$ in \eqref{eq:sumk} is $1+O_t(p^{-2})$. Moreover, it is clear that $|S_l|=O_t(p^{-2})$ for all $l\geq 2$, which together with \eqref{eq:sigma_euler_factors} shows that the remaining summands have negligible contribution. 
\end{proof}

\section{Equidistributed  nilsequences}\label{sec:equidistributed_nilsequences}
Here we collect a few facts about equidistribution of nilsequences and their polynomial subsequences, which are
probably well known. However, we were not able to find references in exactly the form we need, so we prove the required results here. 
We start with simple observations about composites of binomial polynomials with integer polynomials.

\begin{lemma}\label{lem:poly_subsequence_coefficients}
  Let $j,D\geq 0$. In the polynomial ring $\QQ[a_0,\ldots,a_D,x]$ in $D+2$ variables, write 
  \begin{equation}\label{eq:binom_poly_comp}
    \binom{a_0+a_1x+\cdots+a_Dx^D}{j}=\sum_{k=0}^{jD}s_{j,k}\binom{x}{k},
  \end{equation}
  with polynomials $s_{j,k}\in\QQ[a_0,\ldots,a_D]$.
  \begin{enumerate}
  \item We have $s_{j,jD}= \frac{(jD)!}{j!}a_D^j$.
  \item We have $s_{j,k}(\ZZ^D)\subseteq \ZZ$ for all $k$.  
  \item If we assign each variable $a_i$ the degree $\deg a_i:=1$, then $\deg s_{j,k}\leq j$ for all $k$.
  \item If we assign each variable $a_i$
    the degree $\deg a_i := i$, then $s_{j,k}$ contains no terms of
    total degree smaller than $k$.
\end{enumerate}
\end{lemma}

\begin{proof}
  Comparing the coefficients of $x^{jD}$ on both sides of \eqref{eq:binom_poly_comp}, we see that $a_D^j/j! = s_{j,jD}/(jD)!$, which shows all four assertions of the lemma in case $k=jD$. Assertion \emph{(2)} is clear, as for $a_0,\ldots,a_D\in\ZZ$ the left-hand side of \eqref{eq:binom_poly_comp} is an integer-valued polynomial, so it has integral coefficients in the binomial basis $\binom{x}{k}$.

  For arbitrary $k$, the coefficient of $x^{k}$ on the left-hand side of \eqref{eq:binom_poly_comp} consists of terms of the form
  \begin{equation*}
    t=c a_0^{i_0}\cdots a_D^{i_D}
  \end{equation*}
  with $c\in\QQ$ and $i_0,\ldots,i_D\geq 0$ satisfying $i_0+\cdots+i_D\leq j$ and $i_1+2i_2+\cdots+D i_D=k$. The coefficient of $x^k$ on the right-hand side has the form
  \begin{equation*}
    \frac{1}{k!}s_{j,k}+c_{k+1}s_{j,k+1}+\cdots+c_{D}s_{j,D},
  \end{equation*}
  where $c_i\in\QQ$ is the coefficient of $x^{k}$ in $\binom{x}{i}$. Hence, if $0\leq k<jD$ and assertions \emph{(3)} and \emph{(4)} are true for $s_{j,k+1},\ldots,s_{j,jD}$, then they follow for $s_{j,k}$ by comparing the coefficients of $x^k$ on both sides described above.  
\end{proof}

The following lemma can be thought of as a polynomial version of \cite[Lemma 7.10]{MR2877065}. However, even in the case $D=1$, our assumptions on the coefficients are slightly different. Recall the definition of the smoothness norms $\vecnorm{p}_{\Cinf[N]}$ of polynomial sequences $p:\ZZ\to\RR/\ZZ$ given in \cite[Definition 2.7]{MR2877065}:
writing $p$ in the binomial basis as
\begin{equation}\label{eq:p_binomial_basis}
  p(n)=\sum_{j=0}^d\alpha_j\binom{n}{j}\quad \text{ with }\quad \alpha_j\in\RR,
\end{equation}
we let $\vecnorm{\alpha_j}_{\RmZ}$ denote the distance of $\alpha_j$ to the nearest integer and 
\begin{equation*}
  \vecnorm{p}_{\Cinf[N]}:=\sup_{1\leq j\leq d}N^j\vecnorm{\alpha_j}_{\RmZ}.
\end{equation*}

\begin{lemma}[Polynomial extrapolation]\label{lem:polynomial_extrapolation}
  Let $d,D,Q,N\in\NN$. Let $P=\sum_{i=0}^{D} a_ix^i\in\ZZ[x]$ with coefficients satisfying $|a_i|\leq QN^{1-i/D}$ and $a_{D}\neq 0$. Set $N':=\lfloor N^{1/D}\rfloor$. Let $p:\ZZ\to\RR/\ZZ$ be a polynomial sequence of degree $d$ and write $\tilde p = p\circ P$. Then
  \begin{equation}\label{eq:poly_extrapolation_1}
    \vecnorm{\tilde p}_{\Cinf[N']}\ll_{d,D} Q^{d}\vecnorm{p}_{\Cinf[N]},
  \end{equation}
and there is $q\in\NN$ with $q\ll_{d,D}Q^{O_{d,D}(1)}$, such that
  \begin{equation}\label{eq:poly_extrapolation_2}
    \vecnorm{qp}_{\Cinf[N]} \ll_{d,D} Q^{O_{d,D}(1)}\vecnorm{\tilde p}_{\Cinf[N']}.
  \end{equation}
\end{lemma}

\begin{proof}
  Write $p$ in the binomial basis as in \eqref{eq:p_binomial_basis}.
  Then, with the coefficients $s_{j,k}=s_{j,k}(a_0,\ldots,a_D)\in\ZZ$ as in Lemma \ref{lem:poly_subsequence_coefficients}, we have
  \begin{equation*}
    \tilde p(n)=\sum_{j=0}^d\alpha_j\binom{P(n)}{j}=\sum_{k=0}^{dD}\left(\sum_{j=\lceil k/D\rceil}^ds_{j,k}\alpha_j\right)\binom{n}{k}=\sum_{k=0}^{dD}\beta_k\binom{n}{k},
  \end{equation*}
  with
  \begin{equation}\label{eq:beta_k_formula}
  \beta_k=\sum_{j=\lceil k/D\rceil}^ds_{j,k}\alpha_j\in\RR.
\end{equation}
  By our hypothesis, any $t=a_0^{i_0}\cdots a_D^{i_D}$ satisfies $|t| = |a_0^{i_0}\cdots a_D^{i_D}|\leq (QN)^{i_0+\cdots+i_D}N^{\frac{-i_1-2i_2-\cdots-Di_D}{D}}$. Hence, by Lemma \ref{lem:poly_subsequence_coefficients}, we obtain the estimate
  \begin{equation}\label{eq:sjk_bound}
    |s_{j,k}|\ll_{d,D} Q^{j} N^{j-k/D}.
  \end{equation}
  Hence,
  \begin{equation*}
    (N')^k\vecnorm{\beta_k}_{\RmZ}\ll_{d,D}\sum_{j=\lceil k/D\rceil}^dQ^{j}N^j\vecnorm{\alpha_j}_{\RmZ}\ll_d Q^d\vecnorm{p}_{\Cinf[N]},
  \end{equation*}
  which shows \eqref{eq:poly_extrapolation_1}. For \eqref{eq:poly_extrapolation_2}, we start by observing that $s_{j,jD}\in\ZZ\smallsetminus\{0\}$ for all $1\leq j\leq d$ by Lemma \ref{lem:poly_subsequence_coefficients} and our hypothesis $a_D\neq 0$. Take $q_i:=\prod_{j=i}^ds_{j,jD}$, then  $q_i\ll_{d,D}Q^{O_{d,D}(1)}$ by \eqref{eq:sjk_bound}. We show that
  \begin{equation}\label{eq:a_j_desired_bound}
    N^i\vecnorm{q_i\alpha_i}_{\RmZ}\ll_{d,D} Q^{O_{d,D}(1)}\vecnorm{\tilde p}_{\Cinf[N']}
  \end{equation}
  holds for all $1\leq i\leq d$, which is enough to prove \eqref{eq:poly_extrapolation_2} with $q=q_1$.

  Let $1\leq i\leq d$ and assume that we have shown \eqref{eq:a_j_desired_bound} already for all $i<j\leq d$. Using \eqref{eq:beta_k_formula} and \eqref{eq:sjk_bound} with $k=iD$, we see that
  \begin{align*}
    N^i\vecnorm{q_i\alpha_i}_{\RmZ}&\leq N^iq_{i+1}\vecnorm{\beta_{iD}}_{\RmZ} + \sum_{j=i+1}^dN^i|s_{j,iD}|\vecnorm{q_{i+1}\alpha_j}_{\RmZ}\\
    &\ll_{d,D} Q^{O_{d,D}(1)}\left(N^i(N')^{-iD}\vecnorm{\tilde p}_{\Cinf[N']}+\sum_{j=i+1}^dN^{i-j}|s_{j,iD}|\vecnorm{\tilde p}_{\Cinf[N']}\right)\\ &\ll Q^{Q_{d,D}(1)}\vecnorm{\tilde p}_{\Cinf[N']},
  \end{align*}
as desired.
\end{proof}

Next, we prove a simple weak converse to the quantitative Leibman theorem
\cite[Theorem 2.9]{MR2877065}. That such a converse holds was first shown by Matthiesen in \cite[Proposition 14.3]{MR2949879}, our statement here is a minor variation of this.

\begin{lemma}[Horizontal characters obstruct equidistribution]\label{lem:characters_obstruct}
  Let $m,d\geq 0$, then there is a constant $C(m,d)\geq 1$, such that the following
  holds. Let $0<\delta<1/2$ and $N\gg_{m,d} 1$. 
  
  Suppose that $G/\Gamma$ is an $m$-dimensional nilmanifold with a filtration
  $G_\bullet$ and a $\delta^{-1}$-rational Mal'cev basis $\cX$ adapted to this
  filtration. Let $g\in\poly(\ZZ,G_\bullet)$, and suppose that there is a
  nontrivial horizontal
  character $\eta:G/\Gamma\to \RR/\ZZ$ with $|\eta|\leq \delta^{-1}$, such that
  \begin{equation}\label{eq:hor_char_norm_bound}
    \vecnorm{\eta\circ g}_{\Cinf[N]}\leq \delta^{-1}.
  \end{equation}
  Then $(g(n)\Gamma)_{n\in[N]}$ is not totally $\delta^{C}$-equidistributed.
\end{lemma}

\begin{proof}
  This is inspired by arguments in the proofs of \cite[Lemma A.17]{MR2877065} and \cite[Proposition
  2.1]{MR2877066}. By \cite[Lemma 2.8]{MR2877065}, our assumption
  \eqref{eq:hor_char_norm_bound} implies that
  \begin{equation*}
    \vecnorm{\eta(g(n))-\eta(g(n-1))}_{\RmZ}\ll_d(\delta N)^{-1}\quad\text{ for
    all }n\in \{2,\ldots,N\}.
  \end{equation*}
  Let $N':=\lceil \delta^{C}N\rceil$ with large enough $C$, then this shows that the
  values $\eta(g(n))$, $n\in N'$ lie in an interval $I$ of length $\leq 1/2$ on the
  torus. Let $H:\RR/\ZZ\to [-1,1]$ be a function of Lipschitz norm
  $\vecnorm{H}_{\Lip}\ll 1$ and mean zero, such that $H=1$ on $I$. Let
  $F:=H\circ\eta$, then
  \begin{equation*}
    \vecnorm{F}_{\Lip} \leq 1 + \sup_{\substack{x,y\in
        G/\Gamma\\\eta(x)\neq\eta(y)}}\frac{|H(\eta(x))-H(\eta(y))|}{\vecnorm{\eta(x)-\eta(y)}_{\RmZ}}\frac{\vecnorm{\eta(x)-\eta(y)}_{\RmZ}}{d_{G/\Gamma}(x,y)}\ll_d1+\sup_{\substack{x,y\in
        G/\Gamma\\\eta(x)\neq\eta(y)}}\frac{\vecnorm{\eta(x)-\eta(y)}_{\RmZ}}{d_{G/\Gamma}(x,y)}.
  \end{equation*}
  Let $x,y$ be as in the supremum above. Let $\psi:G\to\RR^m$ denote the coordinates (of the second kind) with respect to our Mal'cev basis $\cX$.
  By \cite[Lemma A.14]{MR2877065}, there are $x'\in x$, $y'\in y$
  with $\psi(x'),\psi(y')\in[0,1)^m$. By the proof of \cite[Lemma
  A.15]{MR2877065}, there is $\gamma\in \Gamma$, such that
  \begin{equation*}
    d_{G/\Gamma}(x,y)=d_{G/\Gamma}(x'\Gamma,y'\Gamma)=d(x',y'\gamma).
  \end{equation*}
  By \cite[Lemma A.16]{MR2877065}, this quantity is $\ll_{m,d} \delta^{-O_{m,d}(1)}$.
  Using \cite[Lemma A.4]{MR2877065}, we also have $d(x',\ident_{G})\ll_{m,d}
  \delta^{-O_{m,d}(1)}$, and by the triangle inequality
  $d(y',\ident_{G})\ll_{m,d}\delta^{-O_{m,d}(1)}$. Hence, \cite[Lemma
  A.4]{MR2877065} shows that
  \begin{equation*}
   |\psi(x')-\psi(y'\gamma)| \ll_{m,d}\delta^{-O_{m,d}(1)}d(x',y'\gamma).
 \end{equation*}
 We have $\eta(x)=k\cdot\psi(x)$ for some $k\in\ZZ^m$ with
 $|k|=|\eta|\leq\delta^{-1}$. Hence, as $\eta$ is a homomorphism that
 annihilates $\Gamma$,
 \begin{equation*}
   \vecnorm{\eta(x)-\eta(y)}_{\RR/\ZZ}\leq |\eta(x')-\eta(y'\gamma)| \ll\delta^{-1}|\psi(x')-\psi(y'\gamma)|\ll_{m,d}\delta^{-O_{m,d}(1)}d(x',y'\gamma)=\delta^{-O_{m,d}(1)}d(x,y),
 \end{equation*}
 thus showing that $\vecnorm{F}_{\Lip}\ll_{m,d}\delta^{-O_{m,d}(1)}$.

 As $\eta$ is a surjective continuous homomorphism, $\int_{\RR/\ZZ}H = 0$ implies
 $\int_{G/\Gamma}F=0$. Hence, we obtain
 \begin{equation*}
   \left|\EE_{n\in [N']}F(g(n)\Gamma)-\int_{G/\Gamma}F\right| = \left|\EE_{n\in [N']}H(\eta(g(n)))\right|=1>\delta^{C}\vecnorm{F}_{\Lip},
 \end{equation*}
 if only $C$ is large enough in terms of $m,d$. This shows that
 $(g(n)\Gamma)_{n\in[N]}$ is not totally $\delta^{C}$-equidistributed, as desired. 
\end{proof}

The following proposition states, essentially, that polynomial subsequences of equidistributed polynomial nilsequences are again equidistributed polynomial nilsequences. The first results concerning equidistribution of polynomial subsequences are due to Matthiesen \cite[Proposition 15.3 and Proposition 15.4]{MR2949879}. Our version is a minor technical variation of \cite[Proposition 15.3]{MR2949879}, based on the same ideas.

\begin{proposition}[Equidistribution of polynomial subsequences]\label{prop:equidist_poly_subseq}
  Let $d,D,m\in\NN$. Let $G/\Gamma$ be an $m$-dimensional nilmanifold
  together with a filtration $G_\bullet$ of degree $d$ and a rational
  Mal'cev basis $\cX$ adapted to this filtration. Let
  $g\in\poly(\ZZ,G_\bullet)$ be a polynomial sequence, and let
  $P=\sum_{i=0}^Da_ix^i\in\ZZ[x]$ be a polynomial of degree $D$. Then
  \begin{enumerate}
  \item The function $\tilde g := g\circ P:\ZZ\to G$ is a polynomial sequence of degree $dD$. More precisely, $\tilde g\in\poly(\ZZ,\tilde G_\bullet)$, with the filtration $G=\tilde G_0=\tilde G_1\geq\cdots\geq \tilde G_{dD+1}=\{0\}$, where $\tilde G_i := G_{\lceil i/D\rceil}$. Moreover, the Mal'cev basis $\cX$ is also adapted to the filtration $\tilde G_\bullet$.
  \item There is a constant $c=c(m,d,D)\in (0,1)$, depending only on $m,d$ and $D$, such that the following holds: Let $N\in\NN$ such that $N\gg_{m,d,D} 1$ and set $\tilde N = \lfloor N^{1/D}\rfloor$.
    Assume that $\delta\in (N^{-1/2D},2^{-1/c})$, that
    \begin{equation*}
    |a_i|\leq \delta^{-c}N^{1-i/D}\quad\text{ for all }0\leq i\leq D,
  \end{equation*}
  and that the Mal'cev basis $\cX$ is $\delta^{-c}$-rational. If $(g(n)\Gamma)_{n\in[N]}$ is totally $\delta$-equidistributed, then the sequence $(\tilde g(n)\Gamma)_{n\in[\tilde N]}$ is totally $\delta^c$-equidistributed.
  \end{enumerate}
\end{proposition}

\begin{proof}
 For $i\in\NN$, write $\tilde m_i := \dim \tilde G_i = m_{\lceil i/D\rceil}$. With $H_j$ as in \cite[Definition 2.1]{MR2877065}, we have $\tilde G_i = G_{\lceil i/D\rceil} = H_{m-\tilde m_i}$, hence $\cX$ is also a Mal'cev basis adapted to the filtration $\tilde G_\bullet$.
  
 Let $\psi:G\to\RR^m$ denote the coordinates (of the second kind) with respect to our Mal'cev basis $\cX$. By \cite[Lemma 6.7]{MR2877065}, we have
  \begin{equation*}
    \psi(g(n)) = \sum_{j=0}^d\vt_j\binom{n}{j},
  \end{equation*}
  with vectors $\vt_j\in\RR^m$ satisfying $(\vt_j)_i=0$ for all $i\leq m-m_j$,
  where $m_j=\dim G_j$. Then, using the coefficients $s_{j,k}(a_0,\ldots,a_D)$ from Lemma \ref{lem:poly_subsequence_coefficients}, 
  \begin{equation*}
    \psi(g(P(n)))=\sum_{j=0}^d\vt_j\binom{P(n)}{j}=\sum_{j=0}^d\vt_j\sum_{k=0}^{jD}s_{j,k}\binom{n}{k}  =\sum_{k=0}^{dD}\left(\sum_{j=\lceil k/D\rceil}^{d}s_{j,k}\vt_j\right)\binom{n}{k}=:\sum_{k=0}^{dD}\tilde\vt_k\binom{n}{k}.
  \end{equation*}
  For any $j\geq \lceil k/D\rceil$, we have $m_j\leq m_{\lceil k/D\rceil}=\tilde m_k$. Hence, if $i\leq m-\tilde m_k$, then also $i\leq m-m_j$ for all $j\geq \lceil k/D\rceil$, and thus $(\tilde\vt_k)_i=0$. Again by \cite[Lemma 6.7]{MR2877065}, this shows that $p\circ P\in \poly(\ZZ,\tilde G_\bullet)$, as desired in \emph{(1)}.

  For \emph{(2)}, suppose that the sequence $(\tilde g(n)\Gamma)_{n\in[\tilde N]}$ is not totally $\delta^c$-equidistributed. Then there is a progression $P=\{u+vn\ :\ n\in [\bar N]\}\subseteq [\tilde N]$ with $\bar N\geq \delta^c\tilde N$ and a Lipschitz function $F:G/\Gamma\to \CC$, such that
  \begin{equation*}
    \left|\EE_{n\in P}F(\tilde g(n)\Gamma)-\int_{G/\Gamma}F\right|\leq \delta^c\vecnorm{F}_{\Lip},
  \end{equation*}
  i.e. the sequence $(\bar g(n)\Gamma)_{n\in [\bar N]}$ with $\bar g(n)=\tilde g(u+vn) = g(P(u+vn))$ is not $\delta^c$-equidistributed.
  The quantitative Leibmann theorem \cite[Theorem 2.9]{MR2877065} yields a
  nontrivial horizontal character $\bar\eta:G\to \RR/\ZZ$ with $0<|\bar\eta|\ll_{m,d}\delta^{-O_{m,d}(c)}$, such that
  \begin{equation*}
    \vecnorm{\bar\eta\circ\bar g}_{\Cinf[\bar N]}\ll_{m,d}\delta^{-O_{m,d}(c)}.
  \end{equation*}
  As $\delta>\tilde N^{-1/2}$, we get $\delta^c>\tilde N^{-c/2}>\tilde N^{-1/2}\geq 2\tilde N^{-1}$. Hence,
  \begin{equation*}
    1\leq v\leq \frac{\tilde N-1}{\bar N-1}\leq \frac{2\tilde N}{\delta^c\tilde N}=2\delta^{-c}\leq \tilde N
  \end{equation*}
  and thus also $|u|\leq\tilde N$. By Lemma \ref{lem:polynomial_extrapolation} with $D=1$, $Q=2\delta^{-c}$, $N=\tilde N$, $p=\bar\eta\circ\tilde g$, $P(X)=u+vx$, we find $\bar q\in \NN$ with $\bar q\ll_{d}\delta^{-O_{d}(c)}$ such that
  \begin{equation*}
\vecnorm{\bar q\bar\eta\circ\tilde g}_{\Cinf[\tilde N]}\ll_{d}\delta^{-O_{d}(c)}\vecnorm{\bar\eta\circ\bar g}_{\Cinf[\tilde N]}\ll_{m,d}\delta^{-O_{m,d}(c)}.
\end{equation*}
Another application of Lemma \ref{lem:polynomial_extrapolation}, this time with $Q=\delta^{-c}$ and $p=\bar q\bar\eta\circ g$ yields $q\in\NN$ with $q\ll_{d,D}\delta^{-O_{d,D}(c)}$, such that
  \begin{equation*}
\vecnorm{q\bar q\bar\eta\circ g}_{\Cinf[N]}\ll_{d,D}\delta^{-O_{d,D}(c)}\vecnorm{\bar q\bar\eta\circ\tilde g}_{\Cinf[\tilde N]}\ll_{m,d,D}\delta^{-O_{m,d,D}(c)}.
\end{equation*}
Now Lemma \ref{lem:characters_obstruct} with $\eta=q\bar q\bar\eta$ and $\delta=\delta^{O_{m,d,D}(c)}$ shows that $(g(n)\Gamma)_{n\in [N]}$ is not totally $\delta^{O_{m,d,D}(Cc)}$-equidistributed, for some $C=C(m,d)\geq 1$. If $c$ was chosen small enough, this contradicts our assumption that $(g(n)\Gamma)_{n\in[N]}$ is totally $\delta$-equidistributed. 
\end{proof}

Finally, we require the following result which allows us to realise a given equidistributed polynomial nilsequence as a linear subsequence of another equidistributed polynomial nilsequence. The proof is similar to that of Proposition \ref{prop:equidist_poly_subseq}.

\begin{proposition}\label{prop:spread_nilsequence}
  Let $m,d\geq 0$. Let $G/\Gamma$ be an $m$-dimensional nilmanifold with a filtration
  $G_\bullet$ and a rational Mal'cev basis $\mathcal{X}$ adapted to this
  filtration. Let $g\in\poly(\ZZ,G_\bullet)$ be a polynomial sequence, and $q\in\NN$ and $b\in\ZZ$.
  \begin{enumerate}
  \item There is a polynomial sequence $\tilde g\in\poly(\ZZ,G_\bullet)$ that
    satisfies
    \begin{equation}\label{eq:def_g_tilde}
      \tilde g(n) = g\left(\frac{n-b}{q}\right)\quad\text{ whenever }n\in\ZZ\text{ with
      }n\equiv b\bmod q.
    \end{equation}
  \item There is a constant $c=c(m,d)\in(0,1)$, depending only on $m$ and $d$,
    such that the following holds: Let $N\in\NN$, assume that $N\gg_{m,d}1$, and set $\tilde N=qN+b$. Assume that $\delta\in(\tilde N^{-1/2}, 2^{-1/c})$, $q\leq
    \delta^{-c}$ and $|b|\leq \delta^{-c}N$, and that the Mal'cev basis $\mathcal{X}$
    is $\delta^{-c}$-rational.  If $(g(n)\Gamma)_{n\in[N]}$ is totally
    $\delta$-equidistributed, then the sequence $(\tilde g(n)\Gamma)_{n\in [\tilde
      N]}$ is totally $\delta^{c}$-equidistributed.
  \end{enumerate}
\end{proposition}

\begin{proof}
  Let $\psi:G\to\RR^m$ denote the coordinates (of the second kind) with respect to our Mal'cev basis $\cX$. By \cite[Lemma 6.7]{MR2877065}, we have
  \begin{equation*}
    \psi(g(n)) = \sum_{j=0}^d\vt_j\binom{n}{j},
  \end{equation*}
  with vectors $\vt_j\in\RR^m$ satisfying $(\vt_j)_i=0$ for all $i\leq m-m_j$,
  where $m_j=\dim G_j$. We set
  \begin{equation*}
    \tilde g(n) := \psi^{-1}\left(\sum_{j=0}^d\vt_j\binom{(n-b)/q}{j}\right),
  \end{equation*}
  then it is clear that \eqref{eq:def_g_tilde} holds.
  Note that each $\binom{(x-b)/q}{j}$ is a polynomial of degree $j$ in
  $\QQ[x]$. Writing it in the binomial basis, we obtain coefficients
  $s_{j,0},\ldots,s_{j,j}\in\QQ$ such that
  \begin{equation*}
    \binom{(x-b)/q}{j} = \sum_{i=0}^js_{j,i}\binom{x}{i}.
  \end{equation*}
  Hence,
  \begin{equation*}
    \psi(\tilde g(n)) = \sum_{j=0}^d\sum_{i=0}^j\vt_js_{j,i}\binom{n}{i}=\sum_{i=0}^d\left(\sum_{j=i}^d\vt_js_{j,i}\right)\binom{n}{i}=:\sum_{i=0}^d\tilde\vt_i\binom{n}{i}.
  \end{equation*}
  The vectors $\tilde\vt_i\in\RR^m$ then satisfy for all $k\leq m-m_i$ that
  \begin{equation*}
    (\tilde\vt_i)_k = \sum_{j=i}^d(\vt_j)_ks_{j,i}=0,
  \end{equation*}
  as $k\leq m-m_i\leq m-m_j$ for all $j\geq i$. Again by \cite[Lemma
  6.7]{MR2877065}, this shows that $\tilde g\in\poly(\ZZ,G_\bullet)$, thus
  proving \emph{(1)}.

  Suppose now that, for some sufficiently small $c$, the sequence $(\tilde g(n)\Gamma)_{n\in\tilde N}$ is not totally
  $\delta^c$-equidistributed. Then
  there is a progression $P=\{u+vn\where n\in [\bar N]\}\subseteq [\tilde N]$ with
  $\bar N\geq \delta^c\tilde N$ and a Lipschitz function $F:G/\Gamma\to \CC$, such that
  \begin{equation*}
    \left|\EE_{n\in P}F(g(n)\Gamma)-\int_{G/\Gamma}F\right|\leq\delta^c\vecnorm{F}_{\Lip}.
  \end{equation*}
  In other words, the sequence $(\bar g(n)\Gamma)_{n\in[\bar N]}$
  with $\bar g(n)=\tilde g(u+vn)$ is not $\delta^c$-equidistributed.
  The
  quantitative Leibmann theorem \cite[Theorem 2.9]{MR2877065} shows the
  existence of a
  nontrivial horizontal character $\bar\eta$ with $0<|\bar\eta|\ll_{m,d} \delta^{-O_{m,d}(c)}$,
  such that
  \begin{equation*}
    \vecnorm{\bar\eta\circ \bar g}_{\Cinf[\bar N]}\ll_{m,d}\delta^{-O_{m,d}(c)}.
  \end{equation*}
  Note that, crudely, $\bar N\geq 2$ and thus $1\leq v\leq 2\delta^{-c}\leq 2\delta^{-1}\leq \tilde N$, which also gives $|u|\leq \tilde N$. Hence, by Lemma \ref{lem:polynomial_extrapolation} with $D=1$ and $Q=2\delta^{-c}$,
  there is $q\in\ZZ$ with $1\leq |q|\ll_q
  \delta^{-O_d(c)}$, such that the horizontal character $\eta=q\bar\eta$ with
  $|\eta|=\delta^{-O_{m,d}(c)}$ satisfies
  \begin{equation*}
    \vecnorm{\eta\circ\tilde g}_{\Cinf[\tilde
      N]}\leq\delta^{-O_{d}(c)}\vecnorm{\eta\circ\tilde g}_{\Cinf[\bar
      N]}\ll_{d}\delta^{-O_{d}(c)}\vecnorm{\bar\eta\circ\bar g}_{\Cinf[\bar
      N]} \ll_{m,d}\delta^{-O_{m,d}(c)}.
  \end{equation*}
  Now $g(n)=\tilde g(qn+b)$, so applying Lemma \ref{lem:polynomial_extrapolation}
  once more, we see that
  \begin{equation*}
    \vecnorm{\eta\circ g}_{\Cinf[N]}\leq \vecnorm{\eta\circ
      g}_{\Cinf[\tilde N]}\ll_{d}\delta^{-O_d(c)}\vecnorm{\eta\circ
      \tilde g}_{\Cinf[\tilde N]}\ll_{m,d}\delta^{-O_{m,d}(c)}.
  \end{equation*}
  If $\delta_0$ is small enough in terms of $m,d,c$, this implies by Lemma
  \ref{lem:characters_obstruct} that the sequence $(g(n)\Gamma)_{n\in[N]}$ is not totally
  $\delta^{O_{m,d}(Cc)}$-equidistributed, for some $C=C(m,d)\geq 1$. Hence, if
  $c$ was chosen small enough this contradicts our assumption that
  $(g(n)\Gamma)_{n\in[N]}$ is totally $\delta$-equidistributed.  
\end{proof}

\section{Ideal von Mangoldt function: proof of
  Theorem \texorpdfstring{\ref{thm:ideal_von_mangoldt_equidist_nilsequences}}{2.4}\label{sec:ideal_von_mangoldt}}

\subsection{Ideal norms and equidistributed nilsequences}
We will need cancellation when equidistributed nilsequences of mean zero are summed over the norms of ideals of a number field $K$ (i.e. nonzero ideals of the ring of integers $\OO_K$), twisted by Dirichlet characters of $K$. A result of Kane \cite[Lemma 21]{MR3060874} deals with the case of exponential phases $e(\alpha n)$. Here, we generalise Kane's result to the non-abelian setting.

\begin{lemma}\label{lem:norm_coefficients_divisible}
  Let $K$ be a number field of degree $D$ and $a,b\in K$ with $b\neq 0$. Then $N_{K/\QQ}(a+bx)$ is a polynomial in $\QQ[x]$ of degree $D$. Write $N_{K/\QQ}(a+bx)=\sum_{i=0}^Da_ix^i$. Then:
  \begin{enumerate}
  \item Let $\ccc$ an ideal of $K$. If $a,b\in\ccc$, then $a_i$ is an integer divisible by $\norm\ccc$ for all $0\leq i\leq D$.
  \item Let $A,B>0$ with $\abs{a}_v\leq A$ and $\abs{b}_v\leq B$ for all $v\mid\infty$, then $|a_i|\ll_D A^{D-i}B^i$ for all $0\leq i\leq D$.
  \end{enumerate}
\end{lemma}

\begin{proof}
  Let $L$ be the normal closure of $K$, $G=\Gal(L/\QQ)$, $H=\Gal(L/K)\leq G$ and $R\subseteq G$ a system of representatives for $G/H$. Then
  \begin{equation*}
    N_{K/\QQ}(a+bx)=\prod_{\sigma\in R}(\sigma(a)+\sigma(b)x)=\sum_{i=0}^Da_ix^i, 
  \end{equation*}
  with
  \begin{equation}\label{eq:norm_poly_coeff_formula}
    a_i = \sum_{\substack{I\subseteq R\\|I|=i}}\prod_{\sigma\notin I}\sigma(a)\prod_{\sigma\in I}\sigma(b)\in L.
  \end{equation}
  For any $\tau\in G$, the set $\{\tau\sigma\ :\ \sigma\in R\}$ is also a system of representatives of $G/H$, which shows that $\tau(a_i)=a_i$, and thus $a_i\in\QQ$.

  For assertion \emph{(1)}, we observe from \eqref{eq:norm_poly_coeff_formula} that $a_i\in \prod_{\sigma\in R}\sigma(\ccc)\OO_L$. Hence, it is sufficient to show that
  \begin{equation}\label{eq:ideal_conjugates_norm_equation}
    \prod_{\sigma\in R}\sigma(\ccc)\OO_L = \norm_{K/\QQ}(\ccc)\OO_L,
  \end{equation}
  as then $a_i\in \norm_{K/\QQ}(\ccc)\OO_L\cap\QQ = \norm_{K/\QQ}(\ccc)\ZZ$. To show \eqref{eq:ideal_conjugates_norm_equation}, we first consider a prime ideal $\PPP$ of $L$. With $p=\PPP\cap\ZZ$ and the factorisation $p\OO_L=(\PPP_1\cdots\PPP_r)^e$ with $\PPP_1=\PPP$ and $[\OO_L/\PPP_i:\FF_p]=f$, we get
  \begin{equation*}
    \prod_{\sigma\in G}\sigma(\PPP)=(\PPP_1\cdots\PPP_r)^{|G|/r}=(\PPP_1\cdots\PPP_r)^{ef}=(p\OO_L)^f=\norm_{L/\QQ}(\PPP)\OO_L.
  \end{equation*}
  This implies \eqref{eq:ideal_conjugates_norm_equation} in case $K=L$. Now let $\ccc$ be any nonzero ideal of $\OO_K$. Then, using what we just proved,
  \begin{equation*}
    \left(\prod_{\sigma\in R}\sigma(\ccc)\OO_L\right)^{|H|} = \prod_{\sigma\in G}\sigma(\ccc\OO_L) = \norm_{L/\QQ}(\ccc\OO_L)\OO_L = \norm_{K/\QQ}(\ccc^{|H|})\OO_L = (\norm_{K/\QQ}(\ccc)\OO_L)^{|H|},
  \end{equation*}
  which shows \eqref{eq:ideal_conjugates_norm_equation} in general and thus \emph{(1)}.

  For \emph{(2)}, it is enough to observe that in \eqref{eq:norm_poly_coeff_formula} one has
  \begin{equation*}
    \left|\prod_{\sigma\notin I}\sigma(a)\prod_{\sigma\in I}\sigma(b)\right| = \prod_{\sigma\notin I}\abs{a}_{v_\sigma}\prod_{\sigma\in I}\abs{b}_{v_\sigma} \leq A^{D-i}B^i. 
  \end{equation*}
\end{proof}

Recall that a Dirichlet character $\xi$ of $K$ is a finite order Hecke character. We consider it as a character of the ray class group $\xi:I(\mmm)/P_\mmm\to S^1$, where $\mmm=\mmm_\infty\mmm_0$ is a cycle, $I(\mmm)$ is the group of fractional ideals relatively prime to $\mmm_0$ and $P_\mmm$ is the subgroup of principal fractional ideals $(\alpha)\in I(\mmm)$ with $\alpha\equiv 1\mods\mmm$. See \cite[Chapter VI]{MR1282723} for more on this notation. As usual, for arbitrary nonzero ideals $\aaa$ of $\OO_K$, we set $\xi(\aaa):=\xi(\aaa P_\mmm)$ if $\aaa\in I(\mmm)$ and $\xi(\aaa):=0$ otherwise. Here is our version of \cite[Lemma 21]{MR3060874}.

\begin{proposition}[Cancellation of equidistributed nilsequences along ideal
  norms in progressions]\label{prop:norms_vs_equidist_nilsequences}
  Let $m,d,D\in\NN$. Then there is a constant $c=c(m,d,D)\in(0,1/2)$, such that the following holds. Let $N,Q\in\NN$ and $\delta\in(0,1)$.
  Let $G/\Gamma$ be an
  $m$-dimensional nilmanifold together with a filtration $G_\bullet$ of degree
  $d$ and a $Q$-rational Mal'cev basis adapted to this
  filtration. Let $g\in\poly(\ZZ,G_\bullet)$ be a polynomial sequence of degree
  $d$ and suppose that $(g(n)\Gamma)_{n\in[N]}$ is totally
  $\delta$-equidistributed in $G/\Gamma$. Let $F:G/\Gamma\to\CC$ be a
  Lipschitz-function such that $\int_{G/\Gamma}F=0$.  
  Let $K$ be a number field of degree $D$ and $\xi$ a Dirichlet character of $K$ of modulus $\mmm$. 
  Let $P\subseteq [N]$ be an arithmetic progression of
  length at least $N/Q$. Then
  \begin{equation*}
    \left|\sum_{\substack{\aaa\in\ideals\\\norm\aaa\in P}}\xi(\aaa)F(g(\norm\aaa)\Gamma)\right|\ll_{m,d,K,\mmm} N\delta^{c(m,d,D)}Q\vecnorm{F}_{\textnormal{Lip}}.
  \end{equation*}
\end{proposition}

\begin{proof}
  We start with a few standard reductions as in the proof of \cite[Proposition 2.1]{MR2877065}.
  As $\vecnorm{F}_\infty\leq\vecnorm{F}_{\Lip}$, the desired bound holds trivially whenever $\delta\gg_{m,d,K,\mmm} 1$. Moreover, if $\delta\leq 1/N$, then the total $\delta$-equidistribution of $(g(n)\Gamma)_{n\in[N]}$ implies that $|\xi(\aaa)F(g(\norm\aaa)\Gamma)|\leq \delta\vecnorm{F}_{\Lip}$ for every ideal $\aaa$ with $\norm\aaa\leq N$, again making the conclusion trivial. Hence,
  we may assume that $\delta>N^{-1/2}$. Indeed, for $\delta>1/N$ we may then apply the result for $\delta^{1/2}$, thus halving the value of $c(m,d,D)$. Therefore, we assume from now on that
  \begin{equation}\label{eq:kane_prop_delta_bounds}
  N^{-1/2}< \delta \leq \eta(m,d,K,\mmm), 
\end{equation}
where the choice of $\eta(m,d,K,\mmm)<1$ will be described later in the proof. Now we may clearly assume that $N$ is sufficiently large in terms of $m,d,K,\mmm$. Taking $c=c(m,d,D)$ to be the constant from Proposition \ref{prop:equidist_poly_subseq}, we will show the desired bound with $c/2$ in place of $c$.

Hence, we may assume that $Q\leq \delta^{-c/2}<\delta^{-c}$, as otherwise the result is again trivial. Hence, our Mal'cev basis is $\delta^{-c}$-rational, and the progression $P$ has length $\geq \delta^{c/2}N$.
  
Next, we recall the classical method to sum Dirichlet characters over ideals of bounded norm by lattice point counting. Fix a ray class $\cC\in I(\mmm)/P_\mmm$ and an ideal $\ccc\in I(\mmm)$ with $\norm\ccc\ll_{K,\mmm} 1$ such that $\ccc^{-1}$ represents the class $\cC$. Then ideals $\aaa\in \cC$ have the form $\aaa=a\ccc^{-1}$ for $a\in\ccc\smallsetminus\{0\}$ with $a\equiv 1\mods\mmm$. As $a\ccc^{-1}=b\ccc^{-1}$ if and only if $ab^{-1}\in U_\mmm$, the group of units of $\OO_K$ congruent to $1 \mods \mmm$, the ideals $\aaa\in \cC$ are parameterised in this way by classes of elements $a$ as above modulo $U_\mmm$. We embed $K$ into the $\RR$-algebra $K_\infty := K\otimes_\QQ\RR = \prod_{v\mid\infty}K_v$ of dimension $D$, then the elements $a\in\ccc$ with $a\equiv 1\mods\mmm_0$ form a translate of the lattice $\ccc\mmm_0\in K_\infty$. Let us call this translate $T_\ccc$ and observe that $T_\ccc\subseteq\ccc$. To incorporate the archimedean part of $\mmm$, we consider the subgroup
\begin{equation*}
J_K(\infty,\mmm):=\{x\in K_\infty^\times\ :\ x_v>0 \text{ for real }v\mid\mmm\}\subseteq K_\infty^\times = \prod_{v\mid\infty}K_v^\times,
\end{equation*}
on which $U_\mmm$ acts by multiplication. Hence, ideals $\aaa\in\ccc$ are parameterised by elements $a\in T_\ccc\cap J_K(\infty,\mmm)$, up to the action of $U_\mmm$, via $a\mapsto a\ccc^{-1}$. In particular, $\norm\aaa=|N_{K/\QQ}(a)|\norm\ccc^{-1}$, so $\norm\aaa\leq B$ holds if and only if $|N_{K/\QQ}(a)|\leq B\norm\ccc$.

Identifying each $K_v$ with $\RR$ or $\CC$ provides us with notions of euclidean norm and distance, as well as volume on $K_\infty$ and thus on $K_\infty^\times$. Let $V$ be a free abelian subgroup of $U_\mmm$ that generates $U_\mmm$ modulo roots of unity. Then $V$ is finitely generated by Dirichlet's unit theorem. A classical construction (see e.g. \cite[VI,\S 3, Lemma 1]{MR1282723}) yields a fundamental domain $\cF \subset J_K(\infty,\mmm)$ for the action of $U_\mmm$ on $J(\infty,\mmm)$ with the following properties: $t\cF=\cF$ for all $t>0$, and the set $\cF(1)\subseteq \cF$ has a $(D-1)$-Lipschitz-parameterisable boundary, where
\begin{equation*}
\cF(B):=\{\alpha\in \cF\where |N_{K/\QQ}(\alpha)|\leq B\}.
\end{equation*}
It follows from these properties that $\cF(B)=B^{1/D}\cF(1)$, in particular $\cF(B)$ has diameter $\asymp_{K,\mmm} B^{1/D}$ and a $(D-1)$-Lipschitz-parameterisable boundary with Lipschitz constant $\asymp_{K,\mmm} B^{1/D}$.
  
We write the progression $P\subseteq [N]$ as $P=\{n\in (N_1,N_2]\where n\equiv u\bmod q\}$, with integers $0\leq u<q$ and $0\leq N_1< N_2\leq N$ such that $N_2-N_1\gg q\delta^{c/2} N$. In particular, observe that $q\ll \delta^{-c/2}$. Writing $R:=\cF(N_2\norm\ccc)\smallsetminus \cF(N_1\norm\ccc)$ and denoting by $\omega_\mmm$ the number of roots of unity in $U_\mmm$, we see that
\begin{align}\label{eq:kane_prop_ray_classes_split}
  \sum_{\substack{\aaa\in \cC\\\norm\aaa\in P}}\xi(\aaa)F(g(\norm\aaa)\Gamma) &= \frac{\xi(\ccc)}{\omega_\mmm}\sum_{a\in T_\ccc\cap R}1_{u+q\ZZ}(|N_{K/\QQ}(a)|\norm\ccc^{-1})F(g(|N_{K/\QQ}(a)|\norm\ccc^{-1})\Gamma).
\end{align}
As there are $\ll_{K,\mmm} 1$ ray classes $\cC$, it suffices to bound a sum such as on the right-hand side above.

Fix a primitive element $r\in\ccc\mmm_0$ (i.e. $r\notin k\ccc\mmm_0$ for any $k\in\NN$, $k>1$) with $|r|_v\ll_{K,\mmm} 1$ for all $v\mid \infty$ and an integer $M=\lfloor N^{1/D}C\delta^{c/2}\rfloor$, with a sufficiently large constant $C=C(m,d,K,\mmm)$. We define a \emph{line} in $K_\infty$ to be a line segment of the form $L=b+(0,M]r\subseteq K_\infty$ of length $(M+1)|r|$, where $b\in T_\ccc$. Then $L_\ccc := L\cap T_\ccc = \{b+r,\ldots,b+Mr\}$. Clearly, every point $a\in T_\ccc$ is contained in exactly $M$ lines, so we can write
\begin{multline}\label{eq:splitting_into_lines}
  \sum_{a\in T_\ccc\cap R}1_{u+q\ZZ}(|N_{K/\QQ}(a)|\norm\ccc^{-1})F(g(|N_{K/\QQ}(a)|\norm\ccc^{-1})\Gamma) \\= \frac{1}{M}\sum_{L \text{ line}}\sum_{a\in L_\ccc\cap R}1_{u+q\ZZ}(|N_{K/\QQ}(a)|\norm\ccc^{-1})F(g(|N_{K/\QQ}(a)|\norm\ccc^{-1})\Gamma).
\end{multline}
Due to the Lipschitz-parameterisability of the boundaries of $\cF(N_1\norm\ccc)$ and $\cF(N_2\norm\ccc)$, the boundary of $R$ is contained in the union of $\ll_{K,\mmm} N^{1-1/D}$ balls of radius $\ll_{K,m} 1$. Each such ball intersects the closures of at most $\ll_{K,\mmm} M$ lines. Hence, there are at most $\ll_{K,\mmm}MN^{1-1/D}$ lines whose closure intersects the boundary of $R$.
As each line contains at most $M$ points of $T_\ccc\cap R$,
the total contribution of these lines to \eqref{eq:splitting_into_lines} is $\ll_{K,\mmm}MN^{1-1/D}\vecnorm{F}_\infty\leq NC\delta^{c/2}\vecnorm{F}_{\textnormal{Lip}}$, which is acceptable for our desired bound.

Next, we bound the contribution to \eqref{eq:splitting_into_lines} coming from a line $L=b+(0,M]r$ whose closure is contained in the interior of $R$, so in particular $b\in R$. Recall from the definition of a line that $b\in T_\ccc$ and $r\in\ccc\mmm_0$, so in particular $b,r\in\ccc$. With Lemma \ref{lem:norm_coefficients_divisible}, we see that $N_{K/\QQ}(b+rx)$ is a polynomial in $x$ of degree $D$ with integer coefficients divisible by $\norm\ccc$. The coefficient of $x^i$ is $\ll_{K,\mmm} N^{1-i/D}$, as $b\in R\subseteq (N\norm\ccc)^{1/D}\cF(1)$ and $\abs{r}_{v}\ll_{K,\mmm}1$ for all $v\mid\infty$. As $R\subseteq \cF\subseteq K_\infty^\times$, we have $N_{K/k}(a)\neq 0$ for all $a\in R$. As $L$ does not meet the boundary of $R$, it lies entirely within one connected component, so the sign of $N_{K/\QQ}(a)$ is constant on $L$. Hence, for $a=b+rx\in L$ we may write
\begin{equation*}
 |N_{K/\QQ}(a)|\norm\ccc^{-1}=P_L(x),
\end{equation*}
for a polynomial $P_L\in\ZZ[x]$ of degree $D$, in which the coefficient $a_i$ of $x^i$ satisfies $ a_i \ll_{K,\mmm}  N^{1-i/D}$. We choose $\eta(m,d,K,\mmm)$ sufficiently small in \eqref{eq:kane_prop_delta_bounds} so that this implies $|a_i|\leq \delta^{-c}N^{1-i/D}$.

By Proposition \ref{prop:equidist_poly_subseq}, the sequence $((g\circ P_L)(n)\Gamma)_{n\in \tilde N}$ is totally $\delta^{c}$-equidistributed in $G/\Gamma$, where $\tilde N=\lfloor N^{1/D}\rfloor$. The sum we are trying to bound now takes the form
\begin{equation}\label{eq:application_of_poly_subsequence}
  \sum_{a\in L_\ccc}1_{u+q\ZZ}(|N_{K/\QQ}(a)|\norm\ccc^{-1})F(g(|N_{K/\QQ}(a)|\norm\ccc^{-1})\Gamma) = \sum_{n\in [M]}1_{u+q\ZZ}(P_L(n))F(g(P_L(n))\Gamma).
\end{equation}
To deal with the factor $1_{u+q\ZZ}(P_L(n))$, we observe that it is constant on residue classes modulo $q$. Hence, splitting $[M]$ into such classes, the sum above becomes
\begin{equation*}
  \sum_{v\bmod q}1_{u+q\ZZ}(P_L(v))\sum_{\substack{n\in[M]\\n\equiv v\bmod q}}F(g(P_L(n))\Gamma)
\end{equation*}
Recall that $q\ll \delta^{-c/2}$. Hence we may choose
$C$ and $N$ sufficiently large in order to ensure that the set $\{n\in[M]\where n\equiv v\bmod q\}$ is a subprogression of $[\tilde N]$ of length $\geq \delta^{c}\tilde N$. Therefore, total $\delta^{c}$-equidistribution yields the estimate
\begin{equation*}
  \sum_{\substack{n\in [M]\\n\equiv v\bmod q}}F(g(P_L(n))\Gamma) \ll \frac{M}{q}\delta^{c}\vecnorm{F}_{\textnormal{Lip}},
\end{equation*}
and hence the bound $\ll M\delta^{c}\vecnorm{F}_{\textnormal{Lip}}$ for the expression in \eqref{eq:application_of_poly_subsequence}.

The number of lines $L$ whose closure is contained in the interior of $R$ is clearly bounded by $|\ccc\cap \cF(N\norm\ccc)|\ll_{K,\mmm} N$. Summing the above bound over all these lines in \eqref{eq:splitting_into_lines} and including our bound for the lines intersecting the boundary shown earlier, we see that
\begin{equation*}
  \sum_{a\in T_\ccc\cap R}1_{u+q\ZZ}(|N_{K/\QQ}(a)|\norm\ccc^{-1})F(g(|N_{K/\QQ}(a)|\norm\ccc^{-1})\Gamma) \ll_{K,\mmm} N\delta^{c/2}\vecnorm{F}_{\textnormal{Lip}}.
\end{equation*}
Inserting this in \eqref{eq:kane_prop_ray_classes_split} shows the desired bound.
\end{proof}

\begin{corollary}\label{cor:norms_vs_equidist_nilsequences_log}
  Let $m,d,D\in\NN$. Then there is a constant $c=c(m,d,D)\in(0,1/2)$, such that under the hypotheses of Proposition \ref{prop:norms_vs_equidist_nilsequences}, we have
    \begin{equation*}
    \left|\sum_{\norm\aaa\in P}\xi(\aaa)F(g(\norm\aaa)\Gamma)\log(\norm\aaa)\right|\ll_{m,d,K,\mmm} N\log( N+2)\delta^{c(m,d,D)}Q\vecnorm{F}_{\textnormal{Lip}}.
  \end{equation*}
\end{corollary}

\begin{proof}
  As at the start of the proof of Proposition \ref{prop:norms_vs_equidist_nilsequences}, we may assume that $1/N<\delta<1$, as otherwise the bound is either trivial or follows immediately from the total $\delta$-equidistribution of $(g(n)\Gamma)_{n\in[N]}$.
  
  Let $c=c(m,d,D)$ be the constant in Proposition \ref{prop:norms_vs_equidist_nilsequences}. We will show the corollary's conclusion with $c/2$ in place of $c$. Hence, we may assume that $Q\leq \delta^{-c/2}\leq N^{1/4}$, as otherwise the result is trivial. Thus, we may also assume that $N\gg_{m,d,K,\mmm} 1$.
  
  Write the progression as $P=\{b+nq\ :\ n\in [M]\}$, for some $M\geq N/Q$. We note that then $q\leq (N-1)/(M-1)\ll N/M \leq Q$. Set $a_0:=0$ and 
  \begin{equation*}
    a_n := \sum_{\norm\aaa=b+nq}\xi(\aaa)F(g(b+nq)\Gamma)
  \end{equation*}
  for $n\in [M]$. For $t\geq 0$, let $A(t)=\sum_{0\leq n\leq t}a_n$. From the ideal theorem, we have the first bound
  \begin{equation*}
    |A(t)| \leq \sum_{\norm\aaa\leq b+qt}\vecnorm{F}_\infty \ll_K (b+qt)\vecnorm{F}_{\Lip}.
  \end{equation*}
  Set $\tilde Q:=\lceil 2\delta^{-c/2}Q\rceil$, then $N^{1/3}\leq \lceil N/\tilde Q\rceil \leq M$, as $N\gg 1$.

  For any $t\geq\lceil N/\tilde Q\rceil$, consider the progression $P_t=\{b+nq\ :\ n\in (0,t]\cap\NN\}$ of length $\lfloor t\rfloor \geq N/\tilde Q$. As $\tilde Q\geq Q$, our Mal'cev basis is also $\tilde Q$-rational, and we may apply Proposition \ref{prop:norms_vs_equidist_nilsequences} with $\tilde Q$ in place of $Q$ to obtain the bound
    \begin{equation*}
      |A(t)|=\left|\sum_{\norm\aaa\in P_t}\xi(\aaa)F(g(\norm\aaa)\Gamma)\right|\ll_{m,d,K,\mmm} N\delta^c\tilde Q\vecnorm{F}_{\Lip}\ll N\delta^{c/2}Q\vecnorm{F}_{\Lip}\quad\text{ for }t\geq\lceil N/\tilde Q\rceil.  
    \end{equation*}
  
Using the Abel summation formula, we write the sum to be estimated in the corollary as
\begin{align*}
  &\sum_{n\in[M]}a_n\log(b+nq) = A(M)\log(n+Mq)-\int_{1}^M\frac{qA(t)}{b+qt}\mathrm{d}t\\ &\ll_{m,d,K,\mmm}N\delta^{c/2}Q\vecnorm{F}_{\Lip}\log N + \int_{1}^{\lceil N/\tilde Q\rceil} \frac{q(b+qt)\vecnorm{F}_{\Lip}}{b+qt}\mathrm{d}t + \int_{\lceil N/\tilde Q\rceil}^M\frac{qN\delta^{c/2}Q\vecnorm{F}_{\Lip}}{b+qt}\mathrm{d}t\\
  &\ll N\delta^{c/2}Q\vecnorm{F}_{\Lip}\log N + \frac{N}{\tilde Q}q\vecnorm{F}_{\Lip} + N\delta^{c/2}Q\vecnorm{F}_{\Lip}\int_{\lceil N/\tilde Q\rceil}^M\frac{1}{(t-1)}\mathrm{d}t\\
  &\ll N\delta^{c/2}Q\vecnorm{F}_{\Lip}\log N.
\end{align*}
\end{proof}

\subsection{Ideal von Mangoldt function and
  Vaughan's identity}\label{sec:ideal_moebius_nilsequences}
Let $K$ be a number field of degree $D=[K:\QQ]$. For $n,d\in\NN$ and an ideal $\bbb$ of $K$, we write
\begin{align*}
  a_{K}(n) &:= \#\{\aaa\text{ ideal of }\OO_K\where \norm\aaa=n\},\\
  \tau_K(\bbb) &:= \sum_{\aaa\mid\bbb}1,\\
  \tau_d(n) &:=\sum_{n_1\cdots n_d=n}1,\\
  \tau(n)&:=\tau_2(n)
           =\sum_{k\mid n}1.
\end{align*}

\begin{lemma}\label{lem:ideal_divisor_relations}
  Let $K$ be a number field of degree $D$. For $n\in\NN$ and an ideal $\aaa$ of $K$, we have
  \begin{enumerate}
  \item $a_{K}(n) \leq \tau_D(n) \leq \tau(n)^{D-1}$,
  \item $\tau_K(\aaa)\leq \tau(\norm(\aaa))^D$.
  \end{enumerate}
\end{lemma}

\begin{proof}
For the first inequality in \emph{(1)}, see, e.g., \cite[pg. 940]{Blomer2018}; the second one can be easily proved using induction on $d$. To prove the second part, we use this to obtain
\[
\tau_K(\aaa)=\sum_{\bbb|\aaa}1\leq \sum_{\bbb,\norm \bbb|\norm \aaa}1=\sum_{n|\norm \aaa}a_{K}(n)\leq \sum_{n|\norm \aaa}\tau(n)^{D-1}\leq \sum_{n|\norm \aaa}\tau(\norm \aaa)^{D-1}=\tau(\norm \aaa)^{D}.
\]
\end{proof}

\begin{lemma}\label{lem:divisor_log_bound}
  For any $k\in\NN$, we have
  \begin{equation*}
    \sum_{n\leq N}\frac{\tau(n)^k}{n} \ll (\log N)^{2^k}.
  \end{equation*}
\end{lemma}

\begin{proof}
  The divisor moment bound $\EE_{n\leq N}\tau(n)^k\ll (\log N)^{2^k-1}$ is well known, for references see \cite[Lemma C.1]{MR2473624}. Splitting dyadically, we see that
\begin{align*}
\sum_{n\leq N}\frac{\tau(n)^k}{n}&\leq \sum_{0\leq i\leq \lfloor\log N/\log 2\rfloor}\sum_{2^{i}\leq n\leq 2^{i+1}}\frac{\tau(n)^k}{n}\leq \sum_{1\leq i\leq \lfloor\log N/\log 2\rfloor}\frac{1}{2^i}\sum_{2^{i}\leq n\leq 2^{i+1}}\tau(n)^k\\ &\ll \sum_{1\leq i\leq \lfloor\log N/\log 2\rfloor}(\log N)^{2^k-1}\ll (\log N)^{2^k}. 
\end{align*}
\end{proof}

In the following, we prove a version of Vaughan's identity for von Mangoldt function $\Lambda_K(\aaa)$ for ideals $\aaa$ of $K$.

\begin{lemma}[Twisted Vaughan's identity for $\Lambda_K$] \label{lem:vaughan}
 Let $K$ be a number field, consider a completely multiplicative function $\xi:\ideals\to\CC$ and a function $f:\NN\rightarrow \CC$.
 For $N,U,V\geq 1$ with $V\leq N/2$ and $UV\leq N$, we have
\begin{equation*}
\sum_{N/2<\norm \aaa\leq N}\Lambda_K(\aaa)\xi(\aaa)f(\norm \aaa)=T_{\Ia}-T_{\Ib}+T_{\II},
\end{equation*}
where $T_{\Ia},T_{\Ib},T_{\II}$ are given as follows:
\begin{align*}
  T_{\Ia}&:=\sum_{\norm \ddd\leq U}\mu_K(\ddd)\xi(\ddd)\sum_{\frac{N}{2\norm \ddd}<\norm \bbb\leq \frac{N}{\norm \ddd}} \log(\norm \bbb)\xi(\bbb)f(\norm\ddd\bbb),\\
  T_{\Ib}&:=\sum_{ \norm \ddd\leq UV}a_{\ddd}\xi(\ddd)\sum_{\frac{N}{2\norm \ddd}<\norm \bbb\leq \frac{N}{\norm \ddd}}\xi(\bbb)f(\norm\ddd\bbb),\\
  T_{\II}&:=\sum_{V<\norm \ddd\leq \frac{N}{U}}b_{\ddd}\xi(\ddd)\sum_{\max\big(U,\frac{N}{2\norm \ddd}\big)<\norm \bbb\leq \frac{N}{\norm \ddd}}\mu_K(\bbb)\xi(\bbb)f(\norm\ddd\bbb),
\end{align*}
with
\begin{equation*}
  a_{\ddd}:=\sum_{\substack{\bbb\ccc=\ddd\\\norm \bbb\leq U,\norm \ccc\leq V}}\mu_K(\bbb)\Lambda_K(\ccc),\quad\text{ and }\quad b_{\ddd}:=\sum_{\substack{\ccc|\ddd\\\norm \ccc>V}}\Lambda_K(\ccc).
\end{equation*}
\end{lemma}

\begin{proof}
  A straightforward generalisation of \cite[Proposition 13.4]{MR2061214} to ideals in number fields
  yields
\[
\Lambda_K(\aaa)=\sum_{\substack{\bbb|\aaa\\ \norm \bbb \leq U}}\mu_K(\bbb)\log (\norm \aaa\bbb^{-1})-\sum_{\substack{\bbb\ccc|\aaa\\ \norm\bbb\leq U,\norm \ccc\leq V}}\mu_K(\bbb)\Lambda_K (\ccc)+\sum_{\substack{\bbb\ccc|\aaa\\ \norm\bbb> U,\norm \ccc> V}}\mu_K(\bbb)\Lambda_K (\ccc)
\]
whenever $\norm \aaa > V$. When multiplied by $\xi(\aaa)f(\norm \aaa)$ and summed over $N/2<\norm \aaa\leq N$, the first summand becomes
\[
\sum_{\norm \bbb\leq U}\mu_K(\bbb)\xi(\bbb)\sum_{\frac{N}{2\norm \bbb}<\norm \ccc\leq \frac{N}{\norm \bbb}} \log(\norm \ccc)\xi(\ccc)f( \norm\bbb\ccc)=T_{\Ia}.
\]
Treating the second summand in the same way, we obtain
\begin{align*}
& \sum_{N/2<\norm \aaa\leq N}\Bigg(\sum_{\substack{\bbb\ccc|\aaa\\ \norm\bbb\leq U,\norm \ccc\leq V}}\mu_K(\bbb)\Lambda_K (\ccc)\Bigg)\xi(\aaa)f(\norm \aaa)\\= & \sum_{\norm \ddd \leq UV}\Bigg(\sum_{\substack{\bbb\ccc=\ddd\\ \norm \bbb\leq U,\norm \ccc\leq V}}\mu_K(\bbb)\Lambda_K(\ccc)\Bigg)\xi(\ddd)\sum_{\frac{N}{2\norm\ddd}<\norm \w\leq \frac{N}{\norm \ddd}}\xi(\w)f(\norm\ddd\w)=T_{\Ib}.
\end{align*}
Finally, the last summand becomes
\begin{align*}
&\sum_{N/2<\norm \aaa\leq N}\Bigg(\sum_{\substack{\bbb\ccc|\aaa\\ \norm\bbb> U,\norm \ccc> V}}\mu_K(\bbb)\Lambda_K (\ccc)\Bigg)\xi(\aaa)f(\norm \aaa)\\=&\sum_{V<\norm \ddd \leq \frac{N}{U}}\Bigg(\sum_{\substack{\ccc\mid\ddd\\ \norm \ccc> V}}\Lambda_K(\ccc)\Bigg)\xi(\ddd)\sum_{\max\big(U,\frac{N}{2\norm \ddd}\big)<\norm \bbb\leq \frac{N}{\norm \ddd}}\mu_K(\bbb)\xi(\bbb)f(\norm\ddd\bbb)=T_{\II}.
\end{align*}
\end{proof}

Using Lemma \ref{lem:vaughan}, we deduce the following version of \cite[Proposition 4.2]{MR2473624} for $\Lambda_K$ (twisted by a bounded multiplicative function) in place of $\mu$.

\begin{proposition}\label{prop:inverse_theorem_mu}
Let $K$ be a number field of degree $D$, $\xi:\ideals\to \CC$ a completely multiplicative function and $f:\NN\rightarrow\CC$ be a function with $\vecnorm{\xi}_\infty,\vecnorm{f}_\infty\leq 1$. Let $N\geq 3$ and $\epsilon\in(0,1)$, such that
\begin{equation*}
  \left|\sum_{N/2<\norm \aaa\leq N}\Lambda_K(\aaa)\xi(\aaa)f(\norm \aaa)\right|\geq \epsilon N\log N.
\end{equation*}
Then at least one of the following statements holds:
\begin{itemize}
\item (Weighted type I sum is large): There exists an integer $1\leq R\leq N^{1/3}$, such that 
\[
|\EE_{\frac{N}{2r}<\norm \bbb\leq \frac{N}{r}}\log(\norm\bbb)\xi(\bbb)f(r\norm\bbb)|\gg_K (\epsilon/\log N)^{O_D(1)}
\]
holds for $\gg_K R(\epsilon/\log N)^{O_D(1)}$ integers $r$ with $R/2<r\leq R$. 
\item (Type I sum is large): There exists an integer $1\leq R\leq N^{2/3}$, such that 
\[
|\EE_{\frac{N}{2r}<\norm \bbb\leq \frac{N}{r}}\xi(\bbb)f(r\norm\bbb)|\gg_K (\epsilon/\log N)^{O_D(1)}
\]
holds for $\gg_K R(\epsilon/\log N)^{O_D(1)}$ integers $r$ with $R/2<r\leq R$. 
\item (Type II sum is large): There exist integers $R,B$ with
  \begin{equation*}
  N^{1/3}/2\leq R\leq 2 N^{2/3} \quad\text{ and }\quad N/2\leq RB\leq 2N,
\end{equation*}
  such
  that
  \begin{equation*}
    \left|\EE_{B/2<\norm\bbb\leq B}f(r\norm\bbb)f(r'\norm\bbb)\right|\gg_K (\epsilon/\log
  N)^{O_D(1)}
\end{equation*}
holds for $\gg_K R^2(\epsilon/\log N)^{O_D(1)}$ pairs of integers $(r,r')$ with $R/2<r,r'\leq R$.   
\end{itemize}  
\end{proposition}

\begin{proof}
Recall that, by the ideal theorem, the number of ideals of $K$ of norm up to $N$ is $\asymp_K N$.   
We apply Lemma \ref{lem:vaughan} with $U=V=N^{1/3}$. Under the hypotheses of the Proposition, at least one of the quantities $|T_{\Ia}|, |T_{\Ib}|, |T_{\II}|$ is $\gg N(\log N)\epsilon$.

Let us start with the case where $|T_{\Ia}|\gg N(\log N)\epsilon$. As $|\mu_K(\ddd)\xi(\ddd)|\ll 1$, we see that 

\begin{equation}\label{eq:vaughan_0_sum_r}
\epsilon\ll \frac{|T_{\Ia}|}{N} \ll_K \sum_{\norm\ddd\leq N^{1/3}}\frac{1}{\norm \ddd}\left|\EE_{\frac{N}{2\norm \ddd}<\norm \bbb\leq\frac{N}{\norm \ddd}}\log (\norm \bbb )\xi(\bbb)f(\norm \ddd\bbb)\right| = \sum_{r\leq N^{1/3}}\frac{a_K(r)g_0(r)}{r},
\end{equation}
where
\begin{equation*}
  g_0(r):=\left|\EE_{\frac{N}{2r}<\norm \bbb\leq\frac{N}{r}}\log (\norm \bbb )\xi(\bbb)f(r\norm\bbb)\right|.
\end{equation*}
Using Lemma \ref{lem:ideal_divisor_relations} and Lemma \ref{lem:divisor_log_bound}, we see for any exponent $e\in (0,1)$ that
  \begin{equation}\label{eq:vaughan_I_divisor_estimate}
    \sum_{r\leq N^{e}}\frac{a_{K}(r)^2}{r}\leq \sum_{r\leq N^{e}}\frac{\tau(r)^{2D}}{r} \ll_e (\log N)^{2^{2D}}.
  \end{equation}
Therefore, applying the Cauchy-Schwarz inequality to \eqref{eq:vaughan_0_sum_r}, we obtain
\[
\sum_{r\leq N^{1/3}}\frac{g_0(r)^2}{r}\gg_K (\epsilon/\log N)^{O_D(1)}.
\]
Splitting the interval $[1,N^{1/3}]$ dyadically and applying the pigeonhole principle, we find an integer $1\leq R \leq N^{1/3}$, such that 
\begin{equation*}
\sum_{R/2< r\leq R} \frac{g_0(r)^2}{r} \gg_K (\epsilon/\log N)^{O_D(1)},\quad\text{ and thus }\quad \sum_{R/2< r\leq R} g_0(r)^2 \gg_K R(\epsilon/\log N)^{O_D(1)}
\end{equation*}
As $g_0(r)\leq \log N$, this implies that $g_0(r) \gg_K (\epsilon/\log N)^{O_D(1)}$
must hold for at least $R(\epsilon/\log N)^{O_D(1)}$ values of $r$ with $R/2<r\leq R$, and thus the first of the three situations in the proposition's conclusion.
We proceed similarly in the situation where $|T_{\Ib}|\gg N(\log N)\epsilon$.
Note that the coefficients $a_\ddd$ in $T_{\Ib}$ satisfy
\begin{equation*}
  |a_{\ddd}|\ll \sum_{\ccc\mid\ddd}\Lambda_K(\ccc)=\log\norm\ddd\leq\log N.
\end{equation*}
Similarly as before,
we obtain
\begin{equation}\label{eq:vaughan_I_sum_r}
 \epsilon\ll \frac{|T_{\Ib}|}{N\log N}\ll_K \sum_{\norm\ddd\leq N^{2/3}}\frac{1}{\norm\ddd}\left|\EE_{\frac{N}{2\norm \ddd}<\norm \bbb\leq \frac{N}{\norm \ddd}}\xi(\bbb)f(\norm\ddd\bbb)\right| = \sum_{r\leq N^{2/3}}\frac{a_{K}(r)g_I(r)}{r},
\end{equation}
where
\begin{equation*}
    g_I(r):=|\EE_{\frac{N}{2r}<\norm \bbb \leq \frac{N}{r}}\xi(\bbb)f(r\norm\bbb)|.
  \end{equation*}
The remaining argument is analogous to that for $g_0(r)$ above. First, \eqref{eq:vaughan_I_divisor_estimate} and the Cauchy-Schwarz inequality show that
\[
\sum_{r\leq N^{2/3}}\frac{g_I(r)^2}{r}\gg_K (\epsilon/\log N)^{O_D(1)},
\]
from which we deduce the existence of an integer $1\leq R\leq N^{2/3}$ with
\[
\sum_{R/2< r\leq R} \frac{g_I(r)^2}{r} \gg_K (\epsilon/\log N)^{O_D(1)},\quad\text{ and thus }\quad \sum_{R/2< r\leq R} g_I(r)^2 \gg_K R(\epsilon/\log N)^{O_D(1)}.
\]
As $\leq g_I(r)\leq 1$, this shows that $g_I(r) \gg_K (\epsilon/\log N)^{O_D(1)}$ must hold for at least $R(\epsilon/\log N)^{O_D(1)}$ values of $r$, and thus establishes the second situation in the proposition's conclusion.

Now, we consider the case where $|T_{\II}|\gg N(\log N)\epsilon$. 
As again
\begin{equation*}
  |b_{\ddd}|\leq \sum_{\ccc\mid\ddd}\Lambda_K(\ccc)\leq\log\norm\ccc\leq\log N,
\end{equation*}
we get
\begin{align*}
  N\epsilon&\ll\frac{|T_{\II}|}{\log N}\leq \sum_{N^{1/3}<\norm\ddd\leq N^{2/3}}\left|\sum_{\max\big(N^{1/3},\frac{N}{2\norm \ddd}\big)<\norm \bbb\leq \frac{N}{\norm \ddd}}\mu_K(\bbb)\xi(\bbb)f(\norm\ddd\bbb)\right|\\ &= \sum_{N^{1/3}<r\leq N^{2/3}}a_{R}(r)\Bigg|\sum_{\frac{N}{2r}<\norm \bbb\leq \frac{N}{r}}1_{\norm \bbb>N^{1/3}}\mu_K(\bbb)\xi(\bbb)f(r\norm\bbb)\Bigg|.
\end{align*}
By the Cauchy-Schwarz inequality and \eqref{eq:vaughan_I_divisor_estimate},
 we obtain
\[
\sum_{N^{1/3}<r\leq N^{2/3}}r\Bigg|\sum_{\frac{N}{2r}<\norm \bbb\leq \frac{N}{r}}1_{\norm \bbb>N^{1/3}}\mu_K(\bbb)\xi(\bbb)f(r\norm\bbb)\Bigg|^2 \gg_K N^2(\epsilon/\log N)^{O_D(1)}.
\]
Splitting the intervals for $r$ and $\norm\bbb$ dyadically and applying the pigeonhole principle, we find integers $R,B$ with $N^{1/3}\leq R\leq 2N^{2/3}$ and $N\leq RB\leq 4N$, such that

\begin{equation}
\sum_{R/2<r\leq R}\Bigg|\sum_{B/2<\norm \bbb\leq B}1_{I_{r}}(\norm\bbb)\mu_K(\bbb)\xi(\bbb)f(r\norm\bbb)\Bigg|^2 \gg_K \frac{N^2}{R}(\epsilon/ \log N)^{O_D(1)},\label{eq:vaughan_II_splitting}
\end{equation}
where $I_{r}$ is the set $\Big\{n\in\NN\where n>N^{1/3}\text{ and }\frac{N}{2r}<n\leq \frac{N}{r}\Big\}$. Sorting the ideals $\bbb$ by their norms $n=\norm\bbb$, we see that

\begin{equation*}
\sum_{B/2<\norm \bbb\leq B}1_{I_r}(\norm\bbb)\mu_K(\bbb)\xi(\bbb)f(r\norm\bbb)=\sum_{B/2<n\leq B}1_{I_r}(n)g(r,n),
\end{equation*}
where 
\[
g(r,n):=\sum_{\norm \bbb=n}\mu_K(\bbb)\xi(\bbb)f(rn).
\]
Now \cite[Lemma A.2]{MR2473624} gives 
\begin{align*}
\sum_{R/2<r\leq R}\bigg|\sum_{B/2<n\leq B}1_{I_r}(n)g(n)\bigg|^2&\ll \log^2(1+B/2)\sum_{R/2<r\leq R}\bigg|\sum_{B/2<n\leq B}g(r,n)e(\alpha n)\bigg|^2\\ &\ll (\log N)^2 \sum_{R/2<r\leq R}\bigg|\sum_{B/2<\norm \bbb\leq B}\mu_K(\bbb)\xi(\bbb)f(r\norm\bbb)e(\alpha \norm \bbb)\bigg|^2,
\end{align*}
for some $\alpha\in\RR/\ZZ$, where as usual $e(x):=e^{2\pi i x}$. Inserting this back into \eqref{eq:vaughan_II_splitting}, we obtain 
\[
\sum_{R/2<r\leq R}\bigg|\sum_{B/2<\norm \bbb\leq B}\mu_K(\bbb)\xi(\bbb)f(r\norm\bbb)e(\alpha \norm \bbb)\bigg|^2\gg_K\frac{N^2}{R}(\epsilon/\log N)^{O_D(1)}. 
\]
Expanding the left-hand side, we get
\[
\sum_{B/2<\norm \bbb,\norm \bbb'\leq B}\sum_{R/2<r\leq R}\vb(\bbb,\bbb')f(r\norm\bbb)f(r\norm\bbb'),
\]
with some coefficients $\vb(\bbb,\bbb')\in\CC$ that satisfy $|\vb(\w,\w')|\leq 1$. Applying
\cite[Lemma A.10]{MR2473624} with $x=(\bbb,\bbb')$, $y=r$,
$X=\{(\bbb,\bbb')\where \bbb,\bbb' \in\ideals,\  B/2<\norm \bbb,\norm \bbb'\leq B\}$, $Y=\ZZ\cap (R/2,R]$
and
$h(x,y)=f(r\norm\bbb)f(r\norm\bbb')$,
we obtain
\begin{align*}
|\EE_{x\in X}\EE_{y\in Y}\vb(x)h(x,y)|^2&\ll |\EE_{x\in X}\EE_{y,y'\in Y}h(x,y)\overline{h(x,y')}|\\ &\hspace{-2cm}=|\EE_{B/2<\norm \bbb,\norm \bbb'\leq W}\EE_{R/2<r,r'\leq R}f(r\norm\bbb)f(r\norm\bbb')f(r'\norm\bbb)f(r'\norm\bbb')|.
\end{align*}
As $\frac{N^2}{R^2B^2}\gg 1$, this shows that
\begin{align*}
  &|\EE_{B/2<\norm \bbb,\norm \bbb'\leq B}\EE_{R/2<r,r'\leq R}f(r\norm\bbb)f(r\norm\bbb')f(r'\norm\bbb)f(r'\norm\bbb')|\\
  &\gg_K \frac{1}{B^2R}\left(\sum_{B/2<\norm \bbb,\norm \bbb'\leq B}\sum_{R/2<r\leq R}\vb(\bbb,\bbb')f(r\norm\bbb)f(r\norm\bbb')\right)^2\\
  &\gg_K  \frac{1}{B^2 R}\frac{N^2}{R}(\epsilon/\log N)^{O_D(1)}\gg (\epsilon/\log N)^{O_D(1)}.
\end{align*}
Rewriting this as
\begin{equation*}
  \EE_{R/2<r,r'\leq R}\left|\EE_{B/2<\norm\bbb\leq B}f(r\norm\bbb)f(r'\norm\bbb)\right|^2\gg_K(\epsilon/\log N)^{O_D(1)},
\end{equation*}
we conclude that
\begin{equation*}
  \left|\EE_{B/2<\norm\bbb\leq B}f(r\norm\bbb)f(r'\norm\bbb)\right|\gg_K(\epsilon/\log N)^{O_D(1)}
\end{equation*}
holds for at least $\gg_KR^2(\epsilon/\log N)^{O_D(1)}$ pairs $(r,r')$ with $R/2< r,r'\leq R$,
as desired for the third situation in the proposition's conclusion.
\end{proof}

\subsection{Proof of Theorem \ref{thm:ideal_von_mangoldt_equidist_nilsequences}}
Let us start by proving the following version of Theorem
\ref{thm:ideal_von_mangoldt_equidist_nilsequences}, the only difference being
the lower bound of the range of summation.

\begin{lemma} \label{lem:prop2.1analog} For all integers $m\geq 0$ and
  $d,D\geq 1$, there is a constant $c(m,d,D)>0$, such that the following
  holds. Under the same hypotheses as in Theorem
  \ref{thm:ideal_von_mangoldt_equidist_nilsequences}, we have 
  \begin{equation}\label{eq:lem_prop2.1analog}
  \left|\sum_{N/2\leq \norm \aaa\leq N}\Lambda_K(\aaa)\xi(\aaa)1_P(\norm \aaa)F(g(\norm \aaa)\Gamma)\right|\ll_{m,d,K,\mmm}
  \delta^{c(m,d,D)}Q\vecnorm{F}_{\textnormal{Lip}}N(\log N)^2. 
  \end{equation}
\end{lemma}

\begin{proof}
  We follow the proof of \cite[Proposition 2.1]{MR2877066}, replacing
  the application of \cite[Proposition 3.1]{MR2877066} by our Propositions
  \ref{prop:inverse_theorem_mu} and \ref{prop:norms_vs_equidist_nilsequences}.

Without loss of generality, we may assume that
$\vecnorm{F}_{\textnormal{Lip}}=1$, so, in particular
$\vecnorm{F}_\infty\leq 1$. If $\delta\leq 1/N$ then the total
$\delta$-equidistribution of $(g(n)\Gamma)_{n\in[N]}$ implies
$|F(g(n)\Gamma))|\leq \delta$ for all $n\in[N]$, in which case
\eqref{eq:lem_prop2.1analog} holds trivially for any $c(m,d,D)\leq 1$. Hence, we may
  assume that $\delta>1/N$. This allows us to assume that, in fact,
  $\delta>N^{-\sigma}$, for some small $\sigma=\sigma(m,d,D)\in (0,1)$ that
  will be specified later in the proof. (To deduce from this the result for
  arbitrary $\delta>1/N$, take the result for $\delta^\sigma$ and replace $c$
  by $\sigma c$.)

  Let $\epsilon := (C\delta)^cQ\log N$, where the small $c=c(m,d,D)\in(0,1)$ and the large $C=C(m,d,K,\mmm)\geq 1$ will be specified later in the proof. We may assume that $\epsilon<1$, as otherwise \eqref{eq:prop2.1analog} holds trivially.  
In particular, we may thus assume that $Q,\log N\leq\delta^{-c}$, and we have $\epsilon/\log N\geq (C\delta)^c$.

Under the above assumptions, we suppose that \eqref{eq:lem_prop2.1analog} does not
hold with $C^c$ as the implied constant, so
  \begin{equation*}
    \left|\sum_{N/2\leq \norm \aaa\leq N}\Lambda_K(\aaa)\xi(\aaa)1_P(\norm
    \aaa)F(g(\norm \aaa)\Gamma)\right|\geq \epsilon N \log N.
  \end{equation*}
  We will deduce that $(g(n)\Gamma)_{n\in[N]}$ is not totally
  $\delta$-equidistributed, contradicting the hypotheses of Theorem
  \ref{thm:ideal_von_mangoldt_equidist_nilsequences}. To this end, we
  apply Proposition \ref{prop:inverse_theorem_mu} with $f(n)=\one_P(n)F(g(n)\Gamma)$, leading to one of three cases.

We treat both cases of large type I sums simultaneously. Let $e=1$ in the weighted
and $e=0$ in the unweighted case. In both cases, we get some $R\ll N^{2/3}$ and $\gg_K R(C\delta)^{O_{D}(c)}$ integers $r\in (R/2,R]$, such that
  \begin{equation*}
    \left|\sum_{N/2r< \norm\bbb\leq
        N/r}(\log\norm\bbb)^e\xi(\bbb)1_P(r\norm\bbb)F(g(r\norm\bbb)\Gamma)\right|\gg_K \frac{N}{r}(C\delta)^{O_{D}(c)}.
  \end{equation*}
  Fix a value of $r$. Let $l$ be the common difference of the progression $P$, then $1\leq l\ll
  Q$. Pigeonholing the values of $\norm\bbb$ into residue classes modulo $l$, we
  find some $b\bmod l$, for which
  \begin{equation*}
    \left|\sum_{\substack{N/2r<\norm\bbb\leq N/r\\\norm\bbb\equiv b\bmod l}}(\log\norm\bbb)^e\xi(\bbb)1_P(r\norm\bbb)F(g(r\norm\bbb)\Gamma)\right|\gg_K\frac{N}{rl}(C\delta)^{O_{D}(c)}.
  \end{equation*}
  The conditions $n\in(N/2r,N/r]$, 
  $n\equiv b\bmod l$ and $\one_P(rn)=1$ are equivalent to $n\in P_r$, for some
  progression $P_r \subseteq(N/2r,N/r]$ of length $\asymp |P|/r\geq
  N/(Qr) \geq \delta^{c}N/r$.
  Hence, we obtain
  \begin{equation}\label{eq:moebius_equidist_type_I_norm_bound}
     \left|\sum_{\substack{\norm\bbb\in P_r}}(\log\norm\bbb)^e\xi(\bbb)F(g(r\norm\bbb)\Gamma)\right|\gg_K\frac{N}{rl}(C\delta)^{O_{D}(c)}
        = \frac{N}{r}(C\delta)^{O_{D}(c)}.
  \end{equation}

      Set $c_1:=\sqrt{c}$, $M_r:=\lceil N/r\rceil$, and define the polynomial sequence $g_r(n):=g(rn)\in\poly(\ZZ,G_\bullet)$.
      We claim that the finite sequence
      $(g_r(n))_{n\in [M_r]}$ is not totally
      $\delta^{c_1}$-equidistributed. Indeed, if it was, then Proposition
      \ref{prop:norms_vs_equidist_nilsequences} (in the unweighted case) or Corollary \ref{cor:norms_vs_equidist_nilsequences_log} (in the weighted case), applied with some appropriate $Q\leq\tilde Q\ll Q$ in place of $Q$, would imply
      \begin{equation*}
        \left|\sum_{\norm\bbb\in P_r}(\log\norm\bbb)^e\xi(\bbb)F(g(r\norm\bbb)\Gamma)\right|\ll_{K,\mmm} M_r\delta^{c_1/O_{d,D,m}(1)}Q(\log M_r)^e\ll \frac{N}{r}\delta^{c_1/O_{d,D,m}(1)-O(c)}.
      \end{equation*}
      This contradicts \eqref{eq:moebius_equidist_type_I_norm_bound} if only 
      $c(m,d,D)$ is small enough and $C(m,d,K,\mmm)$ is large enough.

      Hence, there is a progression
      $\tilde P_r=\{a+q,a+2q,\ldots,a+N_rq\}\subseteq [M_r]$ of length
      $N_r\geq \delta^{c_1}M_r$, such that the sequence
      $(g_r(a+nq)\Gamma)_{n\in [N_r]}$ is not
      $\delta^{c_1}$-equidistributed. Define the polynomial sequence $\tilde{g}_r(n):=g_r(a+nq)$. By the
      quantitative Leibmann theorem of Green and Tao \cite[Theorem
      2.9]{MR2877065}, there exists a nontrivial horizontal character $\psi_r$ with
      $|\psi_r|\ll\delta^{-O_{m,d}(c_1)}$, such that
      \begin{equation*}
      \vecnorm{\psi_r\circ\tilde
        g_r}_{C^\infty[N_r]}\ll\delta^{-O_{m,d}(c_1)}.
      \end{equation*}
      Note that $\tilde P_r\subseteq [M_r]$ and $N_r\geq \delta^{c_1}M_r$ imply
      that $q\ll \delta^{-c_1}$ and $|a|\ll \delta^{-c_1}N_r$. Hence, Lemma \ref{lem:polynomial_extrapolation} with $D=1$ implies
      the existence of a positive integer $q_r\ll\delta^{-O_{m,d}(c_1)}$, such
      that 
      \begin{equation}\label{eq:equidist_case_gr_character}
         \vecnorm{q_r\psi_r\circ g_r}_{C^\infty[N_r]}\ll\delta^{-O_{m,d}(c_1)}.
      \end{equation}
      
      Recall from above (with large enough $C$), that this
      holds for $\gg R\delta^{O_{D}(c)}\geq R\delta^{O_{D}(c_1)}$ values of $r$ in
      $(R/2,R]$. If $R\ll \delta^{-O_{D}(c_1)}$,
      \eqref{eq:equidist_case_gr_character} for any such $r$ implies via
      another application of Lemma \ref{lem:polynomial_extrapolation} that
      $\vecnorm{\psi\circ g}_{\Cinf[N]}\ll\delta^{-O_{m,d,D}(c_1)}$ for some
      nontrivial horizontal character $\psi$ with $\abs{\psi}\ll \delta^{-O_{m,d,D}(c_1)}$. For
      sufficiently small $c_1$, this implies that the sequence $(g(n)\Gamma)_{n\in[N]}$ is not
      totally $\delta$-equidistributed by Lemma \ref{lem:characters_obstruct},
      a contradiction.

      Hence, we may assume that $R\delta^{O_D(c_1)}\gg 1$ and, replacing $R$ by $2\lfloor R/2\rfloor$, that $R$ is
      even. Now the arguments in \cite[pp.550--551]{MR2877066}
      apply verbatim with $K = R/2$ in their notation, leading to
      the same contradiction.
      
The type II case can be handled by analogous modifications to the arguments in
\cite{MR2877066}. Thus, we only provide a sketch of its proof. In this case of
Proposition \ref{prop:inverse_theorem_mu}, there are integers $R,B$ with
$N^{1/3}/2\leq R\leq 2 N^{2/3}$, $\frac{N}{2}\leq RB\leq 2N$, such that 
\begin{equation*}
    \left|\sum_{B/2<\norm\bbb\leq B}
    \one_P(r\norm \bbb)\one_P(r'\norm \bbb)F(g(r\norm \bbb)\Gamma)F(g(r'\norm \bbb)\Gamma)\right|\gg_K B(C\delta)^{O_{D}(c)}
\end{equation*}
for $\gg_K R^2(C\delta)^{O_{D}(c)}$ pairs of integers $(r,r')$ with $R/2<r,r'\leq R$.

As before, let $l$ be the common difference of $P$. Then, we can find a congruence class $b\bmod l$ so that 
\begin{equation*}
    \left|\sum_{\substack{B/2<\norm\bbb\leq B\\\norm\bbb \equiv b\bmod l}}
    \one_P(r\norm \bbb)\one_P(r'\norm \bbb)F(g(r\norm \bbb)\Gamma)F(g(r'\norm \bbb)\Gamma)\right|\gg_K \frac{B}{l}(C\delta)^{O_{D}(c)}.
\end{equation*}
The conditions $n\in(B/2,B]$, $n\equiv b\bmod l$, $\one_P(rn)=1$ and $\one_P(r'n)=1$ give an arithmetic progression $P_{r,r'}\subseteq (B/2,B]$ of length $\asymp \frac{B}{Q}\geq \delta^{c}B$, for which
\begin{equation*}
    \left|\sum_{\norm \bbb \in P_{r,r'}} F(g(r\norm \bbb)\Gamma)F(g(r'\norm \bbb)\Gamma)\right|\gg_K \frac{B}{l}(C\delta)^{O_{D}(c)}=B(C\delta)^{O_{D}(c)}.
\end{equation*}
Put $c_1=\sqrt{c}$, and consider the polynomial sequence
$g_{r,r'}(n)=(g(rn),g(r'n))\in\poly(\ZZ,G_\bullet\times G_\bullet)$. Using
Proposition \ref{prop:norms_vs_equidist_nilsequences} similarly as before, with
$\xi=1$ and some appropriate $1\leq \tilde Q\ll Q$ in place of $Q$, we see that
$(g_{r,r'}(n)(\Gamma\times\Gamma))_{n\in [B]}$ cannot be totally
$\delta^{c_1}$-equidistributed in $G/\Gamma\times G/\Gamma$, if only $c$ is
sufficiently small and $C$ sufficiently large.

We can therefore find an arithmetic progression
$\tilde{P}_{r,r'}=\{a+q,\ldots,a+N_{r,r'}q\}\subseteq[B]$ with length
$N_{r,r'}\geq \delta^{c_1}B$ so that
$(g_{r,r'}(a+nq)(\Gamma\times\Gamma))_{n\in [N_{r,r'}]}$ fails to be totally
$\delta^{c_1}$-equidistributed. With $\tilde{g}_{r,r'}(n):=g_{r,r'}(a+nq)$, the
      quantitative Leibmann theorem \cite[Theorem
      2.9]{MR2877065} yields the existence of a nontrivial horizontal character $\psi_{r,r'}$ with
      $|\psi_{r,r'}|\ll\delta^{-O_{m,d}(c_1)}$, such that
      \begin{equation*}
      \vecnorm{\psi_{r,r'}\circ\tilde
        g_{r,r'}}_{C^\infty[N_{r,r'}]}\ll\delta^{-O_{m,d}(c_1)}.
    \end{equation*}
       As $\tilde P_{r,r'}\subseteq [B]$ and $N_{r,r'}\geq \delta^{c_1}B$, we
       get $q\ll \delta^{-c_1}$ and $|a|\ll \delta^{-c_1}N_{r,r'}$. Hence, Lemma \ref{lem:polynomial_extrapolation} with $D=1$ implies
      the existence of a positive integer $q_{r,r'}\ll\delta^{-O_{m,d}(c_1)}$, such
      that 
      \begin{equation*}
         \vecnorm{q_{r,r'}\psi_{r,r'}\circ g_{r,r'}}_{C^\infty[N_{r,r'}]}\ll\delta^{-O_{m,d}(c_1)}.
      \end{equation*}
      If $C$ is large enough, this holds for
      $\gg R^2\delta^{O_{D}(c)}\geq R^2\delta^{O_{D}(c_1)}$ values of $(r,r')$
      with $R/2<r,r'\leq R$. As $R\gg N^{1/3}$, we have $R^2\delta^{O_D(c_1)}\gg
      R$, and hence we may assume that $R$ is even. From here on, the proof in \cite[pp.553]{MR2877066}
      applies verbatim with $K = R/2$.
    \end{proof}

    Let us now deduce Theorem
    \ref{thm:ideal_von_mangoldt_equidist_nilsequences} from Lemma
    \ref{lem:prop2.1analog}.  We may make the same assumptions without loss of
    generality as at the start of the proof of Lemma \ref{lem:prop2.1analog},
    in particular we may assume that $\vecnorm{F}_{\Lip}=1$ and thus
    $\vecnorm{F}_\infty\leq 1$, as well as $\delta> N^{-1/4}$. Write
    $s_\aaa= \Lambda_K(\aaa)\xi(\aaa)1_P(\norm \aaa)F(g(\norm \aaa)\Gamma)$,
    $x_l:=2^{-l}N$ and $N_l:=\lceil x_l\rceil$. Then
      \begin{align}
        \sum_{\norm\aaa\leq N}s_\aaa &= \sum_{l=0}^{\lfloor -(1/2)\log_2\delta\rfloor-1}\left|\sum_{x_l/2<\norm\aaa\leq x_l}s_\aaa\right| + O\left(\sum_{\norm\aaa\leq 2\delta^{1/2}N}|s_\aaa|\right)\nonumber\\
        &=\sum_{l=0}^{\lfloor -(1/2)\log_2\delta\rfloor-1}\left|\sum_{N_l/2<\norm\aaa\leq N_l}s_\aaa\right| + O_K\left((\log N)|\log\delta|+\delta^{1/2}N\log N\right).\label{eq:dyadic_splitting_plus_error}
      \end{align}
      Let $0\leq l\leq \lfloor -(1/2)\log_2\delta\rfloor-1$, let $R\subseteq [N_l]\subseteq [N]$ be an arithmetic progression of length at least $\delta^{1/2}N_l$. As $\delta^{1/2}N_l \geq \delta^{1/2}x_l\geq \delta N$, the total $\delta$-equidistribution of $(g(n)\Gamma)_{n\in[N]}$ yields
      \begin{equation*}
        \left|\sum_{n\in R}H(g(n)\Gamma)-\int_{G/\Gamma}H\right|\leq \delta\vecnorm{H}_{\Lip}\leq\delta^{1/2}\vecnorm{H}_{\Lip}
      \end{equation*}
      for any Lipschitz function $H:G/\Gamma\to\CC$. Hence, the sequence $(g(n)\Gamma)_{n\in [N_l]}$ is totally $\delta^{1/2}$-equidistributed.

      Let $c=c(m,d,D)$ be the constant from Lemma \ref{lem:prop2.1analog}. For our given progression $P\subseteq [N]$, we distinguish between two cases.

      Firstly, if $|P\cap [N_l]|< \delta^{c/4}N_l/Q$, we may trivially estimate
      \begin{equation*}
        \sum_{N_l/2<\norm\aaa\leq N_l}s_\aaa \leq (\log N)\sum_{\norm\aaa\leq N_l}1_P(\norm\aaa)\ll_K \delta^{c/4}N_l(\log N)/Q \leq \delta^{c/4}QN_l(\log N)^2.
      \end{equation*}

      Secondly, if $|P\cap [N_l]|\geq \delta^{c/4}N_l/Q$, then we apply Lemma \ref{lem:prop2.1analog} with $\delta^{1/2}$ in place of $\delta$ and $\delta^{-c/4}Q$ instead of $Q$ to obtain
     \begin{equation*}
        \sum_{N_l/2<\norm\aaa\leq N_l}s_\aaa \ll_{m,d,K,\mmm} \delta^{c/2-c/4}Q N_l(\log N)^2 = \delta^{c/4}Q N_l(\log N)^2.
      \end{equation*}
      Hence, summing over all $l$ we may estimate
      \begin{equation*}
        \sum_{l=0}^{\lfloor -(1/2)\log_2\delta\rfloor-1}\left|\sum_{N_l/2<\norm\aaa\leq N_l}s_\aaa\right|\ll_{m,d,K,\mmm}\delta^{c/4}Q N (\log N)^2\sum_{l=0}^{\infty}2^{-l} \ll \delta^{c/4}Q N (\log N)^2.
      \end{equation*}
      Together with \eqref{eq:dyadic_splitting_plus_error}, the ideal theorem, and our assumptions $\vecnorm{F}_{\Lip}=1$ and $\delta> N^{-1/4}$ made at the start of this proof, this shows the bound stated in Theorem \ref{thm:ideal_von_mangoldt_equidist_nilsequences}, once we replace $c$ by $c/4$.

\section{Von Mangoldt model: proof of Theorem \ref{thm:W_tricked_chebotarev_mangoldt_with_nilsequences_nonconst_avg}}\label{sec:von_mangoldt_proxy}
In this section, we prove Theorem \ref{thm:W_tricked_chebotarev_mangoldt_with_nilsequences_nonconst_avg}.
As in \cite{MR2877066}, we deduce the general result from the following version in the case of equidistributed nilsequences and test functions with mean zero. This result is analogous to \cite[Proposition 2.1]{MR2877066}, with the M\"obius function replaced by our $W$-tricked Chebotarev-von Mangoldt function $\Lambda_{K,C,b,W}(n)$. 

\begin{proposition}[Non-correlation with equidistributed
  nilsequences]\label{prop:equidist_nilsequences}
  For all integers $m\geq 0$, $d,D\geq 1$, there is a constant
  $c(m,d,D)>0$ such that the following holds.

  Let $N\in\NN$ be
  sufficiently large in terms of $m,d$. Let $\delta\in(0,1)$ and $Q\geq 2$. Let $G/\Gamma$ be an $m$-dimensional nilmanifold with a filtration $G_\bullet$ of degree $d$ and a $Q$-rational Mal'cev basis $\cX$ adapted to
  $G_\bullet$. Let $g\in\poly(\ZZ,G_\bullet)$, and suppose that
  $(g(n)\Gamma)_{n\in [N]}$ is totally $\delta$-equidistributed. Let
  $F:G/\Gamma\to [-1,1]$ with $\int_{G/\Gamma}F=0$ and $P\subseteq [N]$ an
  arithmetic progression of size $|P|\geq N/Q$. Let $K$ be a Galois number field of degree $[K:\QQ]=D$ with $\Phi_{K^\ab}\mid W$, and $C\subseteq\Gal(K/\QQ)$ a conjugacy class. Then 
  \begin{equation*}
    \left|\EE_{n\in [N]}\Lambda_{K,C,b,W}(n)\one_P(n)F(g(n)\Gamma)\right|\ll_{K,m,d,\cD}\delta^{c(m,d,D)}Q\vecnorm{F}_{\Lip}(\log N)^4. 
  \end{equation*}
\end{proposition}

We will deduce Proposition \ref{prop:equidist_nilsequences} from Theorem
\ref{thm:ideal_von_mangoldt_equidist_nilsequences}. In order to do so, we
require some preparation.

\subsection{Passing to ideals}
  For our Galois number
field $K$ and conjugacy class $C\subseteq\Gal(K/\QQ)$, fix in addition
an element $c\in C$. We consider the fixed field $L=K^c$. Then $K/L$
is cyclic with Galois group $\Gal(L/K)=\langle c\rangle$. We consider
every character $\xi:\langle c\rangle\to S^1$ as a Hecke character of
$L$ via $\xi(\ppp)=\xi([K/L, \ppp])$ for all prime ideals $\ppp$ of
$L$ unramified in $K/L$. The following lemma and its proof are inspired by \cite[Proposition 6]{MR3060874}.

  \begin{lemma}\label{lem:passing_to_ideals}
    Let $F:\NN\to\CC$ be any $1$-bounded function. For any $q,N\in\NN$ and $b\in\ZZ$ we have 
  \begin{equation*}
    \sum_{\substack{n\leq N\\n\equiv b\bmod q}}\Lambda_{K,C}(n)F(n) = \frac{|C|}{\phi(q)[K:\QQ]}\sum_{\chi\bmod q}\sum_{\xi\in\widehat{\langle c\rangle}}\overline{\chi(b)\xi(c)}\sum_{\substack{\aaa\in\idealsL\\\norm\aaa\leq N}}\xi(\aaa)\chi(\norm\aaa)\Lambda_L(\aaa)F(\norm\aaa) + O_K(\sqrt{N}), 
  \end{equation*}
  where $\chi$ runs through all Dirichlet characters modulo $q$ and $\xi$ through all characters of $\langle c\rangle$.
\end{lemma}

\begin{proof}
    The claim clearly holds if $(b,q)\neq 1$, hence we assume now that
  $(b,q)=1$. Using character orthogonality and the fact that the contribution
  of proper prime powers is negligible, we see that
  \begin{align*}
     \sum_{\substack{n\leq N\\n\equiv b\bmod q}}\Lambda_{K,C}(n)F(n) &=
     \frac{1}{\phi(q)}\sum_{\chi\bmod q}\overline{\chi(b)}\sum_{n\leq
                                                                   N}\chi(n)\Lambda_{K,C}(n)F(n)\\
    &=     \frac{1}{\phi(q)}\sum_{\chi\bmod
      q}\overline{\chi(b)}\sum_{\substack{p\leq N\\ p\nmid\Delta_K\\ [K/\QQ,p]=C}}\chi(p)(\log p) F(p)
      + O(\sqrt{N}).                                                                                       
  \end{align*}
  Let $p$ be a prime unramified in $K$ with
  $[K/\QQ,p]=C$. The Galois group $\Gal(K/\QQ)$ acts transitively on the prime
  ideals of $\OO_K$ lying above $p$, the stabilisers being the
  decomposition groups, which have size $|c|$. Hence, there are exactly $[K:\QQ]/|c|$ prime
  ideals $\qqq$ of $\OO_K$ above $p$, of which exactly $[K:\QQ]/(|c|\cdot |C|)$ satisfy
  $[K/\QQ, \qqq]=c$.

  Each such $\qqq$ has decomposition group
  $\langle c\rangle =\Gal(K/L)\subseteq\Gal(K/\QQ)$, so it is the
  only prime ideal of $\OO_K$ lying above $\ppp=\qqq\cap\OO_L$. Hence, the
  inertia degrees satisfy $f(\qqq/p)=|c|=[K:L]=f(\qqq/\ppp)$, and thus
  $f(\ppp/p)=1$. Therefore,
  \begin{equation*}
    c=[K/\QQ,\qqq] = [K/\QQ,\qqq]^{f(\ppp/p)} = [K/L,\qqq] = [K/L,\ppp].
  \end{equation*}
  
  In summary, passing from primes $p$ to prime ideals $\ppp$ of $\OO_L$, we get
  for each Dirichlet character  $\chi$ modulo $q$,
  \begin{equation}\label{eq:siegel_walfisz_passed_to_L}
    \sum_{\substack{p\leq N\\ p\nmid\Delta_K\\ [K/\QQ,p]=C}}\chi(p)(\log p)F(p)
    =\frac{|C|\cdot|c|}{[K:\QQ]}\sum_{\substack{\ppp\in\idealsL\\\norm\ppp\leq
    N\\ [K/L,\ppp]=c\\ \norm\ppp\text{ prime}}}\chi(\norm\ppp)(\log\norm\ppp)F(\norm\ppp).
  \end{equation}
  Using character orthogonality for $\langle c\rangle$ and
  the facts that prime ideals of higher inertia degree and higher powers of
  prime ideals are irrelevant when counting by norm, the right-hand side of
  \eqref{eq:siegel_walfisz_passed_to_L} becomes
  \begin{align*}
    &\frac{|C|}{[K:\QQ]}\sum_{\substack{\ppp\in\idealsL\\\norm\ppp\leq
    N\\ \norm\ppp\text{ prime}}}\sum_{\xi\in\widehat{\langle c\rangle}}\overline{\xi(c)}\xi(\ppp)\chi(\norm\ppp)(\log\norm\ppp)F(\norm\ppp)\\ &=\frac{|C|}{[K:\QQ]}\sum_{\xi\in\widehat{\langle c\rangle}}\overline{\xi(c)}\sum_{\substack{\aaa\in\idealsL\\\norm\aaa\leq
    N}}\xi(\aaa)\chi(\norm\aaa)\Lambda_{L}(\aaa)F(\norm\aaa)+O_K(\sqrt{N}).\qedhere
  \end{align*}
\end{proof}

\subsection{Proof of Proposition \ref{prop:equidist_nilsequences}}
Without loss of generality, we may assume that
$\vecnorm{F}_{\textnormal{Lip}}=1$, so in particular
$\vecnorm{F}_\infty\leq 1$. If $\delta\leq 1/N$ then the total
$\delta$-equidistribution of $(g(n)\Gamma)_{n\in[N]}$ implies
$|F(g(n)\Gamma))|\leq \delta$ for all $n\in[N]$, in which case
the bound in Proposition \ref{prop:equidist_nilsequences} holds trivially for any $c(m,d,D)\leq 1$. Hence, we may
  assume that $\delta>1/N$. This allows us to assume that, in fact,
  $\delta>N^{-1/4}$: to deduce from this the result for
  arbitrary $1>\delta>1/N$, take the result for $\delta^{1/4}$ and replace $c$
  by $c/4$.

Using $\Lambda_{K,C,b,W}(n)=\frac{\phi(W)}{W}\Lambda_{K,C}(Wn+b)$, we obtain
\begin{equation}\label{eq:proof_equidist_sum}
\sum_{n\in [N]}\Lambda_{K,C,b,W}(n)\one_P(n)F(g(n)\Gamma)=\frac{\phi(W)}{W}\sum_{n\in [WN+b]}\Lambda_{K,C}(n)\one_{\tilde{P}}(n)F\left(g\left(\frac{n-b}{W}\right)\Gamma\right)
\end{equation}
where $\tilde{P}=WP+b$ is a progression of size at least
$\frac{N}{Q}=\frac{WN+b}{\tilde{Q}}$ with $\tilde{Q}=QW+Qb/N\leq
2QW$.

Write $\tilde N=WN+b$ and let $c(m,d)$ and $\tilde g\in\poly(\ZZ,G_\bullet)$ be the
constant and polynomial sequence from Proposition \ref{prop:spread_nilsequence}
with $q=W$, so that 
\begin{equation*}
\tilde g(n)=g\left(\frac{n-b}{W}\right)\quad \text{ whenever }\quad n\equiv b\bmod W.
\end{equation*}
In the statement of Proposition \ref{prop:equidist_nilsequences}, we may assume that  
$\delta^{-c(m,d)}\geq Q\log N\geq 2$, as otherwise the
conclusion is trivial. Hence, the Mal'cev basis is also
$\delta^{-c(m,d)}$-rational and moreover we have $1\leq b\leq W\leq \log N\leq
\delta^{-c(m,d)}$. As moreover $\delta>N^{-1/4}>\tilde N^{-1/2}$,  we conclude from Proposition
\ref{prop:spread_nilsequence} that the sequence
$(\tilde g(n))_{n\in\tilde N}$ is totally $\delta^{c(m,d)}$-equidistributed.

 Using Lemma \ref{lem:passing_to_ideals} with $q=1$, we can
write the sum in \eqref{eq:proof_equidist_sum} as
\begin{multline*}
\frac{\phi(W)}{W}\sum_{n\in[\tilde N]}\Lambda_{K,C}(n)\one_{\tilde P}(n)F(\tilde
g(n)\Gamma)\\=\frac{\phi(W)|C|}{W[K:\QQ]}\sum_{\xi\in\widehat{\langle
    c\rangle}}\overline{\xi(c)}\sum_{\substack{\aaa\in\idealsL\\\norm\aaa\leq
    \tilde N}}\xi(\aaa)\Lambda_L(\aaa)\one_{\tilde{P}}(\norm \aaa)F(\tilde
g(\norm \aaa)\Gamma) + O_K(\phi(W)\sqrt{ N}).
\end{multline*}
Recall that here $L=K^c$ for some fixed $c\in C$. By Theorem
\ref{thm:ideal_von_mangoldt_equidist_nilsequences}, with $K=L$, $g=\tilde g$, $P=\tilde P$, $N=\tilde N$,
$Q=\tilde Q$ and $\delta^{c(m,d)}$ instead of $\delta$, the sum over ideals $\aaa$ with $\norm
\aaa\leq \tilde N$ can be bounded by
\begin{equation*}
\ll_{m,d,L}
\delta^{c(m,d)c(m,d,D_L)}QW\vecnorm{F}_{\Lip}W(N+1)\log(W(N+1))^2\ll_\cD
\delta^{c(m,d)c(m,d,D_L)}Q\vecnorm{F}_{\Lip} N(\log N)^4,
\end{equation*}
where $D_L=[L:\QQ]$. This is enough to conclude our proof of Proposition \ref{prop:equidist_nilsequences}. 
\subsection{Equidistribution in progressions}

In order to show that the $W$-tricked classical von Mangoldt function $\Lambda_{b,W}$ provides a good model of our $\Lambda_{K,C,b,W}$ in arithmetic progressions, we need the following version of Siegel-Walfisz for $\Lambda_{K,C}$.

\begin{proposition}\label{prop:better_than_chebotarev}
  Let $K$ be a Galois number field and $C\subseteq\Gal(K/\QQ)$ a
  conjugacy class. Let $N,A>0$, $q\in\NN$ and $b\in\ZZ$. Then
  \begin{equation*}
    \sum_{\substack{n\leq N\\n\equiv b\bmod q}}\Lambda_{K,C}(n) = \eta_{K,C}(b,q)N + O_{K,A}\left(N(\log N)^{-A}\right). 
  \end{equation*}
  The implied constant is ineffective.
\end{proposition}

This does not follow directly from effective versions of the Chebotarev density theorem (\cite{MR447191,MR3981313}) applied to $K(\mu_q)$, as the degree $[K(\mu_q):\QQ]$ can be too large in terms of $N$. Instead, we adapt techniques of Kane to our situation. We will use the following lemma, inspired by \cite[Corollary 4 and Corollary 5]{MR3060874}, to identify the main term. We identify $(\ZZ/q\ZZ)^\times\cong\Gal(\QQ(\mu_q)/\QQ)$ via the Artin symbol $\sigma_a=[\QQ(\mu_q)/\QQ,a]$ and may thus identify characters $\chi$ with characters $\chi'$ of $\Gal(\QQ(\mu_q)/\QQ)$ via
\begin{equation*}
\chi(a+q\ZZ) = \chi'([\QQ(\mu_q)/\QQ,a]).
\end{equation*}
\begin{lemma}\label{lem:kane_cors}
  Let $K/\QQ$ be a Galois number field and $L\subseteq K$ a subfield such that $\Gal(K/L)$ is abelian. Let $\xi\in\Gal(K/L)^\wedge$ be a character. Let $q\in\NN$. Then the following assertions are equivalent.
  \begin{enumerate}
  \item $\xi$ is trivial on the kernel of the restriction
      \begin{equation*}
         \Gal(K/L)\to\Gal(K\cap\QQ(\mu_q)/\QQ).
       \end{equation*}
  \item $\xi$ extends to a character of $\Gal(K/\QQ)$ that is trivial on $\Gal(K/K\cap\QQ(\mu_q))$.
  \item There is a Dirichlet character $\psi\bmod q$ with $\psi'$ trivial on $\Gal(\QQ(\mu_q)/K\cap\QQ(\mu_q))$, such that
    $\psi(p)=\xi([K/\QQ, p])$ for all primes $p$ not dividing $q\Delta_K$.
  \item There are Dirichlet characters $\chi\bmod q$ such that $\xi([K/L,\ppp])=\chi(\norm\ppp)$ for all prime ideals $\ppp$ of $\OO_L$ not dividing $q\Delta_K$.
  \end{enumerate}
  If these conditions hold, then the Dirichlet characters $\chi\bmod q$ in \emph{(4)} are exactly the characters of the form $\chi=\psi\rho$, with $\psi$ the character from \emph{(3)} and $\rho'$ trivial on $\Gal(\QQ(\mu_q)/L\cap\QQ(\mu_q))$.
\end{lemma}

\begin{proof}
  We write $K_q^\ab:=K\cap\QQ(\mu_q)$ and start with $\emph{(1)}\Rightarrow\emph{(2)}$. 
  If \emph{(1)} holds, we can consider $\xi$ as a character on the image of $\Gal(K/L)$ in $\Gal(K_q^\ab/\QQ)$, which is abelian, and thus extend it to a character on $\Gal(K_q^\ab/\QQ)$. Via restriction $\Gal(K/\QQ)\to \Gal(K_q^\ab/\QQ)$, we may view this extension as a character on $\Gal(K/\QQ)$ trivial on $\Gal(K/K_q^\ab)$, which shows \emph{(2)}.
  
 The implication $\emph{(2)}\Rightarrow\emph{(1)}$ is trivial. To prove $\emph{(2)}\Rightarrow\emph{(3)}$, we consider $\xi$ as a character of $\Gal(K_q^\ab/\QQ)$, so by restriction it induces a character $\psi'$ of $\Gal(\QQ(\mu_q)/\QQ)$ that is trivial on $\Gal(\QQ(\mu_q)/K_q^\ab)$. For the corresponding Dirichlet character $\psi$ and every prime $p\nmid q\Delta_K$, we then have
  \begin{equation*}
    \psi(p) = \psi'([\QQ(\mu_q)/\QQ,p]) = \xi([K_q^\ab/\QQ,p])=\xi([K/\QQ,p]),
  \end{equation*}
  which shows \emph{(3)}.

  For $\emph{(3)}\Rightarrow\emph{(4)}$, we show that $\psi$ is such a Dirichlet character. Let $\ppp$ be a prime ideal of $\OO_L$ not dividing $q\Delta_K$, and write $\norm\ppp=p^f$ for some prime $p\nmid q\Delta_K$ and $f\in \NN$. By the formalism for the Artin symbol,
  \begin{equation*}
    \psi(\norm\ppp)=\psi(p)^f=\xi([K/\QQ,p])^f=\xi([K/\QQ,p]^f)=\xi([K/L,\ppp]),
  \end{equation*}
  as desired.

  For $\emph{(4)}\Rightarrow\emph{(2)}$, let $\chi$ be a Dirichlet character as in \emph{(3)}. Let $M\subseteq\QQ(\mu_q)$ be the fixed field of the kernel of $\chi'$, so we can consider $\chi'$ as an injective character of $\Gal(M/\QQ)$. Consider the compositum $F=KM$. Via restriction $\Gal(F/\QQ)\to\Gal(M/\QQ)$, we may now consider $\chi'$ as a character of $\Gal(F/\QQ)$ that is trivial on $\Gal(F/M)$. For any prime ideal $\ppp$ of $L$ not dividing $q\Delta_K$ with $\norm\ppp=p^f$, we have
  \begin{equation*}
    \xi([K/L,\ppp])=\chi(p)^f=\chi'([M/\QQ,p])^f=\chi'([F/\QQ,p])^f=\chi'([F/L,\ppp]).
  \end{equation*}
  As $[F/L,\ppp]$ hits all conjugacy classes of $\Gal(F/L)$ by Chebotarev, we see that $\chi'(\sigma)=\chi(\sigma|_K)$ for all $\sigma\in\Gal(F/L)$,
  so $\chi'$ is trivial on $\Gal(F/K)=\Gal(M/K\cap M)$. As $\chi'$ is injective on $\Gal(M/\QQ)$, this shows that $K\cap M=M$, and thus $F=K$. Therefore, $\chi'$ extends $\xi$ to $\Gal(K/\QQ)$. As $\chi'$ is trivial on $\Gal(K/M)$, it is in particular trivial on $\Gal(K/K\cap\QQ(\mu_q))$, which shows \emph{(2)}.

  Now suppose that \emph{(1)}--\emph{(4)} hold. For the additional assertion, let $\rho$ be a Dirichlet character modulo $q$. We need to show that $\rho(\norm\ppp)=1$ for all $\ppp\nmid q\Delta_K$ if and only if $\rho'$ is trivial on $\Gal(\QQ(\mu_q)/L\cap\QQ(\mu_q))$. This is similar to the proof of $\emph{(4)}\Rightarrow\emph{(2)}$ above and is shown in \cite[Corollary 5]{MR3060874}.
\end{proof}

\begin{proof}[Proof of Proposition \ref{prop:better_than_chebotarev}]
  If $q>(\log N)^{A+1}$, the statement is trivial due to \eqref{eq:phi_eta_bound}. Hence, let us assume that
  $q\leq (\log N)^{A+1}$. We start from the formula in Lemma \ref{lem:passing_to_ideals} with
  $F=1$. The inner sum over ideals $\aaa$ of $\OO_L$ on the right-hand
  side is just $F_{L,\chi\xi,N}(0)$ in the notation of
  \cite[Definition 4]{MR3060874}. Hence, \cite[Proposition 14]{MR3060874} yields the
  estimate
  \begin{equation}\label{eq:s_w_estimate_with_siegel}
    \sum_{\norm\aaa\leq N}\xi(\aaa)\chi(\norm\aaa)\Lambda_{L}(\aaa) = r_{\chi\xi} N + O(X\exp(-c_{L,\xi,A}\sqrt{\log N})),
  \end{equation}
  where $r_{\chi\xi}$ is the order of the pole at $1$ of the Hecke $L$-function of the Hecke character given by $(\chi\circ\norm)\xi$ and $c_{L,\xi,A}$ is ineffective. We crudely estimate
  \begin{equation*}
    \exp(-c_{L,\xi,A}\sqrt{\log N}) \ll_{K,A}(\log N)^{-A},
  \end{equation*}
  and thus
  \begin{equation}\label{eq:siegel_walfisz_main_term}
    \frac{1}{N}\sum_{\substack{n\leq N\\n\equiv b\bmod q}}\Lambda_{K,C}(n) =
    \frac{|C|}{\phi(q)[K:\QQ]}\sum_{\chi\bmod q}\sum_{\xi\in\widehat{\langle c\rangle}}r_{\chi\xi}\overline{\chi(b)\xi(c)} + O_{K,A}\left((\log N)^{-A}\right).
  \end{equation}

  We apply Lemma \ref{lem:kane_cors} to analyse the main term. To this end, we write $H:=\langle c\rangle=\Gal(K/L)$ and $K_q^\ab:=K\cap \QQ(\mu_q)$. We have $r_{\chi\xi}=1$ exactly when $(\chi\circ\norm)=\xi^{-1}$ as Hecke characters of $L$, and $r_{\chi\xi}=0$ otherwise. Hence, only characters $\xi$ trivial on the Kernel $J$ of the map in Lemma \ref{lem:kane_cors}, \emph{(1)}, contribute. For each such $\xi$, let $\psi_\xi$ be the Dirichlet character $\bmod\  q$ from \emph{(3)}. Then the double sum over $\chi$ and $\xi$  in \eqref{eq:siegel_walfisz_main_term} becomes
  \begin{align*}
    \left(\sum_{\rho'\in\Gal(L\cap \QQ(\mu_q)/\QQ)}\overline{\rho'(\sigma_b|_{L\cap\QQ(\mu_q)})}\right)\left(\sum_{\xi\in\widehat{H/J}}\psi'_\xi(\sigma_b)\overline{\xi(c)}\right).
  \end{align*}
  By character orthogonality,
  \begin{equation*}
    \sum_{\rho'\in\Gal(L\cap \QQ(\mu_q)/\QQ)}\overline{\rho'(\sigma_b|_{L\cap\QQ(\mu_q)})}=
    \begin{cases}
      |\Gal(L\cap \QQ(\mu_q) /\QQ)|,&\text{ if }\sigma_b|_{L\cap \QQ(\mu_q)}=1,\\
      0,&\text{otherwise.}
    \end{cases}
  \end{equation*}
  For the second sum, we consider $\psi_\xi'$ and $\xi$ as characters on
  $\Gal(K_q^\ab/\QQ)$, satisfying $\psi_\xi([K_q^\ab/\QQ, p])=\xi([K_q^\ab/\QQ, p])$ for all primes $p\nmid q\Delta_K$, so by Chebotarev they agree on all of $\Gal(K_q^\ab/\QQ)$. Observe that the image of the map in Lemma \ref{lem:kane_cors}, \emph{(1)}, is equal to $\Gal(K_q^\ab/L\cap \QQ(\mu_q))$.
If $\sigma_b|_{\QQ(\mu_q)\cap L}=1$, then both $c|_{K_q^\ab}$ and
$\sigma_b|_{K_q^\ab}$ are contained in this image, and by character
orthogonality, the inner sum over $\xi$ becomes
  \begin{equation*}
    \sum_{\xi\in\Gal(K_q^\ab/L\cap\QQ(\mu_q))^{\wedge}}\xi(\sigma_b|_{K_q^\ab}c|_{K_q^\ab}^{-1})=
    \begin{cases}
      |\Gal(K_q^\ab/L\cap\QQ(\mu_q))|,&\text{ if }\sigma_b|_{K_q^\ab}=c|_{K_q^\ab},\\
      0,&\text{otherwise.}
    \end{cases}
  \end{equation*}
  In total, the main term in \eqref{eq:siegel_walfisz_main_term} is zero, if $\sigma_b|_{K_q^\ab}\neq c|_{K_q^\ab}$, and equal to
  \begin{equation*}
    \frac{|C||\Gal(L\cap \QQ(\mu_q)/\QQ)||\Gal(K_q^\ab/L\cap\QQ(\mu_q))|}{\phi(q)[K:\QQ]}=\frac{|C|[K_q^\ab:\QQ]}{\phi(q)[K:\QQ]}=\frac{|C|}{[K(\mu_q):\QQ]},
  \end{equation*}
  otherwise. This is exactly $\eta_{K,C}(b,q)$.
\end{proof}

\begin{lemma}[Model in progressions]\label{lem:proxy_in_progressions}
  Let $K$ be a Galois number field and $C\subseteq\Gal(K/\QQ)$ a conjugacy
  class. Let $A>0$ and $N\in\NN$. Assume that $\Phi_{K^\ab}\mid W$ and let $b\in\{0,\ldots,W-1\}$. For
  every arithmetic progression $P\subseteq [N]$,
  we have
  \begin{equation*}
    \EE_{n\in [N]}\one_P(n)\left(\Lambda_{K,C,b,W}(n)-\phi(W)\eta_{K,C}(b,W)\Lambda_{b,W}(n)\right) \ll_{K,A,\cD} (\log N)^{-A}.
  \end{equation*}
\end{lemma}

\begin{proof}
  Write $P=\{l+mq\mid 1\leq m\leq \tilde{N}\}$, so $l+\tilde N q \leq N$. 
  Then 
  \begin{align*}
    \EE_{n\in [N]}\one_P(n)\Lambda_{K,C,b,W}(n) &=\frac{\phi(W)}{NW}\sum_{m=1}^{\tilde{N}}\Lambda_{K,C}((b+lW)+mqW)=\frac{\phi(W)}{NW}\hspace{-0.5cm}\sum_{\substack{b+lW<n\leq b+lW+q\tilde{N}W\\n\equiv b+lW\bmod qW}}\hspace{-0.3cm}\Lambda_{K,C}(n).
  \end{align*}  
  We apply Proposition \ref{prop:better_than_chebotarev} twice and use \eqref{eq:cheb_def_W} to obtain the estimate
  \begin{align}
   \EE_{n\in [N]}\one_P(n)\Lambda_{K,C,b,W}(n)&=\frac{\phi(W)q\tilde{N}}{N}\eta_{K,C}(b+lW,qW) +
                        O_{K,A,\cD}\left((\log N)^{-A}\right).\label{eq:proxy_first}
  \end{align}
  Anaolgous applications of Proposition
  \ref{prop:better_than_chebotarev} with $K=\QQ$ (i.e. the
  Siegel-Walfisz theorem) yield
  \begin{align*}
    \EE_{n\in [N]}\one_P(n)\Lambda_{b,W}(n) &=\frac{\phi(W)}{NW}\sum_{\substack{b+lW<n\leq b+lW+q\tilde{N}W\\n\equiv b+lW\bmod qW}}\Lambda(n)\\
    &=\frac{\phi(W)q\tilde{N}}{N}\frac{\one_{(b+lW,qW)=1}}{\phi(qW)} +
      O_{A,\cD}((\log N)^{-A}).
  \end{align*}
    By Lemma \ref{lem:equidistribution_cheb}, the main terms of $\EE_{n\in
    P}\Lambda_{K,C,b,W}(n)$ and $\phi(W)\eta_{K,C}(b,W)\EE_{n\in
    P}\Lambda_{b,W}(n)$ match up.
\end{proof}

\subsection{Proof of Theorem \ref{thm:W_tricked_chebotarev_mangoldt_with_nilsequences_nonconst_avg}}

\begin{lemma}\label{lem:non_constant_proxy_equidist_nilsequences}
  With the setup of Proposition \ref{prop:equidist_nilsequences} we have
   \begin{equation*}
    \left|\EE_{n\in [N]}(\Lambda_{K,C,b,W}(n)-\phi(W)\eta_{K,C}(b,W)\Lambda_{b,W}(n))\one_P(n)F(g(n)\Gamma)\right| \ll_{K,m,d} \delta^{c(m,d,D)}Q\vecnorm{F}_{\Lip}(\log N)^4. 
  \end{equation*}
\end{lemma}

\begin{proof}
  Using the triangle inequality and Proposition
  \ref{prop:equidist_nilsequences}, once for $K$ and once for $\QQ$, we may
  estimate the average to be bounded in the lemma by
  \begin{align*}
    &\left|\EE_{n\in[N]}\Lambda_{K,C,b,W}(n)\one_P(n)F(g(n)\Gamma)\right|+\phi(W)\eta_{K,C}(b,W)\left|\EE_{n\in[N]}\Lambda_{b,W}(n)\one_P(n)F(g(n)\Gamma)\right|\\ &\ll_{K,m,d} \delta^{c(m,d,D)}Q\vecnorm{F}_{\Lip}(\log N)^4. 
  \end{align*}
\end{proof}

  The proof that Lemma \ref{lem:non_constant_proxy_equidist_nilsequences}  implies Theorem \ref{thm:W_tricked_chebotarev_mangoldt_with_nilsequences_nonconst_avg} is almost the same as the proof, given on
  \cite[pp.544--547]{MR2877066}, that \cite[Proposition 2.1]{MR2877066} implies
  \cite[Theorem 1.1]{MR2877066}, with $\mu(n)$ replaced by our function
  $\Lambda_{K,C,b,W}(n)-\phi(W)\eta_{K,C}(b,W)\Lambda_{b,W}(n)$. The key
  differences are:
  \begin{itemize}
  \item $B$ needs to be chosen sufficiently large depending on $d,[K:\QQ],m,A$,
    as the power $c$ of $\delta$ in Lemma
    \ref{lem:non_constant_proxy_equidist_nilsequences} depends on $d,[K:\QQ],m$. Hence also $N$
    needs to be sufficiently large in terms of these parameters. Smaller $N$
    are captured by adapting the implied constant depending on $K,d,m,A$.
  \item Instead of $|\mu(n)|\leq 1$, we have the weaker trivial bound
    \begin{equation*}
      \Lambda_{K,C,b,W}(n)-\phi(W)\eta_{K,C}(b,W)\Lambda_{b,W}(n)\ll \log N.
    \end{equation*}
    This is relevant in the deduction of \cite[(2.6)]{MR2877066}, but
    the difference can be absorbed immediately by adapting the value of $A$.
  \item Instead of \cite[Proposition A.2]{MR2877066} for $\mu(n)$, we use Lemma
    \ref{lem:proxy_in_progressions} to bound correlations as in
    \cite[(2.9)]{MR2877066} when $F_{j,k}$ is constant.
  \item Instead of the application of \cite[Proposition 2.1]{MR2877066} at the
    very end of the proof, we apply Lemma \ref{lem:non_constant_proxy_equidist_nilsequences}.  
  \end{itemize}

  \section{Hooley's method: proof of Theorem \ref{thm:hooley}}\label{sec:hooley}

  \subsection{The fields $F(q,k,a)$}
  We recall here some results also used in \cite{MR4201547}.
  
\begin{lemma}(\cite[Lemma 2.3]{MR2490093},\cite[Lemma 2.2]{MR4201547})\label{lem:Fqka_degree}
For $k$  square-free 
let $k'= k/\gcd(k, h_a)$. Then 
$
[F(q,k,a) : \QQ] = 
k'
\phi(\lcm(q, k))/
\epsilon(q, k)
$,
where 
\begin{equation*}
\epsilon(q, k)=
  \begin{cases}
    2, &\text{ if } 2 \mid k \ \text{ and } \ \Delta_a\mid \lcm(q, k),\\
    1, &\text{ otherwise.}
  \end{cases}
\end{equation*}
\end{lemma}

\begin{lemma}(\cite[Lemma 2.3]{MR4201547})\label{lem:Fqka_disc}
Let $k' = k/\gcd(k, h_a)$ and $a = g_1^{\gcd(k,h_a)} g_2^{k}$, with $g_1$ free of $k'$-th powers. Then 
\begin{equation*}
\frac{
\log\abs{\Disc(F(q,k,a))}
}{
[F(q,k,a):\QQ]
} 
\leq \log k' + \log(\lcm(q, k)) + 2\log \abs{g_1}.
\end{equation*}
\end{lemma}

\begin{lemma}(\cite[Lemma 2.4]{MR2490093},\cite[Lemma 2.4]{MR4201547})
\label{lem:intersection}
We have
\begin{equation*}
[\QQ(\zeta_q) \cap G(k,a) : \QQ(\zeta_{\gcd(q, k)})] =
      \begin{cases}
        2 &\text{ if } 2\mid k, \ \Delta_a \nmid k \text{ and } \Delta_a\mid \lcm(q, k),\\
        1 &\text{ otherwise.}
      \end{cases}
\end{equation*}
In the first case, the integer $\beta_a(q)$ defined in \eqref{eq:beta_q_def} is a
fundamental discriminant and we have
$\QQ(\zeta_q) \cap G(k,a) = \QQ(\zeta_{\gcd(q, k)},\sqrt{\beta_a(q)})$.
\end{lemma}

\subsection{Proof of Theorem \ref{thm:hooley}}
In this proof, $p,q$ will always denote primes.
We start with some simple reductions.
At the cost of a harmless error
$\ll C(\log N)^{1/3}/\sqrt{N}$, we may 
replace the functios $\Lambda_{a,b,W}(\cdot)$ and $\Lambda_{G(k,a),\{\ident\},b,W}(\cdot)$ by
$\Lambda_{a,b,W}(\cdot)'$ and $\Lambda_{G(k,a),\{\ident\},b,W}'(\cdot)$, respectively. Moreover, we let $f:\NN\to\CC$ be an arbitrary function that satisfies $f(b+nW)=F(n)$
for all $n\in\NN$. Hence, our goal is to prove that
  \begin{equation}\label{eq:hooley_proof_goal}
   \EE_{n\leq N}\Lambda_{a,b,W}'(n)f(b+nW) =
    \hspace{-0.3cm}\sum_{\substack{k\in\NN\\p_+(k)\leq w(N)}}\hspace{-0.3cm}\mu(k)\EE_{n\leq
    N}\Lambda_{G(k,a),\{\ident\},b,W}'(n)f(b+nW) +
    O_{a,C,\cD}\left(\frac{1}{w(N)}\right).
  \end{equation}
If $\gcd(b,W)\neq 1$, then the only prime $p\equiv b\bmod W$ can be
$p=\gcd(b,W)\leq b<W$, which is not counted by $\Lambda_{a,b,W}'$ and
$\Lambda_{G(k,a),b,W}'$.  is at most $1$, hence the
expectations on both sides of the estimate in Theorem \ref{thm:hooley} are zero.
Therefore, we assume from now on that
$\gcd(b,W)=1$.

  Let $\spl(G(k,a))$ denote the set of rational primes $p$ splitting completely in $G(k,a)$. For primes $p,q$,
  we let $R_a(q,p)$ be the property that $p\in\spl(G_{q,a})$.
  Then
\begin{equation*}
  R_{a}(q,p)
  \quad\Longleftrightarrow\quad \QQ_p(\mu_q,\sqrt[q]{a})=\QQ_p
  \quad\Longleftrightarrow\quad q\mid p-1\text{ and }a\in\FF_p^{\times q},
\end{equation*}
and therefore
\begin{equation}\label{eq:primitive_root_splitting_cond}
  \FF_p^\times=\langle a\rangle\quad\Longleftrightarrow\quad p\nmid a\text{ and }\forall q\mid p-1, a\notin
  \FF_p^{\times q}
  \quad\Longleftrightarrow\quad p\nmid a\text{ and }R_a(q,p)\text{ fails for
  all }q.
\end{equation}
For any $1\leq\eta_1\leq\eta_2\leq WN-1$, let
  \begin{equation*}
    M(\eta_1,\eta_2;N):=\#\{p\leq W(N+1)\where p\equiv b\bmod W,\ R_a(q,p)\text{ holds for some
  }q\in (\eta_1,\eta_2]\}.
\end{equation*}
  Using \eqref{eq:primitive_root_splitting_cond}, we see that the
  expression on the left-hand side of \eqref{eq:hooley_proof_goal} is equal to
  \begin{align}
    &\frac{\phi(W)}{WN}\sum_{\substack{b<m\leq WN+b\\m\equiv b\bmod
    W}}\Lambda_a'(m)f(m) = \frac{\phi(W)}{WN}\sum_{\substack{b<p\leq
    WN+b\\p\equiv b\bmod W\\p\nmid a}}(\log p)\prod_{q}(1-\one_{R_a(q,p)})f(p)\nonumber\\
    &= \frac{\phi(W)}{WN}\sum_{\substack{b<p\leq WN+b\\p\equiv b\bmod W\\p\nmid
    a}}(\log p)\prod_{q\leq w(N)}(1-\one_{R_a(q,p)})f(p) + O_C\left(\frac{\phi(W)\log N}{WN}M(w(N),W(N+1);N)+1\right).\label{eq:hooley_asympt}
  \end{align}
  Expanding the product over $q$,
  we see that the main term is equal to
  \begin{align*}
    \frac{\phi(W)}{W}\sum_{\substack{k\in\NN\\p_+(k)\leq
    w(N)}}\frac{\mu(k)}{N}\sum_{\substack{b<p\leq WN+b\\p\equiv b\bmod
    W\\p\nmid a}}(\log p)\prod_{q\mid k}\one_{p\in\spl(G_{q,a})}f(p).
  \end{align*}
  As $k$ is squarefree, the field $G(k,a)$ is the compositum of the fields
  $G_{q,a}$ for all prime divisors $q$ of $a$, and therefore $\prod_{q\mid
    k}\one_{p\in\spl(G_{q,a})}=\one_{p\in\spl(G(k,a))}$. Hence, the expression
  in the latter displayed formula is equal to the main term on the right-hand
  side of \eqref{eq:hooley_proof_goal}, up to a harmless error $\ll_a 1/N$ coming from
  $p\mid a$ and $k=1$. Next, we deal with the error term by analysing the
  quantity $M(w(N),W(N+1);N)$ in a fashion analogous to \cite{MR207630}. The
  main difference is that we have to deal with progressions modulo $W$. Let
  \begin{equation*}
    \xi_2:=\frac{\sqrt{N}}{\sqrt{W}(\log N)^2}\quad\text{ and }\quad\xi_3:=W\sqrt{N}(\log N),
  \end{equation*}
  then
  \begin{equation*}
    M(w(N),W(N+1);N) = M(w(N),\xi_2;N) + M(\xi_2,\xi_3;N) + M(\xi_3,W(N+1);N).
  \end{equation*}
  Let us estimate all three summands individually, writing, for a prime $q$,
  \begin{equation*}
    P(W,q;x):=\#\{p\leq x\where p\equiv b\bmod W\text{ and }R_a(q,p)\text{ holds}\}.
  \end{equation*}
  Let $q>w(N)$, so $q\nmid W$. Then, similarly as in Hooley's proof of \cite[(2)]{MR207630}, we see by Brun-Titchmarsh that
  \begin{equation*}
    P(W,q;W(N+1))\leq \sum_{\substack{p\leq W(N+1)\\p\equiv 1\bmod q\\p\equiv b\bmod W}}1\ll \frac{WN}{\phi(qW)\log(WN/qW)}\ll \frac{W}{\phi(W)}\frac{N}{q\log(N/q)}.
  \end{equation*}
  Therefore,
  \begin{align}
    M(\xi_2,\xi_3;N)&\ll \frac{WN}{\phi(W)\log
                      N}\sum_{\xi_2<q\leq\xi_3}\frac{1}{q}\ll\frac{WN}{\phi(W)(\log
                      N)^2}\sum_{\xi_2<q\leq \xi_3}\frac{\log
                      q}{q}\ll\frac{WN}{\phi(W)(\log N)^2}\log(\xi_3/\xi_2)\nonumber\\
    &\ll\frac{WN\log\log N}{\phi(W)(\log N)^2}.\label{eq:hooley_M_1_est}
  \end{align}
  Our estimation of $M(\xi_3,W(N+1);N)$ is analogous to \cite[(3)]{MR207630}. The condition $R_a(q,p)$ implies in particular that $p\mid a^{(p-1)/q}-1$, and therefore $p\mid a^{2(p-1)/q}-1$. For $q>\xi_3$ and $p\leq W(N+1)$, we have $p/q\leq 2\sqrt{N}(\log N)^{-1}$, hence every prime $p$ counted in $M(\xi_3,W(N+1);N)$ satisfies
  \begin{equation*}
    p\mid\prod_{m<2\sqrt{N}(\log N)^{-1}}(a^{2m}-1).
  \end{equation*}
  The same simple argument as in \cite[(3)]{MR207630} now shows that
  \begin{equation}\label{eq:hooley_M_2_est}
  M(\xi_3,W(N+1);N)\ll_aN(\log N)^{-2}.
\end{equation}
  It remains to estimate $M(w(N),\xi_2;N)$, for which we have to use
  $\HRH(a)$. Again, our arguments are analogous to the ones leading up to
  \cite[(33)]{MR207630}.
  Write
\begin{equation*}
W_0:=\prod_{p\leq w(N)}p,
\end{equation*}
so $W_0$ is squarefree and  $W \asymp_\cD W_0\leq (\log N)^{2/3}$. We start, similarly as above, with the estimate
  \begin{equation*}
    M(w(N),\xi_2;N)\leq \sum_{w(N)<q\leq \xi_2}P(W_0,q;W(N+1)).
  \end{equation*}
  The
  quantity $P(W_0,q;W(N+1))$ is the number of primes up to $W(N+1)$ unramified in $F(W_0,q,a)$, whose Frobenius class $\Frob_p$ in $\Gal(F(W_0,q,a)/\QQ)$ satisfies
  \begin{equation*}
    \Frob_p|_{\QQ(\mu_{W_0})}=\sigma_b:\zeta\mapsto\zeta^b\quad\text{ and }\quad \Frob_p|_{G(q,a)}=\ident_{G(q,a)}.
  \end{equation*}
  Next, we observe that $\HRH(a)$ implies that the Riemann hypothesis holds for the Dedekind zeta function of $F(W_0,q,a)$ for every prime $q>w(N)$. Indeed, the number $qW_0$ is squarefree and $F(qW_0,qW_0,a)/F(W_0,q,a)$ is a Kummer, hence abelian extension. Therefore, the Dedekind zeta function of $F(qW_0,qW_0,a)$  is the product of the Dedekind zeta function of $F(W_0,q,a)$ and some entire Hecke $L$-functions. Thus, every root of the latter zeta function is also a root of the former, which satisfies the Riemann hypothesis by $\HRH(a)$.

  Thus, the conditional  effective Chebotarev
  theorem \cite[Th\'eor\`eme 4]{MR0644559} for
  $F(W_0,q,a)$ yields the asymptotic
  \begin{align*}
    P(W_0,q;W(N+1)) &= \eta_{G(q,a),\{\ident\}}(b,W_0)\Li(W(N+1))\\ &+ O\left(\frac{(WN)^{1/2}\log(|\Disc(F(W_0,q,a))|)}{[F(W_0,q,a):\QQ]}+(WN)^{1/2}\log(WN)\right).
  \end{align*}
  Thus, using $W=\cD W_0$ and using Lemma \ref{lem:Fqka_degree}, Lemma \ref{lem:Fqka_disc} and the
  fact that $\gcd(q,W)=1$, we get
  \begin{equation*}
    P(W_0,q;W(N+1))\ll_\cD \frac{WN}{\log(WN)q\phi(q)\phi(W)} + (WN)^{1/2}\log(WN).
  \end{equation*}
  Therefore, as in \cite[(33)]{MR207630},
  \begin{align}
    M(w(N),\xi_2;N)&\ll\cD\frac{WN}{\log(N)\phi(W)}\sum_{q> w(N)}\frac{1}{q(q-1)}+\frac{\xi_2}{\log\xi_2}(WN)^{1/2}\log(N)\nonumber\\
    &\ll \frac{W}{\phi(W)}\frac{N}{w(N)\log(N)} + \frac{N}{(\log N)^2}.\label{eq:hooley_M_3_est}
  \end{align}
  From \eqref{eq:hooley_M_1_est}, \eqref{eq:hooley_M_2_est} and \eqref{eq:hooley_M_3_est}, we conclude that the error term in \eqref{eq:hooley_asympt} satisfies the desired bound. This concludes our proof of Theorem \ref{thm:hooley}.

\section{\texorpdfstring{$W$}{W}-tricked Artin-von Mangoldt  and nilsequences: proof of Theorem \ref{thm:W_tricked_artin_mangoldt_with_nilsequences_nonconst_avg}}\label{sec:W_tricked_lambda_a_with_nilsequences}
  We use Theorem \ref{thm:hooley} to estimate
  \begin{align*}
    \EE_{n\in [N]}\Lambda_{a,b,W}(n)F(g(n)\Gamma)
                                            &= \sum_{\substack{p_+(k)\leq w(N)}}\mu(k)\EE_{n\in [N]}\Lambda_{G(k,a),b,W}(n)F(g(n)\Gamma)+O_{a,\cD}\left(\frac{1}{w(N)}\right).
  \end{align*}
  Each $k$ with $p_+(k)\leq w(N)$ satisfies $\lcm(\Delta_a,k)\mid W$,
  as $\Delta_a\mid \cD$. As
  $G(k,a)^{\ab}\subseteq\QQ(\mu_k,\sqrt{a})$ by Lemma
  \ref{lem:max_abelian_Gka}, we obtain $\Phi_{G(k,a)^\ab}\mid \lcm(k,|\Delta_a|) \mid
  W$. Hence, we may apply
  Theorem
  \ref{thm:W_tricked_chebotarev_mangoldt_with_nilsequences_nonconst_avg} with $A=2$
  for every field $G(a,k)$ with $p_+(k)\leq w(N)$ to
  obtain, using \eqref{eq:eta_special_case},
  \begin{align}
    &\left|\EE_{n\in [N]}\Lambda_{a,b,W}(n)F(g(n)\Gamma) - \left(\sum_{p_+(k)\leq
      w(N)}\mu(k)\phi(W)\eta(a,b,k,W)\right)\EE_{n \in
      [N]}\Lambda_{b,W}(n)F(g(n)\Gamma)\right|\label{eq:hooley_application_asymptotic}\\
    & \ll_{a,\cD} 
      \sum_{p_+(k)\leq w(N)}|\mu(k)|\frac{C_k}{(\log N)^2} + \frac{1}{w(N)},\nonumber
  \end{align}
  where, for each squarefree $k$, the positive constant $C_k = C_{k,a,m,d,Q,M}$
  arises from the bound in Theorem
  \ref{thm:W_tricked_chebotarev_mangoldt_with_nilsequences_nonconst_avg} with
  $A=2$ and
  depends only on the indicated quantities.   To bound the sum over $k$ in the error term, we now need to restrict  our function $w(N)$ to grow sufficiently slowly.
  We
  define for $t\geq 0$ the function
  \begin{equation*}
    h(t):=\log(2+t)+\sum_{\substack{p_+(k)\leq t}}|\mu(k)|C_k,
  \end{equation*}
  which is positive, strictly increasing and growing to infinity, and therefore
  has a strictly increasing inverse function $h^{-1}:[0,\infty)\to [0,\infty)$ that also grows to
  infinity. We now require that the function $w(N)$ grows to infinity with $N$
  sufficiently slowly to satisfy, in addtion to $w(N)\leq \frac{1}{3}\log\log N$, the condition
  \begin{equation*}
    w(N)\leq h^{-1}(\log N),
  \end{equation*}
  which clearly depends only on $a,m,d,Q,M$. With this restriction on the
  growth of $w(\cdot)$, we obtain the bound
  \begin{align}
    \sum_{p_+(k)\leq
    w(N)}|\mu(k)|\frac{C_k}{(\log N)^2} &\leq \frac{h(w(N))}{(\log N)^2}\leq
                                          \frac{1}{\log N} \leq \frac{1}{w(N)}\label{eq:hooley_application_super_weak_bound}.
  \end{align}
  Finally,
  we estimate, using Lemma \ref{lem:Fqka_degree},
  \begin{align}
    &\sum_{p_+(k)\leq
      w(N)}\mu(k)\phi(W)\eta(a,b,k,W)-\phi(W)\delta(a,b,W)
      \ll_a \sum_{p_+(k)>w(N)}\frac{|\mu(k)|}{k\phi(k/(W,k))}\nonumber\\
    &=\sum_{k_1\mid W}\frac{1}{k_1}\sum_{\substack{(k_2,W)=1\\k_2>1}}\frac{|\mu(k_2)|}{k_2\phi(k_2)}\ll\sum_{k_1\mid W}\frac{1}{w(N)\log w(N)k_1}\ll\frac{1}{w(N)}. \label{eq:hooley_application_complete_summation}
  \end{align}
  Here, we used the estimates
  \begin{align*}
    \sum_{\substack{(k,W)=1\\k>1}}\frac{|\mu(k)|}{k\phi(k)} &= \prod_{p> w(N)}\left(1+\frac{1}{p(p-1)}\right)-1\ll e^{O(1/(w(N)\log w(N)))}-1 \ll \frac{1}{w(N)\log w(N)},\\
    \sum_{k\mid W}\frac{1}{k}&=\prod_{p\leq w(N)}\left(1+\frac{1}{p}\right)\ll e^{\log\log w(N)}\ll\log w(N).
  \end{align*}

  From the $1$-boundedness of $F$, we moreover see that
  \begin{align*}
    \left|\EE_{n \in
      [N]}\Lambda_{b,W}(n)F(g(n)\Gamma)\right|\leq\EE_{n\in [N]}\Lambda_{b,W}(n)\ll 1.
  \end{align*}
  Together with the estimates
  \eqref{eq:hooley_application_super_weak_bound} and
  \eqref{eq:hooley_application_complete_summation} applied to
  \eqref{eq:hooley_application_asymptotic}, this yields the desired bound in Theorem \ref{thm:W_tricked_artin_mangoldt_with_nilsequences_nonconst_avg}.

\section{Constellations with prescribed primitive roots: proof of Theorem \ref{thm:main_artin}}\label{sec:proof_main_artin}
In this section, we will deduce Proposition \ref{prop:gowers_estimate_artin} from Theorem \ref{thm:W_tricked_artin_mangoldt_with_nilsequences_nonconst_avg} and the inverse theorem for Gowers norms. Then we deduce Theorem \ref{thm:main_artin} from Proposition \ref{prop:gowers_estimate_artin} using the generalised von Neumann theorem. For both of these steps, we require a pseudorandom majorant for our functions $\Lambda_{a,b,W}'(n)$.

\subsection{Pseudorandom majorant for $\Lambda_{a,b,W}'(n)$}
\begin{lemma}[Pseudorandom majorant]\label{lem:pseudorandom_majorant}
  Let $t\geq 1$ and $D>1$. Then there is a constant $C_0=C_0(D)$, such that the
  following holds. Let $C\geq C_0$ and let $N'\in[CN,2CN]$ be prime. Let
  $a_1,\ldots,a_t\in\ZZ$, such that no $a_i$ is equal to $-1$ or a perfect
  square. Let $b_0,\ldots,b_t\in\{0,\ldots,W-1\}$ be coprime to $W$. Then there
  exists a function $\nu:\ZZ/N'\ZZ\to (0,\infty)$
  that has the following properties:
  \begin{description}
  \item[Domination] We have
    \begin{equation}\label{eq:pseudorandom_majorant_domination}
      1+\Lambda_{b_0,W}'(n)+\Lambda_{a_1,b_1,W}'(n)+\cdots+\Lambda_{a_t,b_t,W}'(n) \ll_{t,D,C} \nu(n)
    \end{equation}
    for all $n\in[N^{3/5},N]$, where $[N]$ is embedded in $\ZZ/N'\ZZ$ via $n\mapsto n+N'\ZZ$.
  \item[$D$-linear forms condition] For all $\tilde s,\tilde t\in [D]$ and finite complexity systems $\tilde\Psi=(\tilde\psi_1,\ldots,\tilde\psi_{\tilde t})$ of affine-linear forms on $\ZZ^{\tilde s}$ with the coefficients of all linear terms bounded in absolute value by $D$, we have
    \begin{equation*}
      \EE_{n \in (\ZZ/N'\ZZ)^{\tilde s}}\prod_{i\in [\tilde t]}\nu(\tilde\psi_i(n)) = 1 + o_{t,D,C,\cD,w(\cdot)}(1).
    \end{equation*}
  \end{description}
\end{lemma}

\begin{proof}
  As $\Lambda'_{a_i,b_i,W}(n)\leq \Lambda'_{b_i,W}(n)$, it is enough to produce a function $\nu(\cdot)$ in which the domination property \eqref{eq:pseudorandom_majorant_domination} is replaced by
  \begin{equation}\label{eq:pseudorandom_majorant_domination_GT}
    1+\Lambda_{b_0,W}'(n)+\Lambda_{b_1,W}'(n)+\cdots+\Lambda_{b_t,W}'(n) \ll_{t,D,C} \nu(n).
  \end{equation}
  A function $\nu$ satisfying this
  is constructed, essentially, in the proof of
  \cite[Proposition 6.4]{MR2680398} in \cite[Appendix
  D]{MR2680398}. The only difference to our situation is a minor
  change in the definition of $W$. In \cite{MR2680398}, the function
  $w(N)=\log\log\log N$ is fixed and $W=\prod_{p\leq w(N)}p$, whereas
  for us $w(N)$ can be any function that grows to infinity with $N$
  sufficiently slowly, and $W$ is defined as in \eqref{eq:cheb_def_W}.

  The proof in \cite[Appendix D]{MR2680398}, with the obvious necessary
  modifications, goes through in our situation as well. See also
  \cite{MR3656911} for a similar situation. Note that this leads to the
  dependence of the error term in the linear forms condition on $\cD$ and
  $w(\cdot)$. Note, moreover, that $C_0=C_0(D)$ is independent of $\cD,W$, and
  the same applies to the bound in \eqref{eq:pseudorandom_majorant_domination}.

  The majorant constructed in \cite[Appendix D]{MR2680398} also satisfies a
  so-called \emph{correlation condition}. This condition was used in
  \cite{MR2680398} for transferring the inverse theorem for Gowers norms, but
  is no longer required due to recent advances.
\end{proof}

\subsection{Gowers-norm estimate:
  proof of Proposition \ref{prop:gowers_estimate_artin}}

\begin{lemma}\label{lem:gowers_estimate_classical}
    Let $u,N\in\NN$ and $b\in\{0,\ldots,W-1\}$. Then  
  \begin{equation*}
    \vecnorm{\Lambda_{b,W}'(\cdot)-1_{\gcd(b,W)=1}}_{U^{u+1}[N]}= o_{u,\cD,w(\cdot)}(1).
  \end{equation*}
\end{lemma}
\begin{proof}
  The case where $w(N)=\log\log\log N$ and $\cD=1$ was proved by Green and Tao in \cite[Theorem
  7.2]{MR2680398}, and their proof also works in our situation. As a thorough
  verification of this claim places quite a burden on the
  reader, we will now also derive the lemma rigorously from Bienvenu's result
  \cite[Theorem 1.3]{MR3656911},
  concerning affine-linear correlations of the von Mangoldt function in which the
  linear coefficients are allowed to grow logarithmically in $N$. Let us assume $\gcd(b,W)=1$, as otherwise the result is obvious. Then
  \begin{align*}
    \vecnorm{\Lambda_{b,W}'(\cdot)-1}_{U^{u+1}[N]}^{2^{u+1}}&=\EE_{(x,\vh)}\prod_{\vw\in\{0,1\}^{u+1}}\left(\Lambda_{b,W}'(x+\vw\cdot\vh)-1\right)\\
                                                            &= \sum_{\substack{A\subseteq \{0,1\}^{u+1}}}(-1)^{2^{u+1}-|A|}\left(\frac{\phi(W)}{W}\right)^{|A|}\EE_{x,\vh}\prod_{\vw\in A}\Lambda'\bigg(b+W x+\sum_{\omega_i\neq 0} W h_i\bigg)\\
    &=\sum_{\substack{A\subseteq \{0,1\}^{u+1}}}(-1)^{2^{u+1}-|A|}\left(1+o_{u,\cD,w(\cdot)}(1)\right) = o_{u,\cD,w(\cdot)}(1),
  \end{align*}
  where the averages are taken over all $(x,\vh)\in\ZZ^{u+2}$ that satisfy $x+\vw\cdot\vh\in[N]$ for all $\vw\in\{0,1\}^{u+1}$. Here we have applied \cite[Theorem 1.3]{MR3656911} to compute for each $A\subseteq\{0,1\}^{u+1}$ the average
  \begin{equation*}
    \EE_{x,\vh}\prod_{\vw\in A}\Lambda\bigg(b+W x+\sum_{\omega_i\neq 0} W h_i\bigg) = \left(\frac{W}{\phi(W)}\right)^{|A|}\prod_{p\nmid W}(1+o_{u}(1))=\left(\frac{W}{\phi(W)}\right)^{|A|}(1+o_{u,\cD,w(\cdot)}(1)).
  \end{equation*}
  Indeed, one easily sees that the local factors $\beta_p$ in \cite[Theorem 1.3]{MR3656911} for the system of affine-linear forms under consideration above satisfy $\beta_p=(p/\phi(p))^{|A|}$ when $p\mid W$ and $\beta_p=1+o_u(p^{-2})$ otherwise. As those values of $x,\vh$ where one of the forms takes a prime power value are easily seen to be negligible in the above expectation, the same estimate holds with $\Lambda'$ in place of $\Lambda$.
\end{proof}

\begin{lemma}\label{lem:gowers_estimate_artin_nonconst_avg}
  Under the same hypotheses as in Proposition
  \ref{prop:gowers_estimate_artin}, we have 
  \begin{equation*}
    \vecnorm{\Lambda_{a,b,W}'(\cdot)-\phi(W)\delta(a,b,W)\Lambda_{b,W}'(\cdot)}_{U^{u+1}[N]}\leq \delta. 
  \end{equation*}
\end{lemma}

\begin{proof}
  If $\gcd(b,W)\neq 1$, then $\delta(a,b,W)=0$ and $\Lambda_{a,b,W}'(n)=0$ for all
  $n\in [N]$. Hence, we may assume that $\gcd(b,W)=1$. Write, for brevity,
  $f(n):=\Lambda_{a,b,W}'(n)-\phi(W)\delta(a,b,W)\Lambda_{b,W}'(n)$ and
  \begin{equation*}
    \tilde f(n):=\one_{[N^{3/5},N]}(n)f(n).
  \end{equation*}
  Take $t=1$ and $D=4^u$ in Lemma \ref{lem:pseudorandom_majorant}, choose $C=\max\{20,C_0(D)\}$ and any prime $N'\in [CN,2CN]$. Let $\nu:\ZZ/N'\ZZ\to (0,\infty)$ be the pseudorandom majorant constructed in Lemma \ref{lem:pseudorandom_majorant} satisfying the $D$-linear forms condition and
    $|f(n)| \leq 1 + \Lambda_{b,W}'(n)+\Lambda_{a,b,W}'(n) \ll_{u} \nu(n)$ for all $n\in [N^{3/5},N]$. Then there is a small positive constant $c=c(u)$, such that
  \begin{equation}\label{eq:gowers_estimate_majorant}
    |c\tilde f|\leq \nu(n)\quad   \text{ holds for all } n\in[N].
  \end{equation}
  Note, moreover, that the $D$-linear forms condition for $\nu$ implies that 
  \begin{equation}\label{eq:gowers_estimate_pseudorandomness}
    \vecnorm{\nu - 1}_{U^{2u}(\ZZ/N'\ZZ)}=o_{u,\cD,w(\cdot)}(1).
  \end{equation}
  Suppose the conclusion of the Lemma to be wrong, so $\vecnorm{f}_{U^{u+1}[N]}\geq \delta$ for arbitrarily large values of $N$. One easily sees directly from the definition of $\vecnorm{\cdot}_{U^{u+1}[N]}$ and the fact that $|f(n)|\leq \log n$ that $\vecnorm{f-\tilde f}_{U^{u+1}[N]}=o_u(1)$. Hence, for sufficiently large $N$, we have
  \begin{equation}\label{eq:gowers_estimate_f_tilde}
    \vecnorm{c\tilde f}_{U^{u+1}[N]}\geq c\delta/2.
  \end{equation}

  We are now in a position to apply \cite[Theorem 5.1]{MR4436255}, a version of
  the transferred inverse theorem for Gowers norms. Hence, there are a constant
  $M>0$, an $u$-step nilmanifold $(G/\Gamma,d_{G/\Gamma})$ with smooth
  Riemannian metric and a constant $\epsilon>0$, all depending only on $u$ and
  $\delta$, with the following property: if $N$ is sufficiently large in terms
  of $u,\cD,w(\cdot)$, then there is a $1$-bounded linear nilsequence $F(g^nx)$
  on $G/\Gamma$ with Lipschitz constant $\leq M$, such that
  \begin{equation*}
    \left|\EE_{n\in [N]}\tilde f(n) F(g^nx) \right|\geq \epsilon.
  \end{equation*}
  Clearly, as $F$ is $1$-bounded and $|f(n)|\leq \log n$, this implies for large enough $N$, depending only on $\epsilon$ and thus only on $u,\delta$, that
  \begin{equation*}
    \left|\EE_{n\in [N]}f(n) F(g^nx) \right|\geq \epsilon/2.
  \end{equation*}
  The linear nilsequence $n\mapsto g^nx$
  is in $\poly(\ZZ,G_\bullet)$, with $G_\bullet$ the lower central series
  filtration on $G$, which is rational and of degree at most $u$. Hence, $G/\Gamma$ has a $Q$-rational Mal'cev
  basis $\cX$ adapted to $G_\bullet$, for some $Q\ll_{u,\delta} 1$. As $G/\Gamma$ is
  compact, the metric $d_{G/\Gamma}$ is comparable to the metric induced by
  $\cX$, and thus $\vecnorm{F}_{\Lip}\ll_{u,\delta}M$.  
  Therefore, if the function $w(\cdot)$ grows sufficiently slowly in terms of
  $a,u,\delta$ for the conclusion of
  Theorem \ref{thm:W_tricked_artin_mangoldt_with_nilsequences_nonconst_avg} to apply,
  we obtain
  \begin{equation*}
    \EE_{n\in[N]}\left(\Lambda_{a,b,W}(n)-\phi(W)\delta(a,b,W)\Lambda_{b,W}(n)\right)F(g^nx)=o_{a,\cD,w(\cdot)}(1).
  \end{equation*}
  As the contribution of prime powers is negligible, this gives
   \begin{equation*}
    \EE_{n\in[N]}f(n)F(g^nx)=o_{a,\cD,w(\cdot)}(1),
  \end{equation*}
and thus a contradiction as long as $N$ is large enough in terms of
  $\delta,u,a,\cD,w(\cdot)$.
\end{proof}

With the two previous lemmata in place, we can now easily deduce Proposition
\ref{prop:gowers_estimate_artin}. Indeed, the triangle inequality for
$\vecnorm{\cdot}_{U^{u+1}[N]}$ gives
\begin{align*}
  \vecnorm{\Lambda_{a,b,W}'(\cdot)-\phi(W)\delta(a,b,W)}_{U^{u+1}[N]}&\leq
  \vecnorm{\Lambda_{a,b,W}'(\cdot)-\phi(W)\delta(a,b,W)\Lambda_{b,W}'(\cdot)}_{U^{u+1}[N]}\\
  &+ \phi(W)\delta(a,b,W)\vecnorm{\Lambda_{b,W}'(\cdot)-1}_{U^{u+1}[N]}\leq \delta,
\end{align*}
if only $w(\cdot)$ grows sufficiently slowly and $N$ is sufficiently large for
Lemma \ref{lem:gowers_estimate_artin_nonconst_avg} to apply with $\delta/2$ in
place of $\delta$, and such that the $o$-term in Lemma
\ref{lem:gowers_estimate_classical} is $\leq \delta/2$.

\subsection{Application of the generalised von Neumann theorem}\label{sec:von_neumann}
The following lemma is essentially \cite[Proposition 7.1]{MR2680398}. By $\kappa_{a}(\delta)$, we denote a quantity that goes to zero with $\delta$, i.e. a quantity that is smaller in absolute value than any $\epsilon>0$ if only $\delta$ is sufficiently small in terms of $\epsilon$ and $a$.

\begin{lemma}[Generalised von Neumann theorem]\label{lem:neumann}
  Let $s,t,u,L\in\NN$. Then there exist constants $C_1,D$, depending only on
  $s,t,u,L$, such that the following holds. Let $N\in\NN$, let $C_1\leq C\ll_{s,t,u,L} 1$, and
  let $N'\in[CN,2CN]$ be prime. Let $\nu:\ZZ/N'\ZZ\to [0,\infty)$ be a function
  that satisfies the $D$-linear forms condition
  \begin{equation*}
    \left|\EE_{n\in(\ZZ/N'\ZZ)^{\tilde s}}\prod_{i\in [\tilde
        t]}\nu(\tilde\psi_i(n))-1\right|\leq z(N)
  \end{equation*}
 for all $\tilde s,\tilde t\in [D]$ and all finite complexity systems
 $\tilde\Psi=(\tilde\psi_1,\ldots,\tilde\psi_{\tilde t})$ of affine-linear
 forms on $\ZZ^{\tilde s}$ in which the coefficients of all linear terms are
 bounded in absolute value by $D$, and where $z:\NN\to(0,\infty)$ is a function
 satisfying $\lim_{N\to\infty}z(N)=0$.
 
 Let $f_1,\ldots,f_t:[N]\to \RR$ be functions with $|f_i(n)|\leq
 \nu(n)$ for all $i\in[t]$ and $n\in[N]$. Suppose that
 $\Psi=(\psi_1,\ldots,\psi_t)$ is a system of affine-linear forms in $u$-normal
 form with $\vecnorm{\Psi}_{N}\leq L$. Let $X\subseteq [-N,N]^s$ be a convex
 body such that $\Psi(X)\subseteq [N]^t$. If, for some $\delta>0$,
 \begin{equation*}
   \min_{1\leq j\leq t}\vecnorm{f_j}_{U^{u+1}[N]}\leq \delta,
 \end{equation*}
 then
 \begin{equation*}
   \sum_{n\in X}\prod_{i\in [t]}f_i(\psi_i(n))=o_{C,s,t,u,L,\delta,z(\cdot)}(N^s)+\kappa_{C,s,t,u,L}(\delta)N^s.
 \end{equation*}
 \end{lemma}

\begin{proof}
  The statement is almost the same as in \cite[Proposition 7.1]{MR2680398},
  except for two differences. Firstly, we do not suppose that the majorant
  $\nu$ satisfy a $D$-correlation condition. This is irrelevant, as the
  proof of \cite[Proposition 7.1]{MR2680398} given in \cite[Appendix
  C]{MR2680398} uses only the $D$-linear forms condition. Secondly, we state 
  more explicitly the dependence of the estimate on the error term $z(N)$ in
  the linear forms condition, in particular that the term $\kappa_{C,s,t,u,L}(\delta)$ in
  the final estimate does not depend on this error term. One can see this
  by following carefully the proof in \cite[Appendix C]{MR2680398}.
\end{proof}

\begin{lemma}[$W$-tricked Artin primes in lattices in normal form]\label{lem:neumann_application}
  Let $\epsilon>0$, $s,t,u,L,N\in\NN$, $a_1,\ldots,a_t\in\ZZ\smallsetminus\{-1\}$
  such that no $a_i$ is a perfect square, and
  $\Psi=(\psi_1,\ldots,\psi_t):\ZZ^s\to\ZZ^t$ a system of affine-linear forms
  in $u$-normal form with $\vecnorm{\Psi}_N\leq L$. Assume $\HRH(a_i)$ for all $1\leq i\leq t$. Let
  $X\subseteq [-N,N]^s$ be a convex body on which $\psi_1,\ldots,\psi_t>N^{7/10}$. Assume that the function $w(\cdot)$
  grows sufficiently slowly in terms of $\epsilon,u,s,t,L,a_1,\ldots,a_t$, and that
  $\cD_\va$ divides $\cD$. Then, for all
  $b_1,\ldots,b_t\in\{0,\ldots,W-1\}$, we have
  \begin{equation*}
    \left|\sum_{n\in
        X\cap\ZZ^s}\prod_{i\in[t]}\left(\Lambda_{a_i,b_i,W}'(\psi_i(n))-\phi(W)\delta(a_i,b_i,W)\right)\right|\leq
    \epsilon N^s,
  \end{equation*}
  if only $N$ is sufficiently large in terms of $\epsilon,s,t,u,L,\cD,w(\cdot)$.
\end{lemma}

\begin{proof}
  Multiplying $N$ by a constant factor $\ll_{s,L} 1$ and adapting $\epsilon$
  accordingly, we may assume that $\Psi(X)\subseteq [N]^t$. For
  $i\in[t]$, we write
  $f_i(n):=\Lambda_{a_i,b_i,W}'(n)-\phi(W)\delta(a_i,b_i,W)$ and
  \begin{equation*}
    \tilde f_i(n):=\one_{[N^{3/5},N]}(n)f_i(n).
  \end{equation*}
  Let $C_1,D$ be as in Lemma \ref{lem:neumann}, let $C_0(D)$ be as in Lemma
  \ref{lem:pseudorandom_majorant}, and $C:=\max\{C_1,C_0(D)\}$. Let $N'$ be any
  prime in $[CN,2CN]$.
  
  Let $\nu:\ZZ/N'\ZZ\to (0,\infty)$ be the pseudorandom majorant constructed in Lemma \ref{lem:pseudorandom_majorant} satisfying the $D$-linear forms condition and
    $|f(n)| \leq 1 + \Lambda_{a_1,b_1,W}'(n)+\cdots+\Lambda_{a_t,b_t,W}'(n) \ll_{s,t,u,L} \nu(n)$ for all $n\in [N^{3/5},N]$. Then there is a small positive constant $c=c(s,t,u,L)$, such that
  \begin{equation}\label{eq:von_neumann_majorant}
    |c\tilde f_i|\leq \nu(n)\quad   \text{ holds for all } i\in[t]\text{ and }n\in[N].
  \end{equation}
  Note that due to our conditions on $\psi$, we have $\tilde
  f_i(\psi_i(n))=f_i(\psi_i(n))$ for all $n\in X\cap\ZZ^s$. 
  Suppose that the conclusion of the lemma does not hold, then
   \begin{equation}\label{eq:neumann_application_contrapositive}
    \left|\sum_{n\in X\cap\ZZ^s}\prod_{i\in[t]} c\tilde f_i(\psi_i(n))\right|> c\epsilon N^s.
    \end{equation}
  We apply Lemma \ref{lem:neumann} with the functions $c\tilde f_i$ in place of
  $f_i$. Let $\delta=\delta(\epsilon,s,t,u,L)>0$ be small enough so that the term
  $\kappa_{C,s,t,u,L}(\delta)$ in the conclusion of Lemma \ref{lem:neumann} is
  $\leq c\epsilon/2$.

  Now assume that the function $w(\cdot)$ grows sufficiently slowly in terms of
  $c,\delta,u,a_1,\ldots,a_t$, and thus in terms of $\epsilon,u,s,t,L,a_1,\ldots,a_t$,
  so that Proposition \ref{prop:gowers_estimate_artin} can be applied with $\delta/(2c)$
  instead of $\delta$ for each of the
  functions $\Lambda'_{a_i,b_i,W}(\cdot)$. Then
  \begin{equation*}
    \vecnorm{\Lambda'_{a_i,b_i,W}(\cdot)-\phi(W)\delta(a,b,W)}_{U^{u+1}[N]}\leq \delta/(2c),
  \end{equation*}
  and thus also
  \begin{equation*}
    \vecnorm{\tilde f_i}_{U^{u+1}[N]}\leq \delta\quad\text{ for all }\quad i\in [t].
  \end{equation*}
  Hence, Lemma \ref{lem:neumann} shows that
  \begin{equation*}
    \sum_{n\in X\cap\ZZ^s}\prod_{i\in[t]} c\tilde f_i(\psi_i(n))=o_{s,t,u,L,\epsilon,z(\cdot)}(N^s)+\frac{c\epsilon}{2}N^s,
  \end{equation*}
  where $z(\cdot)$ is the error term in the $D$-correlation condition in Lemma
  \ref{lem:pseudorandom_majorant}, and thus depends only on $t,D,C,\cD,w(\cdot)$,
  and therefore only on $s,t,u,L,\cD,w(\cdot)$. This contradicts
  \eqref{eq:neumann_application_contrapositive} if $N$ is sufficiently large in
  terms of $\epsilon,s,t,u,L,\cD,w(\cdot)$.
\end{proof}

\subsection{W-trick}

The following lemma is a simple reduction completely analogous to the deduction of \cite[Theorem 5.1]{MR2680398} from \cite[Theorem 5.2]{MR2680398}.
\begin{lemma}\label{lem:telescoping}
  Let $\epsilon>0$, $s,t,u,L,N\in\NN$, $a_1,\ldots,a_t\in\ZZ\smallsetminus\{-1\}$
  such that no $a_i$ is a perfect square, and
  $\Psi=(\psi_1,\ldots,\psi_t):\ZZ^s\to\ZZ^t$ a system of affine-linear forms
  in $u$-normal form with $\vecnorm{\Psi}_N\leq L$. Assume $\HRH(a_i)$ for all $1\leq i\leq t$. Let
  $X\subseteq [-N,N]^s$ be a convex body on which
  $\psi_1,\ldots,\psi_t>N^{7/10}$. Assume that the function $w(\cdot)$ grows
  sufficiently slowly in terms of $\epsilon,u,s,t,L,a_1,\ldots,a_t$. Then, for all $b_1,\ldots,b_t\in\{0,\ldots,W-1\}$, we have
  \begin{equation*}
    \left|\sum_{n\in
        X\cap\ZZ^s}\left(\prod_{i\in[t]}\Lambda_{a_i,b_i,W}'(\psi_i(n))-\prod_{i\in[t]}\phi(W)\delta(a_i,b_i,W)\right)\right|\leq \epsilon N^s,
  \end{equation*}
  if only $N$ is sufficiently large in terms of $\epsilon,s,t,u,L,\cD,w(\cdot)$.
\end{lemma}

\begin{proof}
  Write $\Lambda_{a_i,b_i,W}'(\cdot) = (\Lambda_{a_i,b_i,W}'(\cdot)-\phi(W)\delta(a_i,b_i,W))+\phi(W)\delta(a_i,b_i,W)$ and thus
  \begin{align*}
    &\prod_{i\in[t]}\Lambda_{a_i,b_i,W}'(\psi_i(n))-\prod_{i\in[t]}\phi(W)\delta(a_i,b_i,W)\\
    = &\sum_{\emptyset\neq J\subseteq [t]}\prod_{i\in
      J}\left(\Lambda_{a_i,b_i,W}'(\psi_i(n))-\phi(W)\delta(a_i,b_i,W)\right)\prod_{i\in
    [t]\smallsetminus J}\phi(W)\delta(a_i,b_i,W).
  \end{align*}
  For any $\emptyset\neq J\subseteq [t]$, we observe that the system
  $(\psi_i)_{i\in J}$ is still in $u$-normal form. Hence, we may apply Lemma
  \ref{lem:neumann_application} to show that, for large enough $N$,
  \begin{equation*}
    \left|\sum_{n\in
        X\cap\ZZ^s}\prod_{i\in J}\left(\Lambda_{a_i,b_i,W}'(\psi_i(n))-\phi(W)\delta(a_i,b_i,W)\right)\right|\leq 2^{-t}\epsilon N^s.
  \end{equation*}
  Applying the triangle inequality to the sum over $J$ and
  \eqref{eq:delta_phi_bound}, the conclusion of the lemma follows.
\end{proof}

We need the following estimate in our application of the $W$-trick in Lemma \ref{lem:artin_primes_normal_form}. Its proof is, essentially, contained in the proof of \cite[Lemma 1.3]{MR2680398}.
\begin{lemma}\label{lem:prime_residues}
  Let $s,t,L\in\NN$. Let $\Psi=(\psi_1,\ldots,\psi_t):\ZZ^s\to\ZZ^t$ be a finite-complexity system
  of affine-linear forms in which the coefficients of all linear terms are
  bounded in absoute value by $L$. For $q\in\NN$, let
  \begin{equation*}
    A(q):=\left\{n\in (\ZZ/q\ZZ)^s\where \gcd(\psi_i(n),q)=1\text{ for all }i\in[t]\right\}.
  \end{equation*}
  Then
  \begin{equation*}
    |A(q)|\ll_{s,t,L}\left(\frac{\phi(q)}{q}\right)^tq^s.
  \end{equation*}
\end{lemma}

\begin{proof}
  Let us show first that, for any prime $p$ that is sufficiently large in terms
  of $s,t,L$, we have
  \begin{equation}\label{eq:W-trick_A_p_bound}
    |A(p)|\leq \left(\frac{p-1}{p}\right)^tp^s\left(1+O_t\left(\frac{1}{p^2}\right)\right).    
  \end{equation}
  To this end, we define for $i\in[t]$ the set
  \begin{equation*}
    A_i(p):=\{n\in\FF_p^s\where \psi_i(n)=0\}. 
  \end{equation*}
  As $p$ is sufficiently large in terms of $s,t,L$, none of the forms $\psi_i$ will be constant and no two of their linear parts will be linearly dependent over $\FF_p$, which implies that
  \begin{align*}
    |A_i(p)|=p^{s-1}\quad\text{ and }\quad |A_i(p)\cap A_j(p)| = p^{s-2}.
  \end{align*}
  Truncating the inclusion-exclusion formula, we see that
  \begin{align*}
    |A(p)|&\leq |\FF_p^s|-\sum_{i=1}^t|A_i(p)| + \sum_{1\leq i<j\leq
    t}|A_i(p)\cap A_j(p)|\\
          &=p^s-tp^{s-1}+\frac{t(t-1)}{2}p^{s-2}\\
          &= p^s\left(1-\frac{t}{p}+O_t\left(\frac{1}{p^2}\right)\right) =
            p^s\left(\left(\frac{p-1}{p}\right)^t+O_t\left(\frac{1}{p^2}\right)\right),
  \end{align*}
  which shows \eqref{eq:W-trick_A_p_bound}. For any $e\geq 1$, this implies that
  \begin{equation*}
    |A(p^e)| = p^{(e-1)s}|A(p)| \leq
    \left(\frac{p-1}{p}\right)^tp^{es}\left(1+O_t\left(\frac{1}{p^2}\right)\right)
    = \left(\frac{\phi(p^e)}{p^e}\right)^tp^{es}\left(1+O_t\left(\frac{1}{p^2}\right)\right).
  \end{equation*}
  Write $p^e\mid\mid q$ if $e\geq 1$ is the exact exponent with which $p$ divides $q$.
  By the Chinese remainder theorem, we conclude that
  \begin{align*}
    |A(q)|&\ll_{L}\prod_{p^e\mid\mid
    q}\left(\frac{\phi(p^e)}{p^e}\right)^tp^{es}\left(1+O_t\left(\frac{1}{p^2}\right)\right)\\
    &= \left(\frac{\phi(q)}{q}\right)^tq^{s}\prod_{p\mid
    q}\left(1+O_t\left(\frac{1}{p^2}\right)\right) \ll_{t}\left(\frac{\phi(q)}{q}\right)^tq^{s}.\qedhere
  \end{align*}
\end{proof}

\begin{corollary} \label{cor:local_density_bound}
  Let $s,t,L\in\NN$. Let $\Psi=(\psi_1,\ldots,\psi_t):\ZZ^s\to\ZZ^t$ be a finite-complexity system
  of affine-linear forms in which the coefficients of all linear terms are
  bounded in absoute value by $L$. Let $a_1,\ldots,a_t\in\ZZ\smallsetminus\{-1\}$ such that
  no $a_i$ is a perfect square. Then, the densities
  $\sigma_{\va,\Psi}(q)$ defined in \eqref{eq:def_sigma} satisfy
  \begin{equation*}
    \Big(\prod_{i\in[t]}\delta(a_i,0,1)\Big)\sigma_{\va,\Psi}(\cD_{\va})\prod_{p\nmid\cD_{\va}}\sigma_{\va,\Psi}(p)\ll_{\va,s,t,L} \sigma_{\va,\Psi}(\cD_\va)\ll_{\va,s,t,L}1.
  \end{equation*}
\end{corollary}

\begin{proof}
  Using Lemma \ref{lem:euler_factor_bounds}, \eqref{eq:delta_phi_bound} and Lemma \ref{lem:prime_residues},
  we see that the expression on the left-hand side is
  \begin{align*}
    \ll_{\va,s,t,L}\sigma_{\va,\Psi}(\cD_{\va})=\cD_\va^{t-s}\sum_{n\in(\ZZ/\cD_\va\ZZ)^s}\prod_{i\in[t]}\frac{\delta(a_i,\psi_i(n),\cD_\va)}{\delta(a_i,0,1)}\ll_{\va,t}
    \frac{\cD_\va^{t-s}}{\phi(\cD_\va)^t}|A(\cD_\va)|\ll_{s,t,L} 1.
  \end{align*}
\end{proof}

\begin{lemma}[Artin primes in lattices in normal form]\label{lem:artin_primes_normal_form}
  Let $s,t,u,L,N\in\NN$, $a_1,\ldots,a_t\in\ZZ\smallsetminus\{-1\}$ such that
  no $a_i$ is a perfect square, and $\Psi=(\psi_1,\ldots,\psi_t):\ZZ^s\to\ZZ^t$
  a system of affine-linear forms in $u$-normal form with
  $\vecnorm{\Psi}_N\leq L$. Assume $\HRH(a_i)$ for all $1\leq i\leq t$. Let
  $X\subseteq [-N,N]^s$ be a convex body on which
  $\psi_1,\ldots,\psi_t>N^{8/10}$. Then
  \begin{equation*}
    \sum_{n\in
        X\cap\ZZ^s}\left(\prod_{i\in[t]}\Lambda_{a_i}(\psi_i(n))-\Big(\prod_{i\in[t]}\delta(a_i,0,1)\Big)\sigma_{\va,\Psi}(\cD_{\va})\prod_{p\nmid\cD_{\va}}\sigma_{\va,\Psi}(p)\right)=o_{s,t,u,L,\va}(N^s).
  \end{equation*}
\end{lemma}

\begin{proof}
   To prove the lemma, we fix $\epsilon>0$ and show that
    \begin{equation*}
    \left|\sum_{n\in
        X\cap\ZZ^s}\left(\prod_{i\in[t]}\Lambda_{a_i}(\psi_i(n))-\Big(\prod_{i\in[t]}\delta(a_i,0,1)\Big)\sigma_{\va,\Psi}(\cD_{\va})\prod_{p\nmid\cD_{\va}}\sigma_{\va,\Psi}(p)\right)\right|\leq \epsilon N^s
  \end{equation*}
  holds if $N$ is sufficiently large in terms of
  $\epsilon,s,t,u,L,a_1,\ldots,a_t$. We define $W$ as in \S\ref{sec:W_trick}
  with $\cD=\cD_\va\asymp_\va 1$ and a fixed function $w(\cdot)$. The choice of
  this function will be made more specific later in this proof and depend only
  on $\epsilon,s,t,L,u,\va$. Let
  \begin{equation*}
    M:=\vol(X)\Big(\prod_{i\in[t]}\delta(a_i,0,1)\Big)\sigma_{\va,\Psi}(\cD_{\va})\prod_{p\nmid\cD_{\va}}\sigma_{\va,\Psi}(p).
  \end{equation*}
  Using Corollary \ref{cor:local_density_bound} and the fact that $|X\cap
  \ZZ^s| = \vol(X)+o_s(N^s)$, we obtain the estimate
  \begin{equation*}
    \sum_{n\in
      X\cap\ZZ^s}\Big(\prod_{i\in[t]}\delta(a_i,0,1)\Big)\sigma_{\va,\Psi}(\cD_{\va})\prod_{p\nmid\cD_{\va}}\sigma_{\va,\Psi}(p) = M + o_{\va,s,t,L}(N^s).
  \end{equation*}
  Hence, using the triangle inequality it suffices to show that
  \begin{equation}\label{eq:W_trick_triangle_inequality}
     \left|\sum_{n\in X\cap\ZZ^s}\prod_{i\in [t]}\Lambda_{a_i}(\psi_i(n))-M\right|<\frac{\epsilon}{2}N^s
   \end{equation}
  for all large enough $N$. Using Lemma \ref{lem:euler_factor_bounds}, we see
  that
  \begin{align*}
    \prod_{p\nmid\cD_\va}\sigma_{\va,\Psi}(p) &=
    \prod_{\substack{p\nmid\cD_\va\\p\leq
    w(N)}}\sigma_{\va,\Psi}(p)\prod_{\substack{p\nmid\cD_\va\\p>
    w(N)}}\sigma_{\va,\Psi}(p)=\Big(\prod_{\substack{p\nmid\cD_\va\\p\leq
    w(N)}}\sigma_{\va,\Psi}(p)\Big)\left(1+o_{\va,s,t,L,w(\cdot)}(1)\right)\\ &=\prod_{\substack{p\nmid\cD_\va\\p\leq
    w(N)}}\sigma_{\va,\Psi}(p)+o_{s,t,L,\epsilon,u,\va}(1),
  \end{align*}
  as $\prod_{p\nmid\cD_\va}\sigma_{\va,\Psi}(p)\ll_{\va,s,t,L}1$ and the choice of
  our function $w(\cdot)$ will depend only on $\epsilon,t,L,u,\va$. Hence, in
  order to prove \eqref{eq:W_trick_triangle_inequality}, it is enough to show
  that
  \begin{equation*}
        \left|\sum_{n\in X\cap\ZZ^s}\prod_{i\in
            [t]}\Lambda_{a_i}(\psi_i(n))-\vol(X)\Big(\prod_{i\in[t]}\delta(a_i,0,1)\Big)\sigma_{\va,\Psi}(\cD_{\va})\prod_{\substack{p\nmid\cD_{\va}\\p\leq
            w(N)}}\sigma_{\va,\Psi}(p)\right|<\frac{\epsilon}{4}N^s
  \end{equation*}
  holds for large enough $N$. As
  \begin{align*}
    \abs{\sum_{n\in X\cap\ZZ^s}\prod_{i\in
            [t]}\Lambda_{a_i}(\psi_i(n))-\sum_{n\in X\cap\ZZ^s}\prod_{i\in
            [t]}\Lambda_{a_i}'(\psi_i(n))}&\ll_{s,L}(\log
    N)^t\sum_{i=1}^t\sum_{e\geq 2}\sum_{n\in X\cap\ZZ^s}\one_{\psi_i(n)\text{
    is }e\text{-th power}}\\ &\ll_{s,t,L} (\log N)^t\sum_{2\leq e\ll_{s,L}\log N}N^{s-1+1/e}=o_{s,t,L}(N^s),
  \end{align*}
 it is enough to show that, for all large enough $N$,   
  \begin{equation}
    \label{eq:W_trick_truncate_product}
        \left|\sum_{n\in X\cap\ZZ^s}\prod_{i\in
            [t]}\Lambda_{a_i}'(\psi_i(n))-\vol(X)\Big(\prod_{i\in[t]}\delta(a_i,0,1)\Big)\sigma_{\va,\Psi}(\cD_{\va})\prod_{\substack{p\nmid\cD_{\va}\\p\leq
            w(N)}}\sigma_{\va,\Psi}(p)\right|<\frac{\epsilon}{8}N^s
    \end{equation}
    Using Lemma \ref{lem:almost_multiplicativity}, the definition of $W$ in
    \eqref{eq:cheb_def_W} with $\cD=\cD_\va$, and the definition of
    $\sigma_{\va,\Psi}(\cdot)$ in \eqref{eq:def_sigma}, we see that
    \begin{equation}\label{eq:W_trick_splitting_RHS}   
    \begin{aligned}
      &\Big(\prod_{i\in[t]}\delta(a_i,0,1)\Big)\sigma_{\va,\Psi}(\cD_{\va})\prod_{\substack{p\nmid\cD_{\va}\\p\leq
          w(N)}}\sigma_{\va,\Psi}(p) =
      \Big(\prod_{i\in[t]}\delta(a_i,0,1)\Big)\sigma_{\va,\Psi}(W)\\ &=W^{t-s}\sum_{c\in(\ZZ/W\ZZ)^s}\prod_{i\in
      [t]}\delta(a_i,\psi_i(c),W)=W^{t-s}\sum_{c\in A(W)}\prod_{i\in
      [t]}\delta(a_i,\psi_i(c),W),
  \end{aligned}
  \end{equation}
  with $A(W)$ defined as in Lemma \ref{lem:prime_residues}. We also split the left-hand side of the difference in
  \eqref{eq:W_trick_truncate_product} into residue classes modulo $W$ to obtain
  \begin{align*}
    \sum_{n\in X\cap\ZZ^s}\prod_{i\in
            [t]}\Lambda_{a_i}'(\psi_i(n))&=\sum_{c\in[W]^s}\sum_{\substack{n\in\ZZ^s\\Wn+c\in
    X}}\prod_{i\in[t]}\Lambda_{a_i}'(\psi_i(Wn+c))\\ &=\sum_{c\in[W]^s}\sum_{\substack{n\in\ZZ^s\\Wn+c\in
X}}\prod_{i\in[t]}\frac{W}{\phi(W)}\Lambda_{a_i,b_i(c),W}'(\psi_{i,c}(n)),
  \end{align*}
  where, writing $\psi_i(c)=Wh_1+h_0$ with $h_0\in\{0,\ldots,W-1\}$, we set
  \begin{equation*}
    \psi_{i,c}(n) := \dot\psi_i(n)+h_1\quad\text{ and }\quad b_i(c) := h_0,
  \end{equation*}
  so that $\psi_i(Wn+c)=W\psi_{i,c}(n)+b_i(c)$. Note that if
  $\gcd(W,\psi_i(c))>1$, then $\gcd(\psi_i(Wn+c),W)>1$ for any $Wn+c\in X$. As
  $W\ll_\va (\log N)^{2/3}$ and $\psi_i\geq N^{8/10}$ on $X$, this shows that
  $\psi_i(Wn+c)$ can not be prime if $N$ is large enough. Hence, identifying $[W]^s$ with
  $(\ZZ/W\ZZ)^s$ in the obvious way, we see that
    \begin{equation}\label{eq:W_trick_splitting_LHS}
    \sum_{n\in X\cap\ZZ^s}\prod_{i\in
            [t]}\Lambda_{a_i}'(\psi_i(n))=\sum_{c\in A(W)}\sum_{\substack{n\in\ZZ^s\\Wn+c\in
    X}}\prod_{i\in[t]}\frac{W}{\phi(W)}\Lambda_{a_i,b_i(c),W}'(\psi_{i,c}(n)).
    \end{equation}
    Finally, as $\psi_i(c)\equiv b_i(c)\bmod W$, we see that
    $\delta(a_i,\psi_i(c),W)=\delta(a_i,b_i(c),W)$. Using this observation
    together with \eqref{eq:W_trick_splitting_RHS},
    \eqref{eq:W_trick_splitting_LHS} and the triangle inequality, we see that
    the left-hand side of \eqref{eq:W_trick_truncate_product} is bounded from
    above by
    \begin{align*}
      &\sum_{c\in A(W)}\left|\sum_{\substack{n\in\ZZ^s\\Wn+c\in
    X}}\prod_{i\in[t]} \frac{W}{\phi(W)}\Lambda_{a_i,b_i(c),W}'(\psi_{i,c}(n))- W^{t-s}\vol(X)\prod_{i\in
      [t]}\delta(a_i,b_i(c),W) \right|\\ &=\left(\frac{W}{\phi(W)}\right)^t\sum_{c\in
                                           A(W)}\left|\sum_{\substack{n\in\ZZ^s\cap
      \frac{X-c}{W}}}\prod_{i\in[t]}\Lambda_{a_i,b_i(c),W}'(\psi_{i,c}(n))- \vol\left(\frac{X-c}{W}\right)\prod_{i\in
      [t]}\phi(W)\delta(a_i,b_i(c),W) \right|.
    \end{align*}
    Hence, in the light of Lemma \ref{lem:prime_residues}, in order to prove
    \eqref{eq:W_trick_truncate_product}, and thus the lemma, it is enough to
    show that, for every $c\in A(W)$, 
    \begin{equation}\label{eq:W_trick_final_application}
      \left|\sum_{\substack{n\in\ZZ^s\cap
      \frac{X-c}{W}}}\prod_{i\in[t]}\Lambda_{a_i,b_i(c),W}'(\psi_{i,c}(n))- \vol\left(\frac{X-c}{W}\right)\prod_{i\in
      [t]}\phi(W)\delta(a_i,b_i(c),W) \right| < \frac{\epsilon}{8C(s,t,L)}\left(\frac{N}{W}\right)^s,
\end{equation}
where $C(s,t,L)$ is the implied constant in Lemma \ref{lem:prime_residues},
    if only $N$ (and thus $N/W$ by \eqref{eq:cheb_def_W}) is large enough.
    Note that $\dot\psi_{i,c}=\dot\psi_i$ and
    \begin{equation*}
      \abs{\psi_{i,c}(0)}=\abs{h_1}\leq \frac{\psi_i(c)}{W}\leq
      \frac{\dot\psi_i(c)}{W}+\frac{\psi_i(0)}{W}\ll_{s,L}1+\frac{N}{W}\leq \frac{N}{W},
    \end{equation*}
    so the forms are still in $u$-normal form and satisfy
    $\vecnorm{\psi_i(c)}_{N/W}\ll_{s,L}1$.

    Hence, \eqref{eq:W_trick_final_application} follows from follows from Lemma
    \ref{lem:telescoping} applied with $2N/W$, $(X-c)/W,\psi_{i,c}$ and
    some $\tilde L\ll_{s,L}1$
    instead of $N,X,\psi_i,L$, and with $\epsilon/C(t,L)$ instead of $\epsilon$.
    The remaining hypotheses of Lemma \ref{lem:telescoping} are
    satisfied, if only $w(\cdot)$ was chosen to grow sufficiently slowly in terms of
    $\epsilon,s,t,L,u,\va$.
  \end{proof}

  \subsection{Completion of the proof of Theorem \ref{thm:main_artin}}\label{sec:completion_of_proof_artin}
  In \cite[\S4]{MR2680398}, the main theorem of \cite{MR2680398} is
  first reduced to \cite[Theorem 4.1]{MR2680398}, removing the archimedean
  factor, and then futher to \cite[Theorem 4.5]{MR2680398}, extending $\Psi$ to a system in normal form.

  Analogous arguments reduce our Theorem \ref{thm:main_artin} to our Lemma
  \ref{lem:artin_primes_normal_form}. For removing the archimedean factor, we may of course also bound
  $\Lambda_{a_i}(n)$ by $\log N$, and the non-archimedean part of our leading
  constant is also bounded by Corollary \ref{cor:local_density_bound}. Of
  course this would also follow directly from the boundedness of the product
  $\prod_p\beta_p$ in \cite{MR2680398}.

  For the reduction from $\Psi$ to the system $\Psi'$ in normal form, note that
  our densities, defined in \eqref{eq:def_sigma}, obviously also satisfy
  $\sigma_{\va,\Psi'}(q)=\sigma_{\va,\Psi}(q)$.

\section{Constellations  with prescribed Artin symbols: proof of Proposition \ref{prop:gowers_estimate_cheb}, Theorem \ref{thm:main_cheb}}\label{sec:proof_main_cheb}

Here we indicate how to prove the remaining results of this paper
concerning the functions $\Lambda_{K,C}$. As
most of the proofs are analogous, but simpler, to what we have already
done for $\Lambda_a$ in the previous sections, we will be very
brief.

\subsection{Gowers-norm estimate: proof of Proposition \ref{prop:gowers_estimate_cheb}}

Clearly, the proof of Lemma \ref{lem:pseudorandom_majorant} also
yields an analogous version of the majorant where the functions $\Lambda_{a_i,b_i,W}$
are replaced by functions $\Lambda_{K_i,C_i,b_i,W}$ with Galois number
fields $K_i$ and conjugacy classes $C_i$ in their respective Galois
groups over $\QQ$. Hence, we may use Theorem
\ref{thm:W_tricked_chebotarev_mangoldt_with_nilsequences_nonconst_avg}
and the inverse theorem for Gowers norms to obtain a version of Lemma
\ref{lem:gowers_estimate_artin_nonconst_avg}, following the proof
of Lemma \ref{lem:gowers_estimate_artin_nonconst_avg}. Our
situation here is even slightly simpler, as we can fix
$w(N)=\log\log\log N$ once and for all, so in particular our choice of
$w(N)$ does not need to depend on $\delta$.

Proposition \ref{prop:gowers_estimate_cheb} then follows from this version of Lemma \ref{lem:gowers_estimate_artin_nonconst_avg}, together with Lemma \ref{lem:gowers_estimate_classical} and the triangle inequality. 

\subsection{Proof of Theorem \ref{thm:main_cheb}}
As we have a pseudorandom majorant already, the deduction of Theorem \ref{thm:main_cheb} from Proposition
\ref{prop:gowers_estimate_cheb} follows the same steps as in \S\S \ref{sec:von_neumann}--\ref{sec:completion_of_proof_artin} with the obvious modifications. For example, we rely on the results concerning $\tau_{\vK,\vC,\Psi}(q)$ established in \S\ref{sec:local_densities_cheb} instead of the analogous properties of $\sigma_{\va,\Psi}(q)$ established in \S\ref{sec:delta}. Again, the situation is somewhat simpler, as the function $w(\cdot)$ is fixed once and for all and does not depend on any of the other parameters. The argument is essentially the same as in \cite[\S\S 4-7]{MR2680398}.

\subsection{Proof of Corollary \ref{cor:existence_cheb}}
For $N\geq 1$, the convex set $X_N:=X\cap [-N,N]^s$ satisfies that $\vol(X\cap\Psi^{-1}(\RR_+^t))\asymp_{X,\Psi} N^s$ by condition \emph{(3)}. By condition \emph{(1)}, we may apply Theorem \ref{thm:main_cheb} with the system $\Psi$ and this convex set. In the formula for $\fS(\vK,\vC,\Psi)$ given in \eqref{eq:cheb_leading_term}, the factors $\beta_p$ are positive and independent of $N$ due to condition \emph{(2)}. The factor
\begin{equation*}
  \EE_{n\in(\ZZ/\cD_\vK\ZZ)^s}\prod_{i\in[t]}\cD_{\vK}\eta_{K_i,C_i}(\psi_i(n),\cD_\vK)
\end{equation*}
is positive, as condition \emph{(4)} implies the existence of at least one $n\in(\ZZ/\cD_{\vK}\ZZ)^s$ for which the corresponding summand is positive. It is clearly also independent of $N$. Hence, the main term has size $\asymp_{X,\Psi,\vK}(N^s)$. The contribution of those $n\in ([-N,N]\cap\ZZ)^s$ for which one of the forms $\psi_i(n)$ takes a proper prime power value, or where two of the forms take the same value, is easily seen to be $o_{\Psi}(N^s)$. Hence, the sum over those $n\in X_N\cap \ZZ^s$ for which the $\psi_i(n)$ take distinct prime values is still $\asymp_{X,\Psi,\vK}(N^s)$. Letting $N\to\infty$, we obtain infinitely many such $n\in X\cap\ZZ^s$.

\section{Previous results and examples} \label{sec:examples}  

\subsection{One linear equation in primes with prescribed Artin symbols} \label{subsec:comparison_Kane}

Here we show that the main term Kane's result \cite[Theorem 2]{MR3060874} for
$t\geq 3$ matches up with a
special case or our Theorem \ref{thm:main_artin}. We may formulate Kane's asymptotic as follows: for $t\geq 3$ and $\vK,\vC$ as in
Theorem \ref{thm:main_cheb}, $\vc=(c_1,\ldots,c_t)\in\ZZ^t$ with
$\gcd(c_1,\ldots,c_t)=1$ and $N,M\in\NN$, we have
\begin{equation*}
     \sum_{\substack{x_i\in\ZZ\cap[0, N] \\ c_1 x_1+\cdots+c_t x_t = M}} \prod_{i=1}^t
     \Lambda_{K_i,C_i}(x_i)=\left(\prod_{i=1}^t\frac{|C_i|}{[K_i:\QQ]}\right)\tau_\infty
     \tau_{\Delta_{\vK}}\prod_{p\nmid \Delta_{\vK}}\tau_p + O_{\vK,\vc,A}(N^{t-1}(\log N)^{-A}),
\end{equation*}
where $\Delta_\vK=\lcm(|\Delta_{K_i^\ab}|\ :\ 1\leq i\leq t)$ is the least
common multiple of the discriminants of the maximal abelian subfields $K_i^\ab$
of $K_i$ and
\begin{align*}
\tau_\infty &= \int_{\substack{x_i\in [0,N] \\ \sum_i c_i x_i = M}} \left(\sum_{i=1}^t \frac{c_i}{\vecnorm{\vc}_2^2} \frac{\partial}{\partial x_i} \right)dx_1 \wedge dx_2 \wedge \ldots \wedge dx_t,\\
\tau_{\Delta_{\vK}} &= \frac{\Delta_{\vK}\prod_{i=1}^t[K_i^\ab:\QQ]}{\phi(\Delta_\vK)^t}\left|\left\{\vx\in
                      \left((\ZZ/\Delta_{\vK}\ZZ)^{\times}\right)^t \where
                      \sigma_{x_i}|_{K_i^\ab}\in C_i|_{K_i^\ab},\  \sum_{i=1}^t c_i x_i \equiv M \bmod \Delta_{\vK} \right\}\right|,\\
\tau_p &= \frac{p}{\phi(p)^t}\left|\left\{\vx\in \left((\ZZ/p\ZZ)^{\times}\right)^t\where
\sum_{i=1}^t c_i x_i \equiv M \bmod p \right\}\right|.
\end{align*}
Here, we have identified $\left(\ZZ/\Delta_\vK\ZZ\right)^\times\simeq\Gal(\QQ(\mu_{\Delta_\vK})/\QQ)$ via $a\mapsto
\sigma_a$.

We may assume that $M\ll N$. Following the proof of \cite[Theorem 1.8]{MR2680398}, we construct a system of
affine-linear forms $\Psi:\ZZ^{t-1}\to \ZZ^t$ of complexity $\leq 1$ and size
$\vecnorm{\Psi}\ll 1$, such that $\Psi$ is injective
with image exactly the full affine sublattice of $\ZZ^t$ defined by $\vc\cdot
\vx=M$. Indeed, we can find an element $\vy$ of this affine sublattice of size $\abs{\vy}\ll N$, and a $t\times(t-1)$-matrix $A$ of rank $t-1$ with entries in $\ZZ$ bounded in absolute value by $O(1)$, such that $A\cdot \ZZ^{t-1}$ is the kernel of $\vx\mapsto\vc\cdot\vx$. Then the affine-linear system
\begin{align*}
  \Psi:\ZZ^{t-1}&\to\ZZ^t\\
  n&\mapsto A\cdot n + \vy
\end{align*}
does what we want. Let us compare the main term in Theorem \ref{thm:main_cheb} for this system
$\Psi$ and the convex set $X=\Psi^{-1}([0,N]^t)$ to the above. Let $V\subseteq \RR^t$ be the affine subspace defined by $\vc\cdot\vx=M$. The $(n-1)$-form in the definition of $\tau_\infty$, let us call it $\omega$, was chosen in the proof of \cite[Proposition 25]{MR3060874} so that $V\cap\ZZ^t$ has covolume $1$ in $V$. Hence, $\int_{[0,1)^{t-1}}\Psi^*(\omega)=\int_{\Psi([0,1)^{t-1})}\omega = 1$, and therefore $\tau_\infty=\int_{\Psi(X)}\omega=\vol(X)$.

Note that Kane's $\Delta_{\vK}$ and our $\cD_{\vK}$ have the same prime factors.  
We claim that for any $q\in\NN$ the $\ZZ/q\ZZ$-affine-linear map
\begin{equation}\label{eq:kane_p_map}
  (\ZZ/q\ZZ)^{t-1}\to \{\vx\in\left(\ZZ/q\ZZ\right)^t\where \vc\cdot\vx\equiv M\bmod q\}
\end{equation}
induced by $\Psi$ is bijective. Indeed, the sequence
\begin{equation*}
  0\to \ZZ^{t-1} \xrightarrow{A\cdot} \ZZ^t \xrightarrow{\vc\cdot} \ZZ \to 0
\end{equation*}
is exact, as $\gcd(c_1,\ldots,c_t)=1$. Tensoring with $\ZZ/q\ZZ$, we get an exact sequence
\begin{equation*}
  (\ZZ/q\ZZ)^{t-1}\xrightarrow{A\cdot}(\ZZ/q\ZZ)^t\xrightarrow{\vc\cdot}\ZZ/q\ZZ\to 0,
\end{equation*}
which shows that the map in \eqref{eq:kane_p_map} is surjective and its image has cardinality $q^{t-1}$. Hence, it is bijective.
Taking $q=p$ for any prime $p\nmid\Delta_{\vK}$, the bijectivity of \eqref{eq:kane_p_map} shows that
\begin{equation*}
  \tau_p= \frac{p}{\phi(p)^t}\left|\left\{n\in(\ZZ/p\ZZ)^{t-1}\ :\ \gcd(\psi_i(n),p)=1\text{ for all }1\leq i\leq t \right\}\right|=\beta_p.
\end{equation*}
With \eqref{eq:cheb_leading_term}, it remains to show that
\begin{equation}\label{eq:kane_tau_eq} \left(\prod_{i=1}^t\frac{|C_i|}{[K_i:\QQ]}\right)\tau_{\Delta_{\vK}}=\EE_{n\in(\ZZ/\cD_\vK\ZZ)^s}\prod_{i=1}^t\cD_{\vK}\eta_{K_i,C_i}(\psi_i(n),\cD_\vK).
\end{equation}
As $\cD_{\vK}\mid \Delta_{\vK}$ and they have the same prime divisors, we see for any $x_i\in\ZZ/\Delta_\vK\ZZ$ that 
\begin{equation*}
  \eta_{K_i,C_i}(x_i,\cD_{\vK}) = \frac{|C_i|}{[K_i(\mu_{\cD_{\vK}}):\QQ]}=\frac{|C_i|}{[K_i:\QQ]}\frac{[\QQ(\mu_{\Delta_{\vK}}):\QQ]}{[\QQ(\mu_{\cD_{\vK}}):\QQ]}
 \frac{[K_i^{ab}:\QQ]}{[\QQ(\mu_{\Delta_{\vK}}):\QQ]}=\frac{|C_i|}{[K_i:\QQ]}\frac{\Delta_{\vK}}{\cD_{\vK}}\frac{[K_i^\ab:\QQ]}{\phi(\Delta_\vK)}
\end{equation*}
if $\gcd(x_i,\Delta_{\vK})=1$ and $x_i|_{K_i^\ab}\in C_i|_{K_i^\ab}$, and $\eta_{K_i,C_i}(x_i,\cD_{\vK})=0$ otherwise. Hence, using again the bijectivity of \eqref{eq:kane_p_map}, we obtain
\begin{align*}
  \left(\prod_{i=1}^t\frac{|C_i|}{[K_i:\QQ]}\right)\tau_{\Delta_\vK}=\frac{1}{\Delta_\vK^{t-1}}\sum_{n\in(\ZZ/\Delta_\vK\ZZ)^{t-1}}\prod_{i=1}^t\cD_\vK\eta_{K_i,C_i}(\psi_i(n),\cD_\vK),
\end{align*}
which shows \eqref{eq:kane_tau_eq}, as desired.

\subsection{Three primes theorem with prescribed primitive roots}\label{subsec:comparison_FKS}
Now, we will show that \cite[Theorem 1.1]{MR4201547} agrees with a special case of Theorem \ref{thm:main_artin}. We may formulate the this result as follows: let $\va=(a_1,a_2,a_3)$ where $a_{i}\neq -1$ is not a perfect square and assume $\HRH(a_i)$ for all $1\leq i\leq 3$. Then, for $N\in\NN$,
\begin{equation} \label{eq:FKS}
\sum_{\substack{x_i\in\NN\\x_1+x_2+x_3=N}}\prod_{i=1}^{3}\Lambda_{a_i}(x_i)=\frac{N^2}{2}\left(\prod_{i=1}^3\delta(a_i,0,1)\right)\tilde{\sigma}_{\va,N}(\tilde{\cD}_{\va})\prod_{p\nmid \tilde{\cD}_{\va}} \tilde{\sigma}_{\va,N}(p)+o(N^2),
\end{equation}
where
\[
\tilde{\cD}_{\va}=2^{\min\{\nu_2(\Delta_{a_i}): 1\leq i\leq 3\}-\max\{\nu_2(\Delta_{a_i}):1\leq i\leq 3\}}\lcm(\Delta_{a_1},\Delta_{a_2},\Delta_{a_3}),
\]
and
\[
\tilde{\sigma}_{\va,N}(q)=q\left(\sum_{\substack{x_1,x_2,x_3\bmod q\\x_1+x_2+x_3\equiv N\bmod q}}\prod_{i=1}^3\frac{\delta(a_i,x_i,q)}{\delta(a_i,0,1)}\right).
\]

Similarly as in \S\ref{subsec:comparison_Kane}, we compare this to Theorem \ref{thm:main_artin} with the same $\va$, the affine-linear system $\Psi(n_1,n_2)=(n_1,n_2,N-n_1-n_2)$ and the convex subset $\{x_1,x_2\geq 0,\ x_1+x_2\leq N\}\subseteq\RR^2$ of volume $N^2/2$. For every $q\in\NN$, the bijection
\begin{align*}
  (\ZZ/q\ZZ)^2&\to \{(x_1,x_2,x_3)\in(\ZZ/q\ZZ)^3\ :\ x_1+x_2+x_3=N\}\\
   (n_1,n_2)&\mapsto (n_1,n_2,N-n_1-n_2)
\end{align*}
shows that $\tilde\sigma_{\va,N}(q)=\sigma_{\va,\Psi}(q)$ for all $q\in\NN$. For even $N$, this quantity is clearly equal to $0$ whenever $2\mid q$, hence, we may assume that $N$ is odd. The only significant difference between \eqref{eq:FKS} and the asymptotic from Theorem \ref{thm:main_artin} is that between $\tilde\cD_\va$ and $\cD_\va$. Let $m:=\min\{\nu_2(a_i):1\leq i\leq 3\}$ and $M:=\max\{\nu_2(a_i):1\leq i\leq 3\}$, so $\tilde\cD_{\va}=2^mq$ and $\cD_\va=2^Mq$ for some odd $q$. We need to show that
\begin{align}
  \tilde\sigma_{\va,N}(2^Mq)&=\tilde\sigma_{\va,N}(2^mq), & &\text{ if }m\geq 2\text{ or }m=M,\label{eq:FKS_goal_1}\\
  \tilde\sigma_{\va,N}(2^Mq)&=\tilde\sigma_{\va,N}(2)\tilde\sigma_{\va,N}(q), & &\text{ if } m=0\text{ and }M\geq 1.\label{eq:FKS_goal_2}
\end{align}
This is clear if $m=M$, so assume from now that $m<M$. Using that $2^{-v_2(\Delta_{a_i})}\Delta_{a_i}\mid q$ for all $1\leq i\leq 3$, one can see from Moree's explicit description in \S\ref{sec:delta} that, for any $l\geq 1$, $a\in\{a_1,a_2,a_3\}$ and $b\in\ZZ$, we have 
\begin{equation*}
  \delta(a,b,2^lq) =
  \begin{cases}
    0& \text{ if }b\equiv 0\bmod 2,\\
        2^{-(l-1)}\delta(a,b,q)\left(1-\left(\frac{\Delta_a}{b}\right)\right)& \text{ if }b\not\equiv 0\bmod 2\text{ and }1\leq v_2(a)\leq l,\\
    2^{-(l-1)}\delta(a,b,q)& \text{ if }b\not\equiv 0\bmod 2\text{ otherwise.}
  \end{cases}
\end{equation*}
Moreover, as $N$ is odd, from the same description we compute $\tilde\sigma_{\va,N}(2)=2$. Without loss of generality, we assume that $\nu_2(\Delta_{a_1})=m$ and $\nu_2(\Delta_{a_3})=M$. Then for any $l\geq 1$,
\begin{align}
  \tilde\sigma_{\va,N}(2^lq)&=2^lq\sum_{\substack{c_1,c_2,c_3\bmod q\\c_1+c_2+c_3\equiv N\bmod q}}\prod_{i=1}^3\frac{\delta(a_i,c_i,q)}{2^{l-1}\delta(a_i,0,1)}\sum_{\substack{x_1,x_2,x_3\bmod 2^l q\\x_1+x_2+x_3\equiv N\bmod 2^l\\\forall i:\ x_i\equiv c_i\bmod q\\\forall i:\ x_i\equiv 1\bmod 2}}\prod_{\substack{1\leq i\leq 3\\1\leq v_2(\Delta_{a_i})\leq l}}\left(1-\left(\frac{\Delta_{a_1}}{x_i}\right)\right).\label{eq:sigma_tilde}
\end{align}
If $m=0$, then $v_2(\Delta_{a_1})=0$ and thus the inner sum is equal to
\begin{align}\label{eq:sigma_tilde_inner}
  \prod_{\substack{i\in\{2,3\}}}\sum_{\substack{x_i\bmod 2^lq\\x_i\equiv c_i\bmod q\\x_i\equiv 1\bmod 2}}\left(1-\one_{0<v_2(\Delta_{a_i})\leq l}\left(\frac{\Delta_{a_i}}{x_i}\right)\right)=2^{2(l-1)}.
\end{align}
Here, we have used that $\left(\frac{\Delta_{a_i}}{\cdot}\right)$ is a primitive character modulo $\Delta_{a_i}$ and $\Delta_{a_i}\nmid 2q$ if $\nu_{2}(\Delta_{a_i})>0$, so that then, for $b_i\in\ZZ$ satisfying $b_i\equiv c_i\bmod q$ and $b_i\equiv 1\bmod 2$,
\begin{equation*}
  \sum_{\substack{x_i\bmod 2^lq\\x_i\equiv c_i\bmod q\\x_i\equiv 1\bmod 2}}\left(\frac{\Delta_{a_i}}{x_i}\right) =\hspace{-0.5cm}\sum_{\substack{y_i\bmod \Delta_{a_i}\\y_i\equiv b_i\bmod (2q,\Delta_{a_i})}}\hspace{-0.5cm}\left(\frac{\Delta_{a_i}}{x_i}\right)\sum_{\substack{x_i\bmod 2^lq\\x_i\equiv y_i\bmod \Delta_{a_i}\\x_i\equiv b_i\bmod 2q}}\hspace{-0.5cm}1= \frac{2^lq}{\lcm(\Delta_{a_i},2q)}\hspace{-0.5cm}\sum_{\substack{y_i\bmod \Delta_{a_i}\\y_i\equiv b_i\bmod (2q,\Delta_{a_i})}}\left(\frac{\Delta_{a_i}}{x_i}\right)=0.
\end{equation*}
Inserting \eqref{eq:sigma_tilde_inner} back into \eqref{eq:sigma_tilde}, we obtain in case $m=0$ that
\begin{equation}\label{eq:sigma_tilde_factor}
  \tilde\sigma_{\va,N}(2^lq)=2q \sum_{\substack{c_1,c_2,c_3\bmod q\\c_1+c_2+c_3\equiv N\bmod q}}\prod_{i=1}^3\frac{\delta(a_i,c_i,q)}{\delta(a_i,0,1)}=\tilde\sigma_{\va,N}(2)\tilde\sigma_{\va,N}(q).
\end{equation}
With $l=M$, this shows \eqref{eq:FKS_goal_2}. Now we assume that $1< m \leq l\leq M$. Then the inner sum in \eqref{eq:sigma_tilde} is equal to
\begin{align*}
  \sum_{\substack{y_1,y_2,y_3\bmod 2^m q\\y_1+y_2+y_3\equiv N\bmod 2^m\\\forall i:\ y_i\equiv c_i\bmod q\\\forall i:\ y_i\equiv 1\bmod 2}}\prod_{i=1}^2\left(1-\one_{\nu_2(\Delta_{a_i})=m}\left(\frac{\Delta_{a_i}}{y_i}\right)\right)\prod_{i=2}^3\sum_{\substack{x_i\bmod 2^l q\\x_i\equiv y_i\bmod 2^m q}}\left(1-\one_{m<\nu_2(\Delta_{a_i})\leq l}\left(\frac{\Delta_{a_i}}{x_i}\right)\right).
\end{align*}
As before, the inner sum in this expression is equal to $2^{l-m}$, as the Kronecker symbol is a primitive character modulo $\Delta_{a_i}$ and $\Delta_{a_i}$ does not divide $2^mq$ if $\nu_2(\Delta_{a_i})>m$. Hence, the expression above becomes
\begin{align*}
  2^{2(l-m)}\prod_{i=1}^2\sum_{\substack{y_i\bmod 2^m q\\y_i\equiv c_i\bmod q\\y_i\equiv 1\bmod 2}}\left(1-\one_{\nu_2(\Delta_{a_i})=m}\left(\frac{\Delta_{a_i}}{y_i}\right)\right)=2^{2(l-1)},
\end{align*}
for analogous reasons as before. Inserting this into \eqref{eq:sigma_tilde}, this shows \eqref{eq:sigma_tilde_factor} also in our current situation, and thus \eqref{eq:FKS_goal_1}.

\subsection{Three-term arithmetic progressions in primes with prescribed
  primitive roots}\label{sec:prog_example}

Here we compute an explicit form for the leading constant $\fS(\va,\Psi)$ in
Theorem \ref{thm:main_artin} 
for the system of affine-linear forms
\begin{equation*}
\Psi(n_1,n_2)=(n_1,n_1+n_2,n_1+2n_2)
\end{equation*}
and with $a_1=a_2=a_3=2$. In this case, we have $\cD_{\va}=\Delta_2=8$, $h_2=1$. From Hooley's formula
\eqref{eq:hooley_density}, we see that $\delta(2,0,1)=\cA_2$. Using Moree's
expression from \S\ref{sec:delta} to evaluate all $\delta(2,\psi_i(n),8)$, we find that
\begin{equation*}
  \sigma_{(2,2,2),\Psi}(8)=8\sum_{n\in(\ZZ/8\ZZ)^2}\prod_{i=1}^3\frac{\delta(2,\psi_i(n),8)}{\cA_2}=2.
\end{equation*}
It remains to discuss $\sigma_{(2,2,2),\Psi}(p)$ for $p> 2$, given by
\begin{equation} \label{eq:ex2_p}
\sigma_{(2,2,2),\Psi}(p)=p\sum_{n\in(\ZZ/p\ZZ)^2}\prod_{i=1}^3\frac{\delta(2,\psi_i(n),p)}{\cA_2},
\end{equation}
with
\[
\delta(2,\psi_i(n),p)=\cA_2(\psi_i(n),p)=\frac{\cA_2}{p-1}\left(1-\frac{1}{p(p-1)}\right)^{-1}\cdot\begin{cases}\left(1-\frac{1}{p}\right)&\text{if
}p|\psi_i(n)-1,\\1 &\text{if }p\nmid
\psi_i(n)(\psi_i(n)-1),\\0&\text{otherwise}.\end{cases}
\]
To get a non-zero summand in (\ref{eq:ex2_p}), we need to have
\begin{equation}
p\nmid n_1(n_1+n_2)(n_1+2n_2),\label{eq:progressions_example_gcd}
\end{equation}
which is true for $(p-1)(p-2)$ vectors
$(n_1,n_2)\in (\ZZ/p\ZZ)^2$.  For such $(n_1,n_2)$, we consider the congruences
\begin{align} n_1\equiv 1\bmod p, \label{al_ex2:1}\\ n_1+n_2\equiv 1\bmod
p, \label{al_ex2:2}\\ n_1+2n_2\equiv 1\bmod p, \label{al_ex2:3}
\end{align}
and obtain the following three cases:
\begin{enumerate}
\item All three congruences are satisfied if and only if $(n_1,n_2)=(1,0)$.
\item There are no $(n_1,n_2)$ for which exactly two of the congruences are satisfied.
\item There are $3(p-3)$ pairs $(n_1,n_2)$ satisfying
  \eqref{eq:progressions_example_gcd} for which exactly one of the congruences
  is satisfied. Indeed, if this congruence is \eqref{al_ex2:1}, then
  $n_1\equiv 1\bmod p$ and $n_2\not \equiv 0\bmod p$. Together with
  \eqref{eq:progressions_example_gcd}, this leaves $p-3$ choices for $n_2$. The
  other cases are similar.
\item For the remaining $(p-1)(p-2)-1-3(p-3)=p^2-6p+10$ values of $(n_1,n_2)$ satisfying
  \eqref{eq:progressions_example_gcd}, none of the above congruences are satisfied.
\end{enumerate}
Therefore, from \eqref{eq:ex2_p} we get
\begin{align*}
\sigma_{(2,2,2),\Psi}(p)&=\frac{p}{(p-1)^3}\left(1-\frac{1}{p(p-1)}\right)^{-3}\Bigg((p^2-6p+10+3(p-3)\left(1-\frac{1}{p}\right)+\left(1-\frac{1}{p}\right)^3\Bigg).\\
                        &=1-\frac{p^4-p^3-3p^2-2p-1}{(p^2-p-1)^3}.
\end{align*}
Putting everything together, we obtain the formula stated in Example \ref{ex:2}.

If we fix the system of affine-linear forms
$\Psi(n_1,n_2)=(n_1,n_1+n_2,n_1+2n_2)$ and consider other values of $a_i$
with $h_{a_i}=1$, the computations are similar as in the previous proof. In particular, if $p\nmid \cD_{\va}$, we
obtain
\[
\frac{\delta(a_i,\psi_i(n),p)}{\delta(a_i,0,1)}=\frac{\cA_{a_i}(\psi_i(n),p)\left(1+\mu(2|\Delta_{a_i}|)\left(\frac{1}{\psi_i(n)}\right)f_a^{\ddagger}(|\Delta_{a_i}|)\right)}{\cA_{a_i}(1+\mu(2|\Delta_{a_i}|)f_{a_i}^{\ddagger}(|\Delta_{a_i}|))}=\frac{\cA_{a_i}(\psi_i(n),p)}{\cA_{a_i}},
\]
and thus the same analysis as above shows that
\[
\prod_{p\nmid\cD_{\va}}\sigma_{\va,\Psi}(p)=\left(1-\frac{p^4-p^3-3p^2-2p-1}{(p^2-p-1)^3}\right).
\]
Table \ref{tab:examples1} below summarizes the value of the remaining part of
$\fS(\va,\Psi)$ for some choices of $\va$. Note that in case $\va=(2,3,6)$,
even though the main term is zero, there are integers $(n_1,n_2)$ such that
$(n_1,n_1+n_2,n_1+2n_2)$ are primes with primitive roots $2,3,6$, respectively,
such as $(3,4), (3,28), (3,40)$. Inserting $n_1=3$ and taking $p=3+n_2$
turns $\Psi$ into the system $(p,2p-3)$ of infinite complexity, about which we
can not say anything. In particular, even under GRH we don't know whether there are infinitely
many pairs $(n_1,n_2)$ as above.

\begin{table} \centering \renewcommand{\arraystretch}{1.5}
\begin{tabular}{|c|c|c|c|} \hline $a_1$ & $a_2$ & $a_3$ &
$\prod_{i=1}^3\delta(a_i,0,1)\sigma_{\va,\Psi}(\cD_{\va})$\\ \hline $2$ & $2$ &
$3$ & $\frac{42}{25}\cA_2^2\cA_3$\\ \hline $3$ & $3$ & $3$ &
$\frac{84}{25}\cA_3^3$\\ \hline $2$ & $2$ & $5$ &
$\frac{13500}{6859}\cA_2^2\cA_5$\\ \hline $3$ & $3$ & $5$ &
$\frac{11340}{6859}\cA_3^2\cA_5$\\ \hline $5$ & $5$ & $5$ &
$\frac{10000}{6859}\cA_5^3$\\ \hline $2$ & $3$ & $5$ &
$\frac{11340}{6859}\cA_2\cA_3\cA_5$\\ \hline $2$ & $2$ & $6$ &
$\frac{42}{25}\cA_2^2\cA_6$\\ \hline $2$ & $3$ & $6$ & $0$\\ \hline $6$ & $6$ &
$6$ & $\frac{42}{25}\cA_6^3$\\ \hline $2$ & $5$ & $6$ &
$\frac{11760}{6859}\cA_2\cA_5\cA_6$\\ \hline
\end{tabular} \hspace{0.5cm}
\begin{tabular}{|c|c|c|c|} \hline $a_1$ & $a_2$ & $a_3$ &
$\prod_{i=1}^3\delta(a_i,0,1)\sigma_{\va,\Psi}(\cD_{\va})$\\ \hline $7$ & $7$ &
$7$ & $\frac{161084}{68921}\cA_7^3$\\ \hline $2$ & $2$ & $7$ &
$\frac{134050}{68921}\cA_2^2\cA_7$\\ \hline $2$ & $3$ & $7$ &
$\frac{112602}{68921}\cA_2\cA_3\cA_7$\\ \hline $2$ & $5$ & $7$ &
$\frac{938350000}{472729139}\cA_2\cA_5\cA_7$\\ \hline $3$ & $3$ & $7$ &
$\frac{570192}{344605}\cA_3^2\cA_7$\\ \hline $10$ & $10$ & $10$ &
$\frac{12890}{6859}\cA_{10}^3$\\ \hline $2$ & $3$ & $10$ &
$\frac{54138}{34295}\cA_2\cA_3\cA_{10}$\\ \hline $2$ & $5$ & $10$ &
$\frac{14000}{6859}\cA_2\cA_5\cA_{10}$\\ \hline $5$ & $5$ & $10$ &
$\frac{1000}{361}\cA_5^2\cA_7$\\ \hline $11$ & $11$ & $11$ &
$\frac{2849748}{1295029}\cA_{11}^3$\\ \hline
\end{tabular}
\caption{Values of part of $\fS(\va,\Psi)$ for $\Psi(n_1,n_2)=(n_1,n_1+n_2,n_1+2n_2)$.} \label{tab:examples1}
\end{table}

\subsection{Three-term progressions whose common difference is one less than a
  prime, all with prescribed primitive roots.}\label{sec:example_intro}

Now let us investigate the leading constant $\fS(\va,\Psi)$ in Theorem
\ref{thm:main_artin} for the system
\[\Psi(n_1,n_2)=(n_1,n_2,n_1+n_2-1,n_1+2n_2-2)\]
with various values of $\va$ with $h_{a_i}=1$ for all $i$. In this case, for $p\nmid 6\cD_\va$, we have
\[
\sigma_{\va,\Psi}(p)=p^2\sum_{n\in(\ZZ/p\ZZ)^2}\prod_{i=1}^4\frac{\cA_{a_i}(\psi_i(n),p)}{\cA_{a_i}}.
\]
Similarly, as in the previous example, we have
\[
\cA_{a_i}(\psi_i(n),p)=\frac{\cA_{a_i}}{p-1}\left(1-\frac{1}{p(p-1)}\right)^{-1}\cdot\begin{cases}\left(1-\frac{1}{p}\right)&\text{if }p|\psi_i(n)-1,\\1 &\text{if }p\nmid \psi_i(n)(\psi_i(n)-1),\\0&\text{otherwise}.\end{cases}
\]
It is not hard to see that
\begin{equation}\label{eq:ex3_p_nonzero}
  p\nmid n_1n_2(n_1+n_2-1)(n_1+2n_2-2),
\end{equation}
holds for 
$(p-2)(p-3)+p-1$ values of $n$. As before, we need to discuss the congruences
\begin{align}
n_1\equiv 1\bmod p, \label{al_ex13:1}\\ 
n_2\equiv 1\bmod p,\label{al_ex13:2}\\
n_1+n_2\equiv 2\bmod p,\label{al_ex13:3}\\
n_1+2n_2\equiv 3\bmod p.\label{al_ex13:4}
\end{align}
Assuming \eqref{eq:ex3_p_nonzero}
we obtain:
\begin{enumerate}
\item All congruences are satisfied if and only if $n=(1,1)$.
\item There are no $(n_1,n_2)$ for which exactly two or three of the congruences are satisfied.
\item For exactly $4(p-3)$ values of $n$ satisfying \eqref{eq:ex3_p_nonzero}, exactly one of
  the congruences is satisfied. More concretely, we obtain $p-3$ cases if we
  suppose (\ref{al_ex13:1}), $p-2$ cases for (\ref{al_ex13:2}), $p-3$ cases for
  (\ref{al_ex13:3}), and $p-4$ cases for (\ref{al_ex13:4}).
\item For the remaining $(p-2)(p-3)+p-1-1-4(p-3)=(p-4)^2$ values of $n$, none of the
  congruences are satisfied.
\end{enumerate} 
Summing these cases up, we obtain
\begin{align*}
  \sigma_{\va,\Psi}(p)&=\frac{p^2}{(p-1)^4}\left(1-\frac{1}{p(p-1)}\right)^{-4}\Bigg((p-4)^2+4(p-3)\left(1-\frac{1}{p}\right)+\left(1-\frac{1}{p}\right)^4\Bigg)\\
  &=1-\frac{p^6-11p^4-4p^3+p^2+4p+1}{(p^2-p-1)^4},
\end{align*}
and therefore
\begin{equation*}
  \prod_{p\nmid6\cD_{\va}}\sigma_{\va,\Psi}(p)=\prod_{p\nmid6\cD_{\va}}\left(1-\frac{p^6-11p^4-4p^3+p^2+4p+1}{(p^2-p-1)^4}\right).
\end{equation*}
Thus, it remains to compute
\begin{equation*}
  \cC_\va:=\prod_{i=1}^3\delta(a_i,0,1)\sigma_{\va,\Psi}(\cD_\va)\prod_{\substack{p\mid 6\\p\nmid\cD_{\va}}}\sigma_{\va,\Psi}(p).
\end{equation*}
For given $\va=(a_1,a_2,a_3,a_4)$, this is a finite problem. Table
\ref{tab:example2a} shows the results of a computer program for some values of
$a_1=a_2=a_3=a_4=a$, and Table \ref{tab:example2b} does so for some 
choices with distinct $a_i$.

\begin{table}[ht]
\centering
\renewcommand{\arraystretch}{1.5}
\begin{tabular}{|c|c|}
\hline
$a$ & $\mathcal{C}_{\va}$\\
\hline
$2$ & $0$\\
\hline
$3$ & $0$\\
\hline 
$5$ & $0$\\
\hline
$6$ & $0$\\
\hline
$7$ & $\frac{914838624}{353220125}\cA_{7}^4$\\
\hline
$10$ & $\frac{315000}{130321}\cA_{10}^4$\\
\hline
$11$ & $\frac{12662473824}{3528954025}\cA_{11}^4$\\
\hline
$12$ & $0$\\
\hline
$13$ & $\frac{233540326656}{72150078125}\cA_{13}^4$\\
\hline
$14$ & $\frac{1008016632}{353220125}\cA_{14}^4$\\
\hline
\end{tabular}
\hspace{0.5cm}
\begin{tabular}{|c|c|}
\hline
$a$ & $\mathcal{C}_{\va}$\\
\hline
$15$ & $0$\\
\hline
$17$ & $\frac{1557162851328}{674197560125}\cA_{17}^4$\\
\hline 
$18$ & $0$\\
\hline
$19$ & $\frac{6434450287776}{1690158870125}\cA_{19}^4$\\
\hline
$20$ & $0$\\
\hline
$21$ & $0$\\
\hline
$22$ & $\frac{62825350896}{17644770125}\cA_{22}^4$\\
\hline
$23$ & $\frac{30998869628832}{8129718828125}\cA_{23}^4$\\
\hline
$24$ & $0$\\
\hline
$26$ & $\frac{263481394176}{72150078125}\cA_{26}^4$\\
\hline
\end{tabular}
\caption{Value of $\mathcal{C}_{\va}$ for several values of $\va=(a,a,a,a)$} \label{tab:example2a}
\end{table}

\begin{table}[ht]
\centering
\renewcommand{\arraystretch}{1.5}
\begin{tabular}{|c|c|c|c|c|}
\hline
$a_1$ & $a_2$ & $a_3$ & $a_4$ & $\mathcal{C}_{\va}$\\
\hline
$2$ & $2$ & $2$ & $3$ & $\frac{72576}{78125}\cA_2^3\cA_3$\\
\hline
$2$ & $2$ & $2$ & $5$ & $\frac{252000}{130321}\cA_2^3\cA_5$\\
\hline
$2$ & $2$ & $3$ & $3$ & $\frac{63504}{15625}\cA_2^2\cA_3^2$\\
\hline
$2$ & $2$ & $3$ & $5$ & $\frac{508032}{130321}\cA_2^2\cA_3\cA_5$\\
\hline
$2$ & $2$ & $3$ & $6$ & $0$\\
\hline
$2$ & $2$ & $5$ & $5$ & $\frac{604800}{130321}\cA_2^2\cA_5^2$\\
\hline
$2$ & $2$ & $5$ & $6$ & $\frac{2159136}{651605}\cA_2^2\cA_5\cA_6$\\
\hline
$2$ & $2$ & $6$ & $6$ & $\frac{72576}{78125}\cA_2^2\cA_6^2$\\
\hline
$2$ & $3$ & $3$ & $3$ & $0$\\
\hline
$2$ &$3$ & $3$ & $5$ & $\frac{508032}{130321}\cA_2\cA_3^2\cA_5$\\
\hline
\end{tabular}
\hspace{0.5cm}
\begin{tabular}{|c|c|c|c|c|}
\hline
$a_1$ & $a_2$ & $a_3$ & $a_4$ & $\mathcal{C}_{\va}$\\
\hline
$2$ & $3$ & $3$ & $6$ & $0$\\
\hline
$2$ & $3$ & $6$ & $6$ & $\frac{63504}{15625}\cA_2\cA_3\cA_6^2$\\
\hline
$2$ & $5$ & $5$ & $5$ & $\frac{403200}{130321}\cA_2\cA_5^3$\\
\hline
$2$ & $5$ & $5$ & $6$ & $\frac{2286144}{651605}\cA_2\cA_5^2\cA_6$\\
\hline
$2$ & $6$ & $6$ & $6$ & $\frac{244944}{78125}\cA_2\cA_6^3$\\
\hline
$3$ & $3$ & $3$ & $5$ & $0$\\
\hline
$3$ & $3$ & $3$ & $6$ & $0$\\
\hline
$3$ & $3$ & $5$ & $5$ & $\frac{3048192}{651605}\cA_3^2\cA_5^2$\\
\hline
$3$ & $3$ & $5$ & $6$ & $\frac{2159136}{651605}\cA_3^2\cA_5\cA_6$\\
\hline
$3$ & $3$ & $6$ & $6$ & $\frac{63504}{15625}\cA_{3}^2\cA_{6}^2$\\
\hline
\end{tabular}
\caption{Value of $\mathcal{C}_{\va}$ for several values of $\va=(a_1,a_2,a_3,a_4)$} \label{tab:example2b}
\end{table}

\markleft{\textsc{Christopher Frei, Joachim K\"onig and Magdal\'ena Tinkov\'a}}
\newpage
\markright{\textsc{Appendix A: Locally cyclic $\Sn_n$-extensions}}

\appendix
\section{Locally cyclic \texorpdfstring{$\Sn_n$}{Sn}-extensions.}\label{sec:appendix}

\begin{center}
{\sc by Christopher Frei, Joachim K\"onig and Magdal\'ena Tinkov\'a}
\end{center}

In this appendix, we apply Corollary \ref{cor:existence_cheb} of the main part
of the paper to prove the following result on locally cyclic normal $\Sn_n$-extensions of $\QQ$ with prescribed Artin symbols at finitely many primes. A Galois number field is \emph{locally cyclic}, if all decomposition groups are cyclic.   

\begin{theorem}\label{thm:locally_cyclic_new}
  Let $n\geq 2$ be an integer. Let $M_1$ be a finite set of primes and $M_2$ a finite set of
  sufficiently large primes in terms of $n$. For each $p\in M_2$, let
  $C_p$ be a conjugacy class of $\Sn_n$. Then there are infinitely many
  linearly disjoint Galois-extensions $K/\QQ$, such that:
  \begin{enumerate}
  \item $\Gal(K/\QQ)\simeq \Sn_n$,
  \item $K/\QQ$ is locally cyclic,
  \item All $p\in M_1\cup M_2$ are unramified in $K$,
  \item For each $p\in M_2$, we have $[K/\QQ,p]=C_p$.
  \end{enumerate}
\end{theorem}

This result answers a question concerning the existence of locally cyclic $\Sn_n$-extensions posed in \cite[p.479]{MR3556829} and 
generalises the case $n=5$ proved in \cite[Theorem 5.5]{MR4177544}. The proof follows the same strategy, except that we need to
replace an application of the theorem of Green and Tao in the case $n=5$ by
Corollary \ref{cor:existence_cheb}, as for larger $n$ it becomes necessary to
deal with splitting conditions in nonabelian fields. This yields the following
new ingredient in the proof of Theorem \ref{thm:locally_cyclic_new}, which in
 case $R=1$ gives a positive answer to the question of Kim and K\"onig
mentioned in the introduction of the main part. For $k\in\NN$, we
denote the splitting field of the polynomial
\begin{equation}\label{eq:def_f_k}
  f_k := x^{k-1}+2x^{k-2}+3x^{k-3}+\cdots+(k-1)x+k
\end{equation}
over $\QQ$ by $F_k$. As in the main part of the paper, $\Phi_{F_{k}^\ab}$
denotes the finite part of the conductor of the maximal abelian subsetension
$F_{k}^\ab/\QQ$ of $F_{k}/\QQ$.

\begin{proposition}\label{prop:koenig_problem_plus}
  Let $n\geq 2$ and define $\cD:=\lcm(n,n-1,\Phi_{F_{n-1}^\ab})$.
  Let $R\in\NN$ with $\gcd(R,\cD)=1$ and $l,k\in\NN$ with $\gcd(lk((n-1)^{n-1}k-n^nl),R)=1$. Then there are infinitely
  many pairs $(s,t)\in\ZZ^2$ with $|s|\neq |t|$ and $(s,t)\equiv(l,k)\bmod R$, such that
 \begin{align}
    &|t|& &\text{ is a prime congruent to } 1\bmod n,\label{eq:app_cond_1}\\
    &|s|& &\text{ is a prime congruent to } 1\bmod n-1,\label{eq:app_cond_2}\\
    &|(n-1)^{n-1}t-n^ns|& &\text{ is a prime splitting completely in } F_{n-1}.\label{eq:app_cond_3}
\end{align}
\end{proposition}

We first prove Proposition \ref{prop:koenig_problem_plus} using Corollary
\ref{cor:existence_cheb}, and then deduce Theorem
\ref{thm:locally_cyclic_new} from it.

\subsection{Proof of Proposition \ref{prop:koenig_problem_plus}}
The following lemma will help with the verification of condition \emph{(4)} in Corollary \ref{cor:existence_cheb}.  

\begin{lemma}\label{lem:app_local}
  There  are $(s,t)\in(\ZZ/\cD\ZZ)^2$ such that
  $\gcd(st,\cD)=1$ and
  \begin{equation*}
    (-1)^{n-1}t\equiv 1\bmod n,\quad -s\equiv 1\bmod n-1,\quad\text{ and }\quad (n-1)^{n-1}t-n^ns\equiv 1\bmod \cD. 
  \end{equation*} 
\end{lemma}

\begin{proof}
  As $\gcd(n,n-1)=1$, we may write $\cD=q_1q_2q_3$ such that
  \begin{align*}
    q_1 &\text{ is the product of all prime powers $p^e\mid\mid \cD$ with $p\mid n$},\\
    q_2 &\text{ is the product of all prime powers $p^e\mid\mid \cD$ with $p\mid n-1$,}\\
    q_3 &\text{ is the product of the remaining prime powers in $\cD$.}
  \end{align*}
  Then $q_1,q_2,q_3$ are pairwise coprime, $n$ divides $q_1$ and $n-1$ divides $q_2$. In particular, $n$ is invertible modulo $q_2q_3$ and $n-1$ is invertible modulo $q_1q_3$. As either $n$ or $n-1$ is even, we moreover see that $2\nmid q_3$. By the Chinese
  remainder theorem, we find $(s,t)\in(\ZZ/\cD\ZZ)^2$ that satisfy
  \begin{align*}
    t&\equiv (n-1)^{-(n-1)}(1+n^n)& &\bmod q_1, & s&\equiv 1 & &\bmod q_1,\\
    t&\equiv 1 & &\bmod q_2,      & s&\equiv n^{-n}((n-1)^{n-1}-1)& &\bmod q_2,\\
    t&\equiv (n-1)^{-(n-1)}\cdot 2 & &\bmod q_3,      & s&\equiv n^{-n}& &\bmod q_3.
  \end{align*}
As $n\mid q_1$ and they have the same prime factors, this implies that $\gcd(st,q_1)=1$ and $t\equiv (-1)^{n-1}\bmod n$. Similarly, as $(n-1)\mid q_2$ and they have the same prime factors, we see that $\gcd(st,q_2)=1$ and $s\equiv -1\bmod n-1$. Moreover, we have $\gcd(st,q_3)=1$, as $2\nmid q_3$. Finally, for all $i\in\{1,2,3\}$, we have $(n-1)^{n-1}t-n^ns\equiv 1\bmod q_i$, and thus the same congruence holds modulo $\cD$.  
\end{proof}  

\begin{proof}[Proof of Proposition \ref{prop:koenig_problem_plus}]
  By the Chinese remainder theorem, we find $u,v\in\NN$ with $\gcd(uv,\cD R)=1$,
  \begin{align*}
    u&\equiv 1\bmod n, &  v&\equiv 1\bmod (n-1),\\
    u&\equiv (-1)^{n-1}k\bmod R, &  v&\equiv -l\bmod R.
  \end{align*}
  We verify the hypotheses of Corollary \ref{cor:existence_cheb} with parameters
  $s=2, t=3$, $K_1=\QQ(\mu_{Rn})$, $K_2=\QQ(\mu_{R(n-1)})$, $K_3=F_{n-1}$,
  $C_1=\{\zeta_{Rn}\mapsto\zeta_{Rn}^u\}$, $C_2=\{\zeta_{R(n-1)}\mapsto\zeta_{R(n-1)}^v\}$, $C_3=\{\ident\}$, the system of forms
\begin{equation*}
  \Psi(s,t)=((-1)^{n-1}t,-s,(n-1)^{n-1}t-n^ns),
\end{equation*}
and the open convex cone
\begin{equation*}
  X=\left\{(s,t)\in\RR^2\ : \ (-1)^{n-1}t>0\ \text{ and }\ -n^n s > \max\{-(n-1)^{n-1}t,0\}\right\}.
\end{equation*}
Hypothesis \emph{(1)} is obvious and \emph{(2)} is easy to check. For \emph{(3)}, we note that for the point  $x=(-1,(-1)^{n-1})\in X$, we have $\dot\psi_i(x)=\psi_i(x)>0$ for all $i$.

Hence, it remains to deal with \emph{(4)}. We have
$\cD_{\vK}\mid\cD R$, hence we may take $q=\cD R$ in \emph{(4)}.  Let
$s_0,t_0$ be the congruence classes modulo $\cD$ resulting from
Lemma \ref{lem:app_local}. Using the Chinese remainder theorem, we
find $(s,t)\in\ZZ^2$ with $(s,t)\equiv (s_0,t_0)\bmod \cD$ and
$(s,t)\equiv (l,k)\bmod R$. These then satisfy
$\gcd(\psi_1(s,t)\psi_2(s,t)\psi_3(s,t),q)=1$ and
\begin{align*}
  \psi_1(s,t)&=(-1)^{n-1}t & &\equiv u \bmod Rn,\\
  \psi_2(s,t)&= -s& &\equiv v \bmod R(n-1),\\
  \psi_3(s,t)&= (n-1)^{n-1}t-n^ns& &\equiv 1 \bmod \cD.
\end{align*}
Then clearly $\sigma_{\psi_1(s,t)}\in\Gal(\QQ(\mu_q)/\QQ)$ restricts
to $C_1\subseteq \Gal(K_1/\QQ)$ and
$\sigma_{\psi_2(s,t)}\in\Gal(\QQ(\mu_q)/\QQ)$ restricts to
$C_2\subseteq \Gal(K_2/\QQ)$. Moreover,
$\sigma_{\psi_3(s,t)}\in\Gal(\QQ(\mu_q)/\QQ)$ restricts to the
identity on $\QQ(\mu_{\cD})$, and hence also on $K_3^\ab\subseteq\QQ(\mu_\cD)$.

We have verified all hypotheses of Corollary \ref{cor:existence_cheb}. Hence, there are infinitely many pairs $(s,t)\in X\cap\ZZ^2$ for which $\psi_1(s,t),\psi_2(s,t),\psi_3(s,t)$ are pairwise distinct primes with the prescribed Artin symbols. Now the condition $[K_1/\QQ,\psi_1(s,t)]=C_1$ is equivalent to $\psi_1(s,t)\equiv u\bmod Rn$.  
For $(s,t)\in X$, we have $\psi_1(s,t)=(-1)^{n-1}t=|t|$, so these conditions translate to $|t|\equiv 1\bmod n$ and $t\equiv k\bmod R$. Similarly, the condition $[K_1/\QQ,\psi_2(s,t)]=C_2$ translates to $|s|\equiv 1\bmod n-1$ and $s\equiv l\bmod R$, and the condition $[K_3/\QQ,\psi_3(s,t)]=C_3$ means exactly that $|(n-1)^{n-1}t-n^ns|=(n-1)^{n-1}t-n^ns$ splits completely in $F_{n-1}$. 
\end{proof}

\subsection{Proof of Theorem \ref{thm:locally_cyclic_new}}
Apart from the application of the new Proposition \ref{prop:koenig_problem_plus}, this is
essentially the same as the proof of \cite[Theorem 5.5]{MR4177544}, so we will
be brief and focus on what is different.

We consider the splitting field $E/\QQ(t)$ of the polynomial
$f(t,X)=X^n-t(X-1)$, which is ramified over the places
$t\mapsto 0, t\mapsto\infty$ and $t\mapsto n^n/(n-1)^{n-1}$ with homogenised
irreducible polynomials $T,S$ and $(n-1)^{n-1}T-n^n S$. As explained on
\cite[p.283]{MR4177544}, the residue fields at $t\mapsto 0$ and $t\mapsto\infty$
are $\QQ(\mu_n)$ and $\QQ(\mu_{n-1})$, and the residue field at $t\mapsto
n^n/(n-1)^{n-1}$ is the splitting field of the polynomial
\begin{equation*}
f\left(\frac{n^n}{(n-1)^{n-1}},X\right) = X^n-\frac{n^n}{(n-1)^{n-1}}X+\frac{n^n}{(n-1)^{n-1}}.
\end{equation*}
As
\begin{equation*}
\left(\frac{n-1}{n}\right)^nf\left(\frac{n^n}{(n-1)^{n-1}},\frac{n}{n-1}X\right) = (X-1)^2f_{n-1},
\end{equation*}
the latter splitting field is $F_{n-1}$.

Proposition \ref{prop:koenig_problem_plus}, with a choice of $R,l,k$ to be
specified later, produces infinitely many coprime pairs $(s_0,t_0)\in\ZZ^2$ satisfying
\eqref{eq:app_cond_1}--\eqref{eq:app_cond_3}. According to \cite[Theorem 3.1
and Theorem 3.2]{MR4177544}, for large enough $|s_0|,|t_0|$, the only primes
 ramified in the specialisation $E_{\frac{t_0}{s_0}}/\QQ$ are the primes $|s_0|,|t_0|$ and
$|(n-1)^{n-1}t_0-n^ns_0|$, and moreover the decomposition groups at these primes are
cyclic, and more precisely generated (in this order) by an $(n-1)$-cycle, an $n$-cycle and a transposition respectively.

This gives infinitely many extensions $E_{\frac{t_0}{s_0}}/\QQ$ that satisfy
conditions \emph{(2)} and \emph{(3)} of Theorem \ref{thm:locally_cyclic_new}. The
proof of \cite[Theorem 5.5]{MR4177544} provides instructions on how to choose
$R,l,k$ in order to ensure that condition \emph{(4)} holds as well, if only the
primes in $M_2$ are sufficiently large.

Moreover, as explained at the start of the proof of \cite[Theorem
5.5]{MR4177544}, we get condition \emph{(1)} for free by enlarging the set
$M_2$ in order to ensure that all conjugacy classes of $\Sn_n$ occur amongst the $C_p$.

\begin{remark}
\label{rem:2sn}
It is furthermore natural to use extensions such as the ones obtained in the proof of Theorem \ref{thm:locally_cyclic_new} to try to solve certain embedding problems and in particular to construct locally cyclic Galois extensions whose Galois group $G$ is a central extension of $\Sn_n$. This was done in \cite[Theorem 5.6]{MR4177544} for the case $n=5$ and for any  $G$ which is a  central extension of $\Sn_5$ in which the transpositions lift to elements of order $2$. The following observation shows that this approach does at least not generalize to arbitrary $n$ using the $\Sn_n$ extensions obtained here. Denote by $2.\Sn_n^+$ the one of the two stem covers of $\Sn_n$ in which the transpositions lift to elements of order $2$. It is then well-known (see, e.g., \cite{HoffmanHumphreys})  that the involutions of $\Sn_n$ which lift to involutions in $2.\Sn_n^+$ are exactly the products of $4j$ or $4j-3$ disjoint transpositions ($j\in \mathbb{N}$). On the other hand, the extensions constructed in Theorem \ref{thm:locally_cyclic_new} arise via trinomials $X^n-\frac{t_0}{s_0}(X-1)$, and those have at most three real roots, meaning that complex conjugation acts as an involution with at most three fixed points. Whenever $n\equiv 6,7 \pmod{8}$, such an involution hence lifts to elements of order $4$ in $2.\Sn_n^+$, meaning that the $\Sn_n$ extensions thus constructed do not embed into $2.\Sn_n^+$-extensions due to an obstruction at the archimedean prime. 
\end{remark}      

\markleft{}
\markright{}

\bibliographystyle{plain}
\bibliography{bibliography}
\end{document}